\RequirePackage{luatex85} % see
\documentclass[11pt]{amsart}
\usepackage{amssymb}
\usepackage{latexsym}
\usepackage{graphicx}
\usepackage{subfigure}
\usepackage{tikz}
\usepackage[linktocpage=true]{hyperref}
\usepackage[left=3.8cm, right=3.8cm, top=3cm, bottom=3cm]{geometry}
\usepackage{color}
\usepackage{here}
\usepackage[backgroundcolor = lightgray,
            linecolor = lightgray,
            textsize = tiny]{todonotes}
\usepackage{stackengine}
\usepackage{mathrsfs}
\usepackage[all]{xy}
\usepackage{scrextend}

\hypersetup{ colorlinks   = true, %Colours links instead of ugly boxes
urlcolor  = blue, %Colour for external hyperlinks
linkcolor    = blue, %Colour of internal links
citecolor   = blue %Colour of citations
}

\def\beq{\begin{equation}}
\def\eeq{\end{equation}}

\def\det{\mathrm{det}\ }

\def\ac{\mathrm{Acc}}
\def\cac{\mathrm{C}}

\newcommand{\cC}{{\mathcal C}}
\newcommand{\cD}{{\mathcal D}}

\newcommand{\cW}{{\mathcal W}}
\newcommand{\cWcu}{{\mathcal W}^{cu}}
\newcommand{\cWcs}{{\mathcal W}^{cs}}
\newcommand{\cWc}{{\mathcal W}^{c}}
\newcommand{\cWu}{{\mathcal W}^{u}}
\newcommand{\cWs}{{\mathcal W}^{s}}
\newcommand{\cPH}{{\mathcal{PH}}}

\newcommand{\cU}{{\mathcal U}}

\newcommand{\Z}{{\mathbb Z}}
\newcommand{\R}{{\mathbb R}}

\newcommand{\N}{{\mathbb N}}

\newtheorem{theorem}{Theorem}[section]
\newtheorem{remark}[theorem]{Remark}

\newtheorem{lemma}[theorem]{Lemma}
\newtheorem{conj}[theorem]{Conjecture}
\newtheorem{defi}[theorem]{Definition}
\newtheorem{prop}[theorem]{Proposition}
\newtheorem{corollary}[theorem]{Corollary}

\newtheorem*{main thm}{Main Theorem}
\newtheorem*{main thm bis}{Main Theorem (alternate version)}

\newtheorem{theoalph}{Theorem}

\begin{document}
\numberwithin{equation}{section}

\title[Accessibility for P.H. Diffeomorphisms with $2$-dimensional center]{Accessibility for dynamically coherent partially hyperbolic diffeomorphisms with  $\mathbf{2D}$  center}

\author[Martin Leguil]{Martin Leguil$^{*}$}
\thanks{$^{*}$M.L. was supported by the ERC project 692925 NUHGD of Sylvain Crovisier and the NSERC Discovery grant of Jacopo De Simoi, reference number 502617-2017.} 
\author[Luis Pedro Pi\~neyr\'ua]{Luis Pedro Pi\~neyr\'ua$^{**}$}
\thanks{$^{**}$L.P.P. was supported by CAP's doctoral scholarship and CSIC group 618}
%\address{}
%\email{}

\maketitle

\begin{abstract}
	We show that for any integer $r \geq 2$, stable accessibility   is $C^r$-dense among partially hyperbolic diffeomorphisms with two-dimensional center that satisfy some strong bunching and are stably dynamically coherent. %, and plaque expansive.
\end{abstract}

\tableofcontents

\section{Introduction}

In 1871 L. Boltzmann stated his ergodic hypothesis when he was studying the motion of gases and thermodynamics. He wanted a property that could let him ``characterize the probability of a state by the average time in which the system \emph{is} in this state". Since then, ergodicity has played a key role in dynamical systems, physics and probability. Recall that a dynamical system $f\colon M\to M$ preserving a finite measure $m$ is \emph{ergodic} if every $f$-invariant set has zero or total measure. 

After Birkhoff's ergodic theorem,  E. Hopf proved in 1939 the ergodicity of the geodesic flow on a surface of constant negative curvature, introducing an argument to get ergodicity which is now called Hopf's argument \cite{H}. Twenty eight years later, D. Anosov \cite{A} improved Hopf's results by proving the ergodicity of the geodesic flow on surfaces of negative (non necessarily constant) curvature and compact manifolds of constant negative curvature. He also showed the ergodicity of uniformly hyperbolic diffeomorphisms, now called Anosov diffeomorphisms. Since hyperbolicity is a $C^{1}$-robust condition, Anosov diffeomorphisms became the first example of \emph{stably ergodic} diffeomorphisms, that is, a $C^{r}$ ergodic diffeomorphism (preserving a measure $m$) that remains ergodic after a $C^{1}$-small perturbation. 

For almost thirty years Anosov diffeomorphisms were the only known examples of stably ergodic systems, until 1995 when M. Grayson, C. Pugh and M. Shub \cite{GPS} proved the $C^{2}$ stable ergodicity of the time-one map of the geodesic flow on surfaces of constant negative curvature, hence the first non-Anosov stably ergodic example. Despite being non globally hyperbolic, this example has a weak form of hyperbolicity called \emph{partial hyperbolicity}. With the evidence of this work they formulated in a 1995 conference \cite{PS} the following conjecture:

\begin{conj}[Pugh-Shub's stable ergodicity conjecture \cite{PS,PS1}] \label{psconj}
On any compact connected   Riemannian manifold, stable ergodicity is $C^{r}$-dense among the set of volume preserving partially hyperbolic diffeomorphisms, for any integer $r\geq 2$.
\end{conj}

They also proposed a program in order to prove this, and split the conjecture into two conjectures:

\begin{conj}[Accessibility implies ergodicity] \label{conjerg}
A $C^{2}$ partially hyperbolic  volume preserving diffeomorphism with the essential accessibility property is ergodic.
\end{conj}
Here, \emph{essential accessibility} is a measure-theoretic version of the accessibility property. 

\begin{conj}[Density of accessibility] \label{conjacc}
For any integer $r\in [2,+\infty]$, stable accessibility is open and dense among the set of $C^{r}$ partially hyperbolic diffeomorphisms, volume preserving or not.
\end{conj}

There has been a lot of progress on these conjectures, mostly depending on the topology and the dimension of the central bundle. 

The main conjecture was proven in \cite{RHRHU} in the case where $\dim E^{c}=1$ and for  the $C^{r}$ topology (in fact the authors showed $C^{\infty}$-density). Recently in \cite{ACW} the conjecture was proved in its full generality (any central dimension) for the $C^{1}$ topology. Despite these remarkable results, in the $C^{r}$ case for $r\geq 2$ the conjecture is far from being solved. Recently, M. Leguil and Z. Zhang \cite{LZ} obtained $C^{r}$-density of stable ergodicity for partially hyperbolic diffeomorphisms (for any center dimension) with a strong pinching condition, introducing a new technique based on  random perturbations.  

With respect to Conjecture \ref{conjerg}, C. Pugh and M. Shub   \cite{PS2} proved that a $C^{2}$ volume preserving partially hyperbolic  diffeomorphism that is dynamically coherent, center bunched and with the essential accessibility property is ergodic. The state-of-the-art on Conjecture \ref{conjerg} is the result of K. Burns and A. Wilkinson  \cite{BW1} where the authors improved Pugh-Shub's result by removing the dynamical coherence hypothesis, and weakening the center bunching condition. In other words, by these works, a possible strategy to show that stable ergodicity is typical in the $C^r$ topology would be to go further towards Conjecture \ref{conjacc}, i.e., that stable accessibility is $C^r$-dense.

Regarding Conjecture \ref{conjacc}, in \cite{DW, ACW2} stable accessibility is obtained for a $C^{1}$-dense set of $\bullet$ all $\bullet$  volume preserving $\bullet$ symplectic partially hyperbolic diffeomorphisms. The authors strongly use $C^{1}$ techniques which seem hard to generalize to other topologies.  

For the $\dim E^{c}=2$ case, there has been many results lately. The first one is the remarkable result by F. Rodr\'iguez-Hertz \cite{RH} where he classified the central accessibility classes and obtained  stable ergodicity of certain automorphisms on the torus $\mathbb{T}^{d}:=\R^d/\Z^d$. Elaborating on these ideas, in \cite{HS} V. Horita and M. Sambarino proved stable ergodicity for skew-products of surface diffeomorphisms over Anosov diffeomorphisms. Recently, A. \'Avila and M. Viana \cite{AV} obtained $C^{1}$-openness of accessibility and $C^{r}$-density for certain %skew-products of surfaces diffeomorphisms over Anosov diffeomorphisms 
\emph{fibered} partially hyperbolic diffeomorphisms with $2$-dimensional center bundle
using different techniques. 

Our purpose in this article is to contribute to the accessibility conjecture (Conjecture \ref{conjacc}) by proving the $C^{r}$-density of accessibility for (stably) dynamically coherent partially hyperbolic diffeomorphism with $2$-dimensional center bundle which satisfy some strong bunching condition, for any integer $r\geq 2$.

%\vspace{.5cm}
%{\noindent \it Acknowledgements.} 
\section*{Acknowledgements}
M.L. wishes to thank the hospitality of Universidad de la Rep\'ublica for a stay during which this work was initiated, and of Université Paris-Saclay while he was a postdoc there, as well as the support of the ERC project 692925 NUHGD of Sylvain Crovisier. The authors would like to thank Sylvain Crovisier, Rafael Potrie and Mart\'in Sambarino (director of the second author) for many discussions about this work, as well as the support of the CSIC group 618. We also wish to thank Carlos Matheus, Davi Obata and Zhiyuan Zhang for several conversations. 

\section{Preliminaries}\label{sec prelim}

\subsection{Partially hyperbolic diffeomorphisms}

Let us fix a compact Riemannian manifold  $M$  of dimension $m \geq 3$. We denote by $\mathrm{Vol}$ the volume form, and we denote by $\|\cdot\|$ the norm on $TM$ associated to the Riemannian metric. We say that a diffeomorphism $f$ of $M$ is  \emph{partially hyperbolic}  if there exist  a nontrivial $Df$-invariant splitting $TM=E_f^s \oplus E_f^c \oplus E_f^u$ of the tangent bundle and continuous positive functions $\nu,\hat \nu,\gamma,\hat \gamma$ with 
\begin{equation}\label{functions nu gamma}
\nu,\hat \nu<1,\qquad \nu< \gamma<\hat \gamma^{-1}< \hat \nu^{-1},
\end{equation}
such that for any $(x,v)\in TM$, it holds
\begin{align*}
	&\|D_x f(v)\|<\nu(x) \|v\|, & \text{if }v \in E_f^s(x)\setminus \{0\},\\
	\gamma(x)  \|v\| <\ & \|D_x f(v)\|< \hat \gamma(x)^{-1} \|v\|, &\text{if } v \in E_f^c(x)\setminus \{0\},\\
	\hat \nu(x)^{-1} \|v\| <\ & \|D_x f(v)\|, & \text{if } v \in E_f^u(x)\setminus \{0\}. 
\end{align*}
For any integer $r \geq 1$, we denote by $\cPH^r(M)$ the set of all partially hyperbolic diffeomorphisms of $M$ of  class $C^r$; we also denote by  $\cPH^r(M,\mathrm{Vol})\subset \cPH^r(M)$ the subset of volume preserving partially hyperbolic diffeomorphisms.

In the rest of this section, we fix an integer $r \geq 1$ and we consider a partially hyperbolic diffeomorphism $f\in\cPH^r(M)$. We will denote $d_s:=\dim E_f^s$ and $d_u:=\dim E_f^u$. 
The strong bundles $E_f^u$ and $E_f^s$ are uniquely integrable to continuous foliations $\cWu_f$ and $\cWs_f$ respectively, called the \emph{strong unstable} and \emph{strong stable} foliations.  For $* = u,s$, and  for any $x \in M$, we denote by $\cW_f^*(x)$ the leaf of $\cW_f^*$ through $x$. The foliation $\cW_f^*$ is invariant under the dynamics, i.e.,  $f(\cW_f^*(x))=\cW_f^* (f(x))$, for all $x \in M$.  Moreover,  each leaf $\cW_f^*(x)$, $x \in M$, is an immersed $C^{r}$ manifold.  

\subsection{Dynamical coherence, plaque expansiveness}
  
The partially hyperbolic diffeomorphism $f$ is \emph{dynamically coherent} if the \emph{center-unstable} bundle $E_f^{cu}:=E_f^c \oplus E_f^u$ and the \emph{center-stable} bundle $E_f^{cs}:=E_f^c \oplus E_f^s$  integrate respectively to foliations $\cW^{cu}_f$, $\cW^{cs}_f$, called the \emph{center-unstable foliation}, resp. the \emph{center-stable foliation}, where $\cW^{u}_f$ subfoliates $\cWcu_f$, while $\cWs_f$ subfoliates $\cWcs_f$. %In this case, for any $x \in X$, we let . 
In this case, the collection $\cWc_f$ obtained by intersecting the leaves of $\cWcs_f$ and $\cWcu_f$ is a foliation which integrates $E_f^{c}$, and subfoliates both $\cWcs_f$ and $\cWcu_f$; it is called the \emph{center foliation}.  

In the following,  for any $*\in \{s,c,u,cs,cu\}$, we denote by $d_{\mathcal{W}_f^*}$ the leafwise distance, and for any $x \in M$, for any  $\epsilon>0$, we denote by $\cW_f^*(x,\varepsilon):=\{y \in \cW_f^*(x): d_{\cW_f^*}(x,y)< \varepsilon\}$ the $\varepsilon$-ball in $\cW_f^*$ of center $x$ and radius $\varepsilon$.

It is an open question whether dynamical coherence is a $C^1$-open condition. A closely related property is \emph{plaque expansiveness}.

\begin{defi}[Plaque expansiveness]\label{plqauexp}
	We say that $f$ is  \emph{plaque expansive} (see \cite[Section 7]{HPS}) if $f$ is dynamically coherent and there exists  $\varepsilon>0$ with the following property: if $(p_n)_{n \geq 0}$ and $(q_n)_{n \geq 0}$ are $\varepsilon$-pseudo orbits which respect $\cW_f^c$ such that $d(p_n,q_n)\leq \varepsilon$ for all $n \geq 0$, then $q_n\in \cW_f^c(p_n)$. It is known that plaque expansiveness is a $C^1$-open condition (see Theorem 7.4 in \cite{HPS}).
\end{defi}

The following result is due to Hirsch-Pugh-Shub. %\marginpar{I checked the following proposition}
\begin{theorem}[Theorem 7.1, \cite{HPS}, see also Theorem 1 in \cite{PSW2}]\label{thmplaqueexpansivetostablydc}
	Let us assume that $f$ is dynamically coherent and plaque expansive. Then any $g \in \cPH^1(M)$ which is sufficiently $C^1$-close to $f$ is also dynamically coherent and plaque expansive. Moreover, there exists a homeomorphism $\mathfrak{h}=\mathfrak{h}_g \colon M\to M$,
	%$C^0$-close to $\mathrm{Id}$, 
	called a \emph{leaf conjugacy}, such that  $\mathfrak{h}$ maps a $f$-center leaf to a $g$-center leaf, and $\mathfrak{h} \circ f(\mathcal{W}_f^c(\cdot))= g \circ \mathfrak{h}(\mathcal{W}_f^c(\cdot))$. %, for all $x \in M$. 
\end{theorem}

\subsection{Holonomies}

Let us assume that the diffeomorphism $f$ is dynamically coherent. 
Let $x_1\in M$ and let $x_2 \in M$ be sufficiently close to $x_1$.\footnote{In the rest of the paper, all the constructions will be local.} By transversality, there exist a neighbourhood  $\mathcal{U}_1$ of $x_1$ within $\cWcu_f(x_1)$  and a neighbourhood $\mathcal{U}_{2}$ of $x_2$ within $\cWcu_f(x_2)$ such that for any $z \in  \mathcal{U}_1$, the local stable leaf through $z$ intersects  $\mathcal{U}_2$ at a  unique point,  denoted by $H_{f,x_1,x_2}^{s}(z)\in \mathcal{U}_{x_2}^{cu}$. We thus get a well defined local homeomorphism
\begin{equation*}
	H_{f,x_1,x_2}^{s}\colon
	\mathcal{U}_1 \to \mathcal{U}_2 \subset \cWcu_f(x_2),
\end{equation*}
called the  \emph{stable holonomy map}. Note that as a consequence of dynamical coherence, if $x_2\in \mathcal{W}_{f,\mathrm{loc}}^s(x_1)$, then the image of the restriction $H_{f,x_1,x_2}^{s}|_{\mathcal{U}_1 \cap \cWc_f(x_1)}$ to the center leaf $\cWc_f(x_1)$ is contained in the center leaf $\cWc_f(x_2)$.  
Unstable holonomies are defined in a similar way, following local unstable leaves. 

\begin{defi}[Center bunching]
We say that $f$ is \emph{center bunched} if the functions    $\nu,\hat \nu,\gamma,\hat \gamma$ in \eqref{functions nu gamma} can be chosen such that 
$$
\max(\nu,\hat \nu)< \gamma \hat \gamma.
$$
\end{defi}

%\begin{defi}[$k$-normal hyperbolicity]
%	For any integer $k \geq 2$, we say that $f$ is $k$-\textit{normally hyperbolic} if the functions    $\nu,\hat \nu,\gamma,\hat \gamma$ in \eqref{functions nu gamma} can be chosen such that 
%	$$
%	\nu< \gamma^k,\quad \text{and}\quad \hat \nu< \hat \gamma^k.
%	$$
%	For any integer $r \geq 1$, we denote by $\cPH^r_k(M)$ the set of $k$-normally hyperbolic $C^r$ partially hyperbolic diffeomorphisms on $M$. 
%\end{defi}

\begin{theorem}[see \cite{HPS} and Theorem B in \cite{PSW1}]\label{cor smooth holonomy maps}
	If $f \in \cPH^{2}(M)$ is dynamically coherent and center bunched, then local stable/unstable holonomy maps between center leaves are $C^1$ when restricted to some center-stable/center-unstable leaf and have uniformly continuous derivatives. 
%	Moreover,  for any integer $k \geq 2$,  and for any dynamically coherent partially hyperbolic diffeomorphism  $f\in  \cPH_k^{k}(M)$, the leaves of $\cW^c_f$ are of class $C^k$. 
\end{theorem}  

Indeed, the authors prove that the strong stable/unstable foliation is $C^1$ when restricted to a center-stable/unstable leaf. However, from their proof, it is not clear how the holonomies $H_{f,x_1,x_2}^s|_{\cWc_{f,\mathrm{loc}}(x_1)}$, resp. $H_{f,x_1,x_2}^u|_{\cWc_{f,\mathrm{loc}}(x_1)}$ vary in the $C^1$-topology with the choices of the points $x_1$ and $x_2\in \mathcal{W}_{f,\mathrm{loc}}^s(x_1)$, resp. $x_2\in \mathcal{W}_{f,\mathrm{loc}}^u(x_1)$. This question is investigated in Obata's work \cite{O}, where it is shown that under some stronger bunching condition, these holonomy maps vary continuously with the choices of the base points $x_1,x_2$. 

\begin{defi}[see \cite{O}]\label{stron bunch}
	For any integer $r \geq 1$, we denote by $\mathcal{PH}_*^r(M)$ the set of all partially hyperbolic diffeomorphisms $f  \in \mathcal{PH}^r(M)$ such that, for some $\theta \in (0,1)$,  
	\begin{gather*}
	\|D_x f|_{E_f^s}\|^\theta < \frac{m(D_x f|_{E_f^c})}{\|D_x f|_{E_f^c}\|},\quad \frac{\|D_x f|_{E_f^c}\|}{m(D_x f|_{E_f^c})}<m(D_x f|_{E_f^u})^\theta,\\
	\|D_x f|_{E_f^s}\|<m(D_x f|_{E_f^c})m(D_x f|_{E_f^s})^\theta,\\ \|D_x f|_{E_f^c}\| \cdot\|D_x f|_{E_f^u}\|^\theta<m(D_x f|_{E_f^u}).
	\end{gather*}
\end{defi}
\noindent Note that any diffeomorphism $f\in \mathcal{PH}^r_*(M)$ is automatically center bunched. 

\begin{theorem}[Theorem 0.3 in \cite{O}]
Assume that $f  \in \mathcal{PH}^2_*(M)$. Then, for $*=s,u$, the family $\big\{H_{f,x_1,x_2}^*|_{\cWc_{f,\mathrm{loc}}(x_1)}\big\}_{x_1\in M,\, x_2\in \mathcal{W}_{f,\mathrm{loc}}^*(x_1)}$ is a family of $C^1$ maps depending continuously in the $C^1$-topology with the choices of the points $x_1$ and $x_2\in \mathcal{W}_{f,\mathrm{loc}}^*(x_1)$. 
\end{theorem}

\subsection{Accessibility classes}
 
A $f$-\emph{accessibility sequence} is a sequence $[x_1,\dots,x_k]$ of $k \geq 1$ points in $M$ such that for any $i\in \{1,\dots,k-1\}$, the points $x_i$ and $x_{i+1}$ belong to the same stable or unstable leaf of $f$. In particular, the points $x_1$ and $x_k$ can be connected by some $f$-\emph{path}, i.e., a continuous path in $M$ obtained by concatenating  finitely  many  arcs in $\cWs_f$ or $\cWu_f$. We will refer to the points $x_1,\dots,x_k$ as the \emph{corners} of the accessibility sequence $[x_1,\dots,x_k]$. 

For any point $x \in M$, we denote by $\ac_f(x)$ the \emph{accessibility class} of $x$. By definition, it is the set of all points $y\in M$ which can be connected to $x$ by some $f$-path. We also let 
$$
\cac_f(x):=\mathrm{cc}(\ac_f(x)\cap \cWc_f(x,1),x)
$$ 
be the connected component containing $x$ of the intersection of the accessibility class of $x$ and the local center leaf through $x$. Similarly, for any $\varepsilon>0$, we let $\cac_f(x,\varepsilon):=\mathrm{cc}(\ac_f(x)\cap \cWc_f(x,\varepsilon),x)$. 
By definition, accessibility classes form a partition of $M$. We say that the diffeomorphism $f$ is \emph{accessible} if this partition is trivial, i.e., the whole manifold $M$ is a single accessibility class; we say that $f$ is \emph{stably accessible} if the  diffeomorphisms which are sufficiently $C^1$-close to $f$ are accessible.

Moreover, given any %$f$-path $\gamma\colon [0,1]\to M$ associated to an 
$f$-accessibility sequence $\gamma=[x_1,\dots,x_k]$, we let $H_{f,\gamma}\colon  \cW_{f,\mathrm{loc}}^c(x_1)\to \cW_{f,\mathrm{loc}}^c(x_k)$ be the holonomy map  obtained by concatenating the local holonomy maps along the arcs of $\gamma$, i.e.,
\begin{equation}\label{compo holonomies}
H_{f,\gamma}:=H_{f,x_{k-1},x_k}^{*_{k-1}}\circ \cdots \circ H_{f,x_{1},x_2}^{*_{1}},
\end{equation}
where for $j \in \{1,\dots,k-1\}$, $*_j\in \{s,u\}$ is such that $x_{j+1}\in \cW_f^{*_j}(x_j)$. 

The next lemma is elementary; it follows from the local product structure and the continuous dependence of the invariant foliations with respect to the diffeomorphism.
\begin{lemma}[Continuation of accessibility sequences]\label{lemma continuation}
Let	$\gamma=[x_0,x_1,\dots,x_k]$ be a $f$-accessibility sequence, for some integer $k \geq 0$. Then there exist a neigbourhood $\mathcal{O}$ of $x_0$ and a $C^1$-neighbourhood $\mathcal{U}$ of $f$  such that for any point $x \in \mathcal{O}$, and for any diffeomorphism $g \in \mathcal{U}$, there exists a natural continuation $\gamma^{x,g}=[x,x_1^{x,g},\dots,x_k^{x,g}]$ of $\gamma$ for $x$ and $g$. Indeed, the $g$-accessibility sequence $\gamma^{x,g}$ is defined as 
	\begin{align*}
	x_1^{x,g} &:= H_{g,x,x_1}^{*_0}(x);\\% \cW_{g,\mathrm{loc}}^{*_0}(x)\cap \cW_{g,\mathrm{loc}}^{c\dagger_0}(x_1);\\
	x_2^{x,g} &:= H_{g,x_1^{x,g},x_2}^{*_1}(x_1^{x,g});\\
	%\cW_{g,\mathrm{loc}}^{*_1}(x_1^{x,g})\cap \cW_{g,\mathrm{loc}}^{c\dagger_1}(x_2);\\
	&\dots\\
	x_k^{x,g} &:= H_{g,x_{k-1}^{x,g},x_k}^{*_{k-1}}(x_{k-1}^{x,g});
	%x_k^{x,g} &:= \cW_{g,\mathrm{loc}}^{*_{k-1}}(x_{k-1}^{x,g})\cap \cW_{g,\mathrm{loc}}^{c\dagger_{k-1}}(x_k);
	\end{align*}	
here, for each $j \in \{0,\dots,k-1\}$, we let $*_j %,\dagger_i\}
\in \{s,u\}$ be such that $x_{j+1}\in \cW_f^{*_j}(x_{j})$. 
Moreover,   $\gamma^{x,g}$  depends continuously on the pair $(x,g)$. 
\end{lemma}

\begin{defi}\label{def 4 legged}
	Given a point $x \in M$ and an integer $n\geq 2$, a \emph{$2n$ us-loop} at $(f,x_0)$ is a $f$-accessibility sequence  $\gamma = [x_0,x_1,x_2,\dots,x_{2n}] \in M^{2n+1}$ with $2n$ legs  such that 
	\begin{align*}
	x_1 &\in \cWu_{f,\mathrm{loc}}(x_0),\\
	x_2 &\in \cWs_{f,\mathrm{loc}}(x_1),\dots\\
	\dots\, x_{2n-1} &\in \cWu_{f,\mathrm{loc}}(x_{2n-2})\cap \cWcs_{f,\mathrm{loc}}(x_0),\\
	x_{2n}&:=H_{f,x_{2n-1},x}^s(x_{2n-1})\in \cWc_{f,\mathrm{loc}}(x_0). 
	\end{align*}
	%$$, \ x_2 \in \cWs_{f}(x_1),\ \dots,\ x_{2n-1} \in \cWu_{f}(x_{2n-2})\cap \cWcs_f(x),\  x_{2n}:=H_{f,x_{2n-1},x}^s(x_{2n-1}).%\in \cWs_f(x_3)\cap \cWc_F(x).
	%$$
	We define $2n$ su-loops accordingly (with $x_1\in \cWs_{f,\mathrm{loc}}(x_0)$ etc.). 
	
	The \emph{length} of a $2n$ us-loop $\gamma= [x_0,x_1,x_2,\dots,x_{2n}] \in M^{2n+1}$ is defined as 
	$$
	\ell(\gamma):=d_{\cWu_f}(x_0,x_1) + \sum_{i=1}^{n-1} \Big[ d_{\cWs_f}(x_{2i-1},x_{2i}) + d_{\cWu_f}(x_{2i},x_{2i+1})\Big] + d_{\cWcs_f}(x_{2n-1},x_0).
	$$
	Moreover, we say that the us-loop $\gamma$ is 
	\begin{itemize}
		\item \emph{closed}, if $x_{2n}=x_0$;
		\item \emph{trivial}, if $x_0=x_1=x_2=\dots=x_{2n}$;
		\item \emph{non-degenerate}, if  $x_1$ is distinct from the other corners $x_0,x_2,\dots,x_{2n}$. 
	\end{itemize}
	We also denote by $\overline{\gamma}$ the $2n$ su-loop $\overline{\gamma} := [x_{2n},x_{2n-1},\dots,x_2,x_1,x_0] \in M^{2n+1}$ at $(f,x_{2n})$. Finally, given an integer $m \geq 2$ and a $2m$ us-loop $\gamma'=[x_{2n},x_{1}',\dots,x_{2m}']$ at $(f,x_{2n})$,  the concatenation $\gamma\gamma'$ of $\gamma$ and $\gamma'$ is the $2(m+n)$ us-loop $\gamma\gamma':=[x_0,x_1,\dots,x_{2n},x_1',\dots,x_{2m}']$ at $(f,x_0)$.  
	
	%Analogously, a \emph{$2n$ su-loop} is a $f$-accessibility sequence  with $2n$ legs such that the first leg connects two points in the same stable manifold. 

\end{defi}

\begin{defi}
	Given  $x \in M$ and $n \geq 2$, a one-parameter family $\gamma=\{\gamma(t) = [x,x_1(t), \dots,x_{2n}(t)]\}_{t\in [0,1]}$ of $2n$ us-loops at $(f,x)$ is said to be \emph{continuous} if for any $i=1,\dots,2n$, the map $t \mapsto x_i(t)$ is continuous. We define $\ell(\gamma) := \sup_{t \in [0,1]} \ell(\gamma(t))$. 
\end{defi}

\subsection{Structure of center accessibility classes}

%Let $X$ be a $C^1$ manifold.  We say that a subset $N\subset X$ is $C^1$-\emph{homogeneous} if for any two points $x,y \in N$, there exist a $C^1$ local   diffeomorphism $\phi$ of $X$  and neighbourhoods $U,V\subset X$ of $x$ and $y$ respectively, such that $\phi(x)=y$ and   $\phi(U\cap N)=V\cap  N$. Let us recall the following result about $C^1$-homogeneous subsets:
%
%\begin{theorem}[\cite{RSS}]
%	Any locally compact subset $N$ of a $C^1$ manifold $X$ that is $C^1$-homogeneous  is a $C^1$ submanifold of $X$. 
%\end{theorem}

Let $M$ be a  compact Riemannian manifold   of dimension $d \geq 4$. Let $r \geq 2$ be some integer, and let $f \in \cPH^r(M)$ be a  partially hyperbolic diffeomorphism  with $\dim E_f^c=2$ that is center bunched and dynamically coherent. 

By Theorem \ref{cor smooth holonomy maps}, for $*=s,u$, the $*$-holonomy maps are $C^1$ when restricted to a $\cW_f^{c*}$ leaf; by $C^1$-homogeneity arguments,  this allowed  \cite{RH,RHV} to obtain a classification of center accessibility classes. 

\begin{theorem}[\cite{RH,RHV}]\label{theorem structure}
For any  point $x \in M$, and for any sufficiently small $\varepsilon>0$, the local center accessibility class $\cac_f(x,\varepsilon)$ can be either
\begin{itemize}
	\item  trivial, i.e., reduced to a point;
	\item  a one-dimensional submanifold of $\cW_f^c(x)$;
	\item open; in this case, $\ac_f(x)$ is also open. 
\end{itemize}
\end{theorem}
In the following,   for any subset $\mathscr{S}\subset M$, we let 
\begin{itemize}
	\item $\Gamma_f^0(\mathscr{S}):=\{x \in \mathscr{S}: \cac_f(x)\text{ is trivial}\}$;
	\item $\Gamma_f^1(\mathscr{S}):=\{x \in \mathscr{S}: \cac_f(x)\text{ is one-dimensional}\}$;
	\item $\Gamma_f(\mathscr{S}):=\Gamma_f^0(\mathscr{S}) \cup \Gamma_f^1(\mathscr{S})$. 
\end{itemize}   
In particular, $\mathscr{S}\setminus \Gamma_f(\mathscr{S})$ is the set of points $x \in \mathscr{S}$ whose accessibility class $\ac_f(x)$ is open. 
When $\mathscr{S}=M$, we abbreviate $\Gamma_f^0(\mathscr{S}),\Gamma_f^1(\mathscr{S}),\Gamma_f(\mathscr{S})$ respectively as $\Gamma_f^0,\Gamma_f^1,\Gamma_f$. 

\section{Main results}\label{section main resul}

We fix a compact Riemannian manifold $M$ of dimension $d \geq 4$ and an integer $r \geq 2$. 
%and let $f \in \cPH^r(M)$ be a  partially hyperbolic diffeomorphism  with $\dim E_f^c=2$ that is center bunched and dynamically coherent. 
Our main result is about the $C^r$-density  of the accessibility property for partially hyperbolic diffeomorphisms with  two-dimensional center which are stably dynamically coherent and satisfy some strong bunching condition as in Definition \ref{stron bunch}, i.e., $f \in   \cPH_*^r(M)$. 

\begin{theoalph}\label{premier theo}
	For any partially hyperbolic diffeomorphism $f \in \cPH_*^r(M)$, resp. $f \in \cPH_*^r(M,\mathrm{Vol})$, with $\dim E_f^c=2$, that is dynamically coherent and plaque expansive, and for any $\delta>0$, there exists a partially hyperbolic diffeomorphism $g \in \cPH^r(M)$, resp. $g \in \cPH^r(M,\mathrm{Vol})$, with $d_{C^r}(f,g)< \delta$, such that $g$ is stably accessible. 
	
	In particular, by the work of Burns-Wilkinson \cite{BW1}, this implies that for any partially hyperbolic diffeomorphism $f \in \cPH_*^r(M,\mathrm{Vol})$, with $\dim E_f^c=2$, that is  dynamically coherent and plaque expansive, and for any $\delta>0$, there exists $g \in \cPH^r(M,\mathrm{Vol})$, with $d_{C^r}(f,g)< \delta$, such that $g$ is stably ergodic.   
\end{theoalph}

One intermediate step is to show that trivial accessibility classes can be broken by $C^r$-small perturbations. This part of the proof also holds when the center is higher dimensional and only requires center bunching. 

\begin{theoalph}\label{deuxieme theo}
	For any partially hyperbolic diffeomorphism $f \in \cPH^r(M)$, resp. $f \in \cPH^r(M,\mathrm{Vol})$, with $\dim E_f^c \geq 2$, that is center bunched, dynamically coherent, and plaque expansive, and for any $\delta>0$, there exists a partially hyperbolic diffeomorphism $g \in \cPH^r(M)$, resp. $g \in \cPH^r(M,\mathrm{Vol})$, with $d_{C^r}(f,g)< \delta$, such that   $\cac_g(x)$ is non-trivial, for all $x \in M$. 
\end{theoalph}\quad
%When the center dimension is equal to two, the $C^r$-density of accessibility essentially follows from the previous result, and the following one. 
%
%\begin{theoalph}\label{theoreme c}
%	For any partially hyperbolic diffeomorphism $f \in \cPH_2^r(M)$, resp. $f \in \cPH_2^r(M,\mathrm{Vol})$, with $\dim E_f^c=2$, that is center bunched, dynamically coherent, and plaque expansive, for any non-periodic point $x \in M$, for any sufficiently small $\rho>0$, and for any $\delta>0$,  there exists $g \in \cPH^r(M)$, resp. $g \in \cPH^r(M,\mathrm{Vol})$, such that  
%	\begin{enumerate} 
%		\item $d_{C^r}(f,g)< \delta$;
%		\item $\cWc_g(x,\rho)$ is contained in a single open accessibility class $\ac_g(x)$. 
%	\end{enumerate}
%\end{theoalph}
 Let us briefly summarize the main steps of the proof:
\begin{enumerate}
	\item\label{descr point un} we study the structure of local center accessibility classes, i.e., the set of points which can be attained within some small center disk around a given point,  following accessibility sequences with a given number of legs of prescribed size; in particular, we identify which are the configurations to break in order to make each accessibility class open;
	\item\label{descr point deux} given a small center disk $\mathcal{D}$, we construct continuous families of local accessibility sequences at points in $\mathcal{D}$; these families depend on the nature of the center accessibility class of the base point (which can be zero, one or two-dimensonal), and allow us to have sufficiently many ``degrees of freedom'' to create local accessibility after perturbation;
	\item\label{descr point trois} once these families are constructed, we design families of perturbations,  localized near one of the corners of the accessibility sequences, and which depend in a differentiable way on the perturbation parameter;
	\item\label{descr point quatre} we study the variation of the endpoint of these accessibility sequences once the perturbation parameter is turned on, and show that for suitable perturbations, we obtain a submersion from the space of perturbations to the phase space; in particular, bad configurations in phase space (non-open accessibility classes) correspond to special configurations in the space of perturbations, which can be broken to create local accessibility;
	\item\label{descr point cinq} we globalize the argument using spanning families.
\end{enumerate}

Let us say a few more words about the previous points. 

The details about point \eqref{descr point un} are given in Section \ref{section one dim}. For partially hyperbolic diffeomorphisms with two-dimensional center that are center bunched, it is known (by the works of Rodriguez-Hertz \cite{RH}, Rodriguez-Hertz and V\'asquez \cite{RHV} etc.) that center accessibility classes are zero, one or two-dimensional submanifolds. Moreover, Horita-Sambarino \cite{HS} have studied the organization of center accessibility classes within a small center disk all of whose points have non-trivial center accessibility classes; in particular, they have shown that the set of one-dimensional center accessibility classes of points in the disk forms a $C^1$ lamination. In Section \ref{section one dim}, we go further in this direction, and investigate the variation of center accessibility classes for perturbations of a given partially hyperbolic diffeomorphism. In particular, we show that if the center accessibility class of a point $x$ remains one-dimensional after perturbation, it stays in a certain ``cone'' around $x$. 

The construction of loops mentioned in point \eqref{descr point deux} is outlined in Section \ref{section adapt loop}. Indeed, in the subsequent argument, given a point $x$ whose  accessibility class is not open, we need to construct (non-trivial) \emph{closed} accessibility sequences at $x$; moreover, we show that it is possible to construct these loops in such a way that they depend nicely on $x$. 

The details about point \eqref{descr point trois} are in Section \ref{section sumb}, and follow the arguments of \cite{LZ}. Given a point $x \in M$ that is non-periodic, we construct a family $\{\gamma(t)\}_{t \in [0,1]}$ of contractible \emph{us-loops} at $(f,x)$, and we define a family of perturbations such that the support of the perturbations is contained in some small neighbourhood of the first corner of $\gamma(1)$. By taking the loop   sufficiently small, the first return time to the support of the perturbation  can be made arbitrarily large, and we show that it induces a change of the holonomy along the continuation of $\gamma(1)$ for the perturbed diffeomorphisms. More precisely, by the results of \cite{LZ}, we get a submersion from the space of perturbations to the phase space -- here, the local center leaf of $x$. 

The submersion property is sufficient to show that after perturbation, the center accessibility class of $x$ can be made non-trivial. This part of the proof is explained in Subsection \ref{break trivial} and holds in a more general setting, as it does not require the center to be two-dimensional. When the center accessibility class of  $x$ is one-dimensional, by point \eqref{descr point un}, it varies continuously with respect to the diffeomorphism in the $C^1$ topology. In particular, if the center accessibility class of $x$ were one-dimensional for every diffeomorphism in a $C^r$-neighbourhood of $f$, then all those classes would stay in some cone around the point $x$; but this is in contradiction with the submersion property for the  family of perturbations we construct. The details of this part are given in Subsection \ref{opening one dim}. 

The details about point \eqref{descr point cinq} are given in Section \ref{setion dens}, where we explain how to globalize  the arguments in order to verify the accessibility property, through the notion of spanning family of center disks, as in \cite{DW} (see Subsection \ref{spanning famm}). In Subsection \ref{break trivial global}, given some small center disk in the family, we explain how by a $C^r$-small perturbation, it is possible to make the center accessibility class of each point in the disk non-trivial. One difficulty is that the perturbation used to break trivial center accessibility classes may create new trivial classes in other places (at points with non-trivial, but very small center accessibility classes). The idea to bypass this difficulty is to take two families of us-loops which we can perturb ``independently'', in order to increase the codimension of ``bad'' situations for which the center accessibility class of some point in the disk would be trivial. Once all classes in the disk are non-trivial, we have to make a further perturbation to make all these classes simultaneously open. One important step in the argument is the aforementioned result (inspired by the work of Horita-Sambarino \cite{HS}) that within the center disk, one-dimensional center accessibility classes vary $C^1$-continuously both in perturbation space and phase  space. In particular, if the center disk is chosen sufficiently small, then the set of tangent directions associated to one-dimensional classes (even for small perturbations of the diffeomorphism $f$) stay in a small cone that is uniform in the points of the disk. Thanks to the submersion property, we can then choose a perturbation for which each point $x$ in the center disk will have a point $y$ in its center accessibility class lying outside this cone, which forces the accessibility class of $x$ to be open. There again, one difficulty is to check that the perturbations which we make preserve the accessibility classes which were already open. Repeating the same argument for each center disk in the spanning family, we thus construct a $C^r$-small perturbation of $f$ that is accessible. 

\section{Variation of one-dimensional center accessibility classes}\label{section one dim}

In this section, given an integer $r\geq 2$, we prove that the set of one-dimensional center accessibility classes varies continuously in the $C^{1}$ topology with respect to $f\in \mathcal{PH}_*^r(M)$. The idea of the proof is similar to Proposition 2.19 from \cite{HS} where it is proved that for a fixed partially hyperbolic diffeomorphism, and for a given center disk, the one-dimensional accessibility classes form a $C^{1}$-lamination. To prove this, we have to see that for a given $x\in M$, the direction $T_{x}\cac_{f}(x)$ varies continuously with respect to $f$ in the $C^{1}$ topology. 

Let us fix an integer $r\geq 2$. We denote by  $\mathscr{F}\subset \mathcal{PH}_*^r(M)$ the set of $C^r$ dynamically coherent, plaque expansive,  partially hyperbolic diffeomorphisms with two-dimensional center which satisfy the bunching condition in Definition \ref{stron bunch}. Let $f\in \mathscr{F}$. By center bunching, for $*=s,u$, for any  $x \in M$, $y \in \mathcal{W}_{f,\mathrm{loc}}^*(x)$,   the   holonomy map $H_{f,x,y}^*$  is $C^1$ when restricted to the  leaf $\mathcal{W}_{f,\mathrm{loc}}^{c*}(x)$. %, resp. center-unstable leaf $\mathcal{W}_f^{cu}(x)$. 
For any $C^1$ neighbourhood $\mathcal{U}$ of $f$, we will denote by $\mathcal{U}^\mathscr{F}$ the set  $\mathcal{U}^\mathscr{F}:=\mathcal{U} \cap \mathscr{F}$. 

In the following, we will need to have uniform control of the differential of the holonomies $H_{f,x,y}^*$ in two ways:
\begin{itemize}
	\item with respect to the points $x,y$ (in the same stable/unstable manifold);
	\item when the diffeomorphism $f$ is replaced with another $C^r$ partially hyperbolic diffeomorphism in a $C^1$-neighbourhood of $f$. 
\end{itemize}
This is the content of the next lemma. 
\begin{lemma}[See \cite{O}, and also \cite{B,BW2}]\label{lemme Davi}
	Let $f\in \mathscr{F}$. Then there exists a $C^1$ neighbourhood $\mathcal{U}$ of $f$ such that for $*=s,u$ and  $\mathcal{U}^\mathscr{F}=\mathcal{U} \cap \mathscr{F}$, the family of $C^1$ maps $\{H_{g,x,y}^*|_{\cW_g^{c}(x)}\}_{g\in \mathcal{U}^\mathscr{F},\, x \in M,\, y \in \mathcal{W}_f^*(x)}$  depends continuously in the $C^1$ topology with the choices of the points $x,y$ and of the map $g \in \mathcal{U}^\mathscr{F}$. 
\end{lemma}

\begin{remark}
In fact, Obata \cite{O} shows that for $*=s,u$, the family of holonomy maps $\{H_{f,x,y}^*|_{\cW_f^{c}(x)}\}_{x \in M,\, y \in \mathcal{W}_f^*(x)}$  depends continuously in the $C^1$ topology with the choices of the points $x,y$, when $f$ is dynamically coherent and satisfies a strong bunching condition. For our purpose, we also  need to have a uniform control with respect to the diffeomorphism $g$ in a $C^1$-small neighbourhood of $f$. It is indeed possible as the estimates in \cite{O} are  written in terms of the functions as in \eqref{functions nu gamma} controlling the growth rates along the different invariant bundles, which depend continuously on the map $g$ in the $C^1$ topology.  
\end{remark}

The holonomy map associated to some accessibility sequence is obtained by composing  the holonomy maps between  two consecutive corners (recall \eqref{compo holonomies}). By the previous lemma, we thus have:
\begin{corollary}\label{lemma dependence family}
	Let $f\in \mathscr{F}$, and let $\gamma=[x_1,x_2,\dots,x_k]\in M^{k}$ be a $f$-accessibility sequence for some integer $k \geq 1$. %For simplicity, we assume that $n$ is odd and that $x_2 \in \cW_f^u(x_1)$; the other cases are similar. 
	We take a small neighbourhood $\mathcal{O}\subset M$ of $x_1$ and a $C^1$ neighbourhood $\mathcal{U}$ of $f$ such that for any $x\in \mathcal{O}$ and for any $g \in \mathcal{U}$, the  continuation $\gamma^{x,g}=[x,x_2^{x,g},\dots,x_k^{x,g}]$ of $\gamma$ starting at $x$  given by Lemma \ref{lemma continuation} is well-defined.  %, where
	%\begin{align*}
	%x_2^g &:= \cW_g^u(y)\cap \cW_g^{cs}(x_2);\\
	%x_3^g &:= \cW_g^s(x_2^g)\cap \cW_g^{cu}(x_3);\\
	%&\dots\\
	%x_n^g &:= \cW_g^s(x_{n-1}^g)\cap \cW_g^{cu}(x_n).
	%\end{align*}
	%Moreover, $\gamma^g(y)$ 
	%and depends continuously on the pair $(x,g)$. 
	
	Then,	%for $\mathcal{U}^\mathscr{F}:=\mathcal{U} \cap \mathscr{F}$, 
	the family of $C^1$ maps $\{H_{g,\gamma^{x,g}}^*|_{\cW_g^{c}(x)}\}_{x \in \mathcal{O},\, g\in \mathcal{U}^\mathscr{F}}$  depends continuously in the $C^1$ topology with the choices of  the point $x \in \mathcal{O}$ and the map $g \in \mathcal{U}^\mathscr{F}$. 
\end{corollary}

For any point $x \in M$ and any subset $\mathscr{G} \subset \mathscr{F}$, we let $\mathscr{G}_1(x)\subset \mathscr{G}$ be the subset of maps $f$ for which the center accessibility class $\cac_{f}(x)$ is one-dimensional. For any $f \in \mathscr{G}_1(x)$, and for any sufficiently small $\theta,\varepsilon>0$, we let $\mathscr{C}_f(x,\theta,\varepsilon)\subset \cW_f^c(x)\cap B(x,\varepsilon)$ be the set of points in the $\varepsilon$-ball $B(x,\varepsilon)$ centered at $x$ which belong to (the image by the exponential map of) the cone of angle $\theta$ around $T_x \cac_f(x)$, i.e., 
\begin{equation}\label{def connne}
\mathscr{C}_f(x,\theta,\varepsilon):=\exp_x\{y\in T_x M: \angle (y,T_x \cac_f(x))<\theta\}\cap B(x,\varepsilon). 
\end{equation}

The main result of this section is the following:
\begin{prop}\label{lemma var des classes}
	Take $f\in \mathscr{F}$ and $x\in M$ such that $\cac_{f}(x)$ is one-dimensional, i.e., $f \in \mathscr{F}_1(x)$. Then, for every $\theta>0$ there exists a $C^1$ neighbourhood $\cU$ of $f$ such that for every $g\in \mathcal{U}^\mathscr{F}_{1}(x)$, the angle at $x$ between $\cac_{f}(x)$ and $\cac_{g}(x)$ satisfies  
	$$
	\angle (T_{x}\cac_{f}(x),T_{x}\cac_{g}(x))<\theta.
	$$ 
	Moreover, there exists $\varepsilon_0>0$ such that for any $g \in \cU^\mathscr{F}_1(x)$ and $\varepsilon \in (0,\varepsilon_0)$,  it holds 
	$$ \cac_g(x,\varepsilon)\subset \mathscr{C}_f(x,\theta,\varepsilon).$$  
\end{prop}

\begin{figure}[H]
	\begin{center}
		\includegraphics [width=13.5cm]{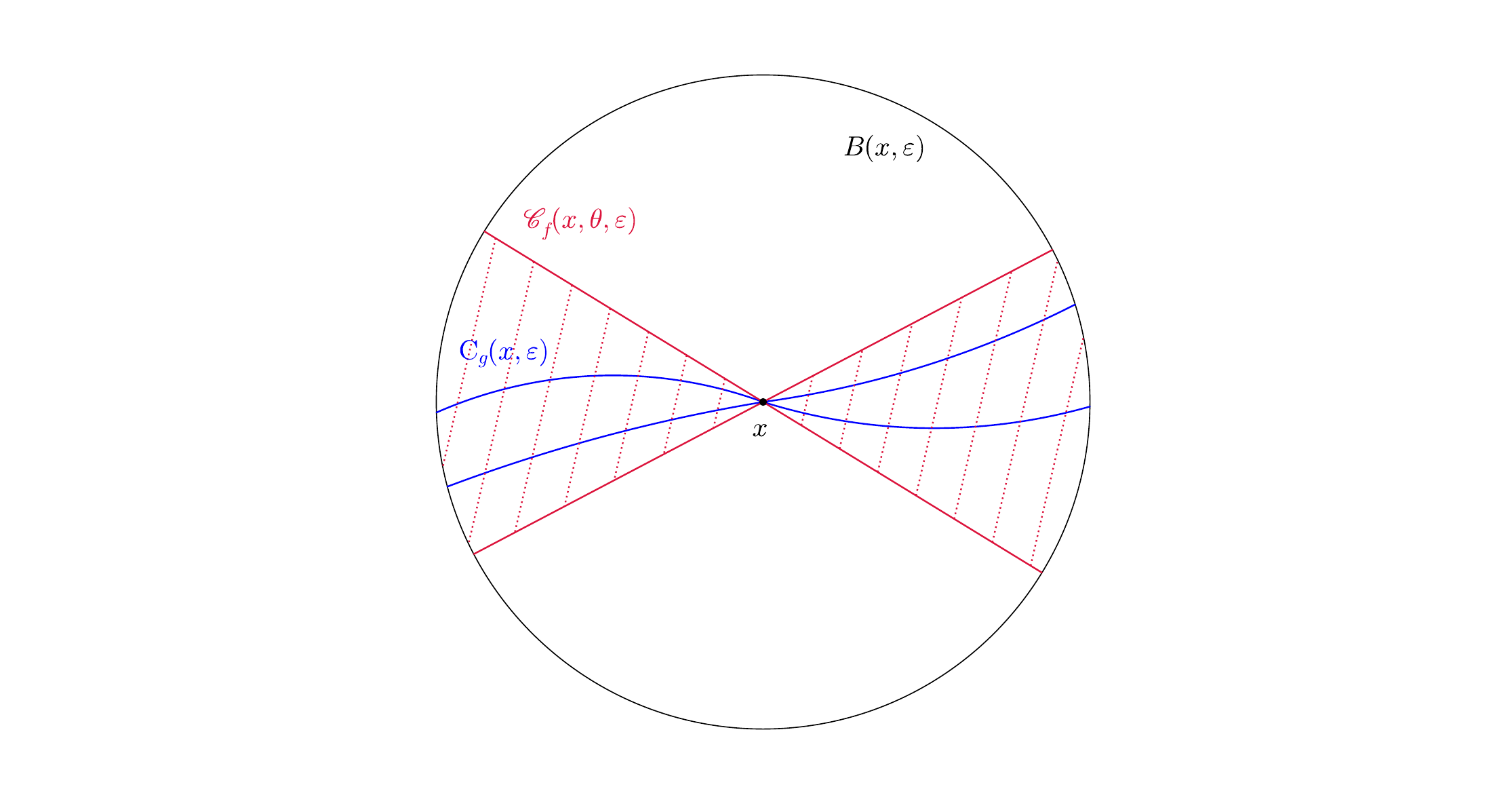}
		\caption{Variation of $1$-dimensional center accessibility classes.}
	\end{center}
\end{figure}

The idea of the proof consists in showing some ``uniform" homogeneity of one-dimensional center accessibility classes $\mathrm{C}_g(x)$ of all maps $g\in \mathcal{U}^\mathscr{F}$, for a sufficiently small $C^1$ neighbourhood $\mathcal{U}$ of a fixed $f \in \mathscr{F}$. Indeed, the tangent spaces at two points $x,y$ in the same center accessibility class are naturally related through the differential of the holonomy map along an accessibility sequence connecting $x$ to $y$. Since everything we are doing here is local, we are able to compare angles and norms of vectors in different tangent spaces, using trivialization charts as follows. Recall that $d:=\dim M$, and that we denote $d_s:=\dim E_f^s$, $d_u:=\dim E_f^u$. 
\begin{lemma}[see Construction 9.1, \cite{LZ}]\label{lemma constr phi}
	There exist $C^2$-uniform constants $\overline{h}=\overline{h}(f)>0$ and   $\overline{C}=\overline{C}(f)  > 1$ such that %for any $h \in (0,\overline{h})$, and 
	for any  $x \in M$, there  exists a $C^r$ volume preserving map $\phi = \phi_x \colon (-\overline{h},\overline{h})^{d} \to M$ such that $\phi(0_{\R^d}) = x$ and
	\begin{enumerate}
		\item\label{notation 3 1} $\cWc_f(x,\frac{\overline{h}}{5}) \subset \phi\big((-\frac{\overline{h}}{4}, \frac{\overline{h}}{4})^{2} \times \{0\}^{d_u + d_s}\big)\subset \phi\big((-\frac{2\overline{h}}{3}, \frac{2\overline{h}}{3})^{2} \times \{0\}^{d_u + d_s}\big) \subset \cWc_f(x,\overline{h})$;
		\item\label{notation 3 2} $\|\phi\|_{C^2} < \overline{C}$;
		\item\label{notation 3 3} $D\phi(0,  \R^{2} \times \{0_{\R^{d_u + d_s}}\})$, $D\phi(0, \{0_{\R^{2}}\} \times \R^{d_u} \times \{0_{\R^{ d_s}}\})$, $D\phi(0, \{0_{\R^{2+d_u}}\} \times \R^{d_s})$ are respectively equal to $E^{c}_f(x),E^{u}_f(x), E^{s}_f(x)$;
		\item\label{notation 3 4} for any $y \in \phi((-\overline{h},\overline{h})^{d})$, $\Pi^c D(\phi^{-1})_y\colon E^c_{f}(y) \to \R^2$ has determinant in $\big(\overline{C}^{-1},\overline{C}\big)$, where $\Pi^c \colon \R^{d} \simeq \R^{2} \times \R^{d_u+d_s} \to \R^{2}$ is the canonical projection;
		\item\label{notation 3 5}
		for any $\zeta > 0$, there exists a $C^1$-uniform constant $\overline{h}_{\zeta}=\overline{h}_{\zeta}(f) \in (0,\overline{h})$ so that if $h\in (0,\overline{h}_{ \zeta})$, then for any $y \in \phi((-h,h)^{d})$, the norm of  $\Pi^c(D\phi^{-1})_y\colon E^{su}_{f}(y)  \to \R^2$ is smaller than $\zeta$.
	\end{enumerate} 
In the following, we will denote by $\Pi^c_x$ the map $\Pi_x^c:= \Pi^c \circ \phi_x^{-1} \colon M \to \R^2$.
\end{lemma}

Before giving the proof of Proposition \ref{lemma var des classes}, let us state an elementary lemma and introduce a notation. Let $\alpha\colon [0,1]\to M$ be a $C^{1}$ arc of $M$ and given $\epsilon>0$, consider an $\epsilon$ tubular neighbourhood $\mathscr{N}_{\alpha,\epsilon}$ of $\alpha$. This tubular neighbourhood is diffeomorphic to $[0,1]\times [-\epsilon,\epsilon]^{d-1}$. We identify points in $\mathscr{N}_{\alpha,\epsilon}$ with pairs $(t,s)$, where $t\in [0,1]$ and $s\in [-\epsilon,\epsilon]^{d-1}$. We call the boundary $\{0\}\times [-\epsilon, \epsilon]^{d-1}$ the \emph{left side} of $\mathscr{N}_{\alpha,\epsilon}$, and we call the boundary $\{1\}\times [-\epsilon,\epsilon]^{d-1}$ its \emph{right side}. We denote by $\xi\colon \mathscr{N}_{\alpha,\epsilon}\to \alpha$ the projection $\xi\colon (t,s)\mapsto \alpha(t)$. 

\begin{lemma}\label{elelemHS}
	With the notation above, given $\delta>0$, there exists $\epsilon>0$ such that if $\beta\colon [0,1] \to \mathscr{N}_{\alpha,\epsilon}$ is a $C^{1}$ curve in $\mathscr{N}_{\alpha,\epsilon}$ from the left to the right side, then there exists some  $(t,s)=\beta(\tilde{t})$ with $\tilde t \in [0,1]$ such that the angle between $\alpha$ and $\beta$ satisfies  
	$$
	\angle(\dot{\alpha}(t),\dot{\beta}(\tilde{t}))<\delta.
	$$
\end{lemma}	 

\begin{proof}[Proof of Proposition \ref{lemma var des classes}]
	Let us show the first part. Suppose by contradiction that for some $\eta>0$, there exists a sequence $(g_n)_{n \geq 0}\in \mathscr{F}^\N$ of maps such that  $g_{n}\to f$ in the $C^1$ topology, with $g_n \in \mathscr{F}_1(x)$ and  
	\begin{equation}\label{angle eta}
	\angle(T_{x}\cac_{f}(x),T_{x}\cac_{g_n}(x))>\eta,\quad \text{for all } n \geq 0.
	\end{equation}
	
	Since $\cac_f(x)$ is one-dimensional, for some integer $n \geq 2$, there exists a $2n$ us-loop $\gamma=[x,x_1,\dots,x_{2n}]$ at $(f,x)$ such that $x_{2n} \neq x$. By shrinking the size of the legs, we get a one-parameter family $\{\gamma(t)=[x,x_1(t),\dots,x_{2n}(t)]\}_{t\in [0,1]}$ of $2n$ us-loops at $(f,x)$, where $\gamma(1)=\gamma$ and $\gamma(0)$ is trivial.   By Lemma \ref{lemma continuation}, there exists %ecause of the local product structure, we can extend this family of $us$-loops to a 
	a $C^1$ neighbourhood $\widetilde \cU$ of $f$ such that for any $g\in \widetilde \cU$  % that is sufficiently $C^1$-close to $f$, 
	and for any $t \in [0,1]$,  %we let
	%\begin{itemize}
	%\item $x_{1}(g,t):=\cW_{g,\mathrm{loc}}^{cs}(x_{1}(t))\cap \cW^{u}_{g,\mathrm{loc}}(x)$;
	%\item $x_{2}(g,t):=\cW_{g,\mathrm{loc}}^{cu}(x_{2}(t))\cap \cW^{s}_{g,\mathrm{loc}}(x_{1}(g,t))$\dots
	%\item \dots $x_{2n-1}(g,t):=\cW_{g,\mathrm{loc}}^{cs}(x)\cap \cW^{u}_{g,\mathrm{loc}}(x_{2n-2}(g,t))$;
	%\item $x_{2n}(g,t):=\cW_{g,\mathrm{loc}}^{c}(x)\cap \cW^{s}_{g,\mathrm{loc}}(x_{2n-1}(g,t))$.
	%\end{itemize}   
	%Thus for any $g \in \widetilde\cU$ we get a family
	there exists a one-parameter family  $\{\gamma^{x,g}(t)=[x,x_1^{x,g}(t),\dots,x_{2n}^{x,g}(t)]\}_{t\in [0,1]}$ of $2n$ us-loops at $(g,x)$ such that $\gamma^{x,g}(0)$ is the trivial loop. We also denote    $\alpha_{g,x} \colon t \mapsto x_{2n}^{x,g}(t)\in \cac_g(x)$. 
	
	For each pair $(g,t)\in \widetilde{\mathcal{U}} \times [0,1]$  we have the corresponding holonomy map $H_g^{t}:=H_{g,\gamma^{x,g}(t)}|_{\cW^{c}_{g,\mathrm{loc}}(x)}\colon\cW^{c}_{g,\mathrm{loc}}(x)\to \cW^{c}_{g,\mathrm{loc}}(x)$. Given some small $h>0$, and assuming that $\widetilde{\mathcal{U}}$ is sufficiently small, then for every  map $g\in \widetilde{\mathcal{U}}^\mathscr{F}$, we   take a $C^1$ chart $\phi_{x,g} \colon(-h,h)^2 \to \cW^{c}_{g,\mathrm{loc}}(x)$ as in Lemma \ref{lemma constr phi}; as center leaves vary continuously with respect to $g$ in the $C^{1}$ topology, the map $g \mapsto \phi_{x,g}$  depends continuously on $g$ in the $C^1$ topology. After replacing $H_g^t$ with $\phi_{x,g}^{-1} \circ H_g^t \circ \phi_{x,g}$, we can compare angles and norms of vectors for diffeomorphisms in a neighbourhood of $f$; by a slight abuse of notation,  we will still denote this map by $H_g^t$ for simplicity. By Corollary \ref{lemma dependence family}, and by compactness of $[0,1]$, we deduce that the family of holonomy maps $\{H_g^{t}\}_{(t,g)\in [0,1]\times \widetilde{\mathcal{U}}^\mathscr{F}}$ is uniformly $C^1$.  In particular,  for any $\delta>0$, %for any $t \in [0,1]$, 
	there exists a $C^1$ neighbourhood $\mathcal{U}_\delta$ of $f$ such that for $\mathcal{U}^\mathscr{F}_\delta:=\mathcal{U}_\delta  \cap\mathscr{F} $, it holds 
	\begin{equation}\label{unif hom}
	\sup_{(t,g) \in [0,1]\times \mathcal{U}^\mathscr{F}_\delta}\|DH_g^{t}-D H_f^{t}\|< \delta.
	\end{equation}
	% can be made uniformly small for $g$ in a sufficiently small $C^1$ neighbourhood of $f$. 
	
	%Now call $y:=x_{4}(f,1)$. 
	Therefore, for every $\theta>0$, there exist $\delta_{0}>0$, $\rho_{0}>0$   such that for $g\in \cU^\mathscr{F}_{\delta_0}$ and for any $t\in [0,1]$, if $y\in \cW^{c}_{g,\mathrm{loc}}(x)$ is such that $d(x,y)<\rho_{0}$ and if the  vectors $v,w\in \R^2$ satisfy $\angle(v,w)>\theta$, then we have  
	\begin{equation}\label{HS1}
	\angle(D_{x} H_{f}^{t}(v),D_{y} H_{g}^{t}(w))>\delta_{0}. 
	\end{equation} 
	%To see this, take $\theta>0$ and fix an arbitrary $\bar{t}\in [0,1]$. By \eqref{unif hom}, there are $\rho_{\bar{t}}, \delta_{\bar{t}}, \Delta_{\bar{t}}>0$ and a $C^1$ neighbourhood $\cU(\bar{t})$ of $f$ such that: for every $t\in (\bar{t}-\Delta_{\bar{t}},\bar{t}+\Delta_{\bar{t}})$, $y\in B(x,\rho_{\bar{t}})$, $g \in \cU^\mathscr{F} (\bar{t})$ and vectors $v,w$ making an angle $\angle(v,w)>\theta$, we have $\angle(D_{x} H_{f}^{t}(v),D_{y}H_{g}^{t}(w))>\delta_{\bar{t}}$. By compactness of $[0,1]$, there exist  $\{\bar{t}_{j}\}_{j=1}^{k}$ such that $[0,1]=\bigcup_{j=1}^{k}(\bar{t}_{j}-\Delta_{\bar{t}_{j}}, \bar{t}_{j}+\Delta_{\bar{t}_{j}})$. The claim follows, by taking
	%\begin{itemize}
	%\item $\cU:=\bigcap_{j=1}^{k} \cU(\bar{t}_j)$;
	%\item $\rho_{0}:=\min \{\rho_{\bar{t}_{j}}: j=1,\dots,k\}$;
	%\item $\delta_{0}:=\min \{\delta_{\bar{t}_{j}}: j=1,\dots,k\}$.
	%\end{itemize} 
	
	% In the same way (taking a smaller $\delta_0$ if necessary), 
	As invariant manifolds depend continuously on the diffeomorphism $g \in \mathcal{U}_{\delta_0}^\mathscr{F}$, for any $\epsilon_0>0$,  there exists $\rho(\epsilon_0)>0$ such that for any $y \in B(x,\rho(\epsilon_0))$, for any $t \in [0,1]$ and for any $g \in \mathcal{U}_{\delta_0}^\mathscr{F}$ (taking a smaller $\delta_0$ if necessary),  it holds 
	\begin{equation}\label{HS3} 
	d(H_{f}^{t}(x),\xi(H_{g}^{t}(y)))<\epsilon_{0}.
	\end{equation}
Since the center accessibility class $\cac_{f}(x)$ is $C^{1}$, the map $\cac_{f}(x)\ni z\mapsto T_{z}\cac_{f}(x)$ is continuous, hence, if %we get that for any $\delta_1>0$ 
	%there exists $\epsilon_{0}>0$ such that 
	$\epsilon_{0}>0$  is chosen sufficiently small, then for any $t \in [0,1]$,  
	$g \in \mathcal{U}_{\delta_0}^\mathscr{F}$, and $y \in B(x,\rho(\epsilon_0))$,  %$d(x,y)<\rho(\epsilon_0)$, %we have  $d(H_{f}^{t}(x),\xi(H_{g}^{t}(y)))<\epsilon_{0}$,  
	we have 
	\begin{equation}\label{HS2}
	\angle(T_{H_{f}^{t}(x)}\cac_{f}(x),T_{\xi(H_{g}^{t})(y)}\cac_{f}(x))<\frac{\delta_{0}}{2}.
	\end{equation} 
	Now we argue as in Proposition 2.19 of \cite{HS}. For $\theta=\eta$ (recall \eqref{angle eta}) we take $\delta_{0}=\delta_{0}(\theta)>0$ as in \eqref{HS1} and we set $\delta:=\frac{\delta_{0}}{2}>0$. %Take $\epsilon_{0}$ as in \eqref{HS2}. %Finally choose $0<\rho<\rho_{0}$ such that \eqref{HS3} holds.
	
	Since $g_n\to f$, we can take $n$ large enough so that $g_n \in \mathcal{U}_{\delta_0}^\mathscr{F}$ and such that the arc $\{\alpha_{g_n,x}(t)\}_{t\in [0,1]}$ is a curve that crosses $\mathscr{N}_{\alpha_{f,x},\epsilon_0}$ from the left to the right side.  Set $\beta:=\alpha_{g_n,x}$. Note that if $t\in [0,1]$ is such that $\beta(t)\in \mathscr{N}_{\alpha_{f,x},\epsilon_0}$, then 
	\begin{equation}\label{homogeneite}
	\textnormal{Span}(\dot{\beta}(t))=D_{x} H_{g_n}^{t}(T_{x}\cac_{g_{n}}(x)).
	\end{equation} 
	Then, by \eqref{angle eta}, \eqref{HS1}, \eqref{HS2}, \eqref{homogeneite}, we deduce that for any $t \in [0,1]$,
	\begin{align*}
		&\angle(\dot{\beta}(t),T_{\xi(\beta(t))}\cac_{f}(x)) = \angle(D_{x}H_{g_n}^{t}(T_{x}\cac_{g_{n}}(x)),T_{\xi(H_{g_n}^{t}(x))} \cac_{f}(x)) \\
		&\geq  \angle(D_{x}H_{g_n}^{t}(T_{x}\cac_{g_n}(x)), T_{H_{f}^{t}(x)}\cac_{f}(x))
		- \angle(T_{H_{f}^{t}(x)}\cac_{f}(x), T_{\xi(H_{g_n}^{t}(x))}\cac_{f}(x)) \\
		&= \angle(D_{x}H_{g_n}^{t} (T_{x}\cac_{g_{n}}(x)),D_{x}H_{f}^{t}(T_{x}\cac_{f}(x)))
		-\angle(T_{H_{f}^{t}(x)}\cac_{f}(x), T_{\xi(H_{g_n}^{t}(x))}\cac_{f}(x)) \\
		&>\delta_{0}-\frac{\delta_{0}}{2}=\delta,
	\end{align*} 
	which contradicts Lemma \ref{elelemHS}. 
	
		\begin{figure}[H]
		\begin{center}
			\includegraphics [width=13.5cm]{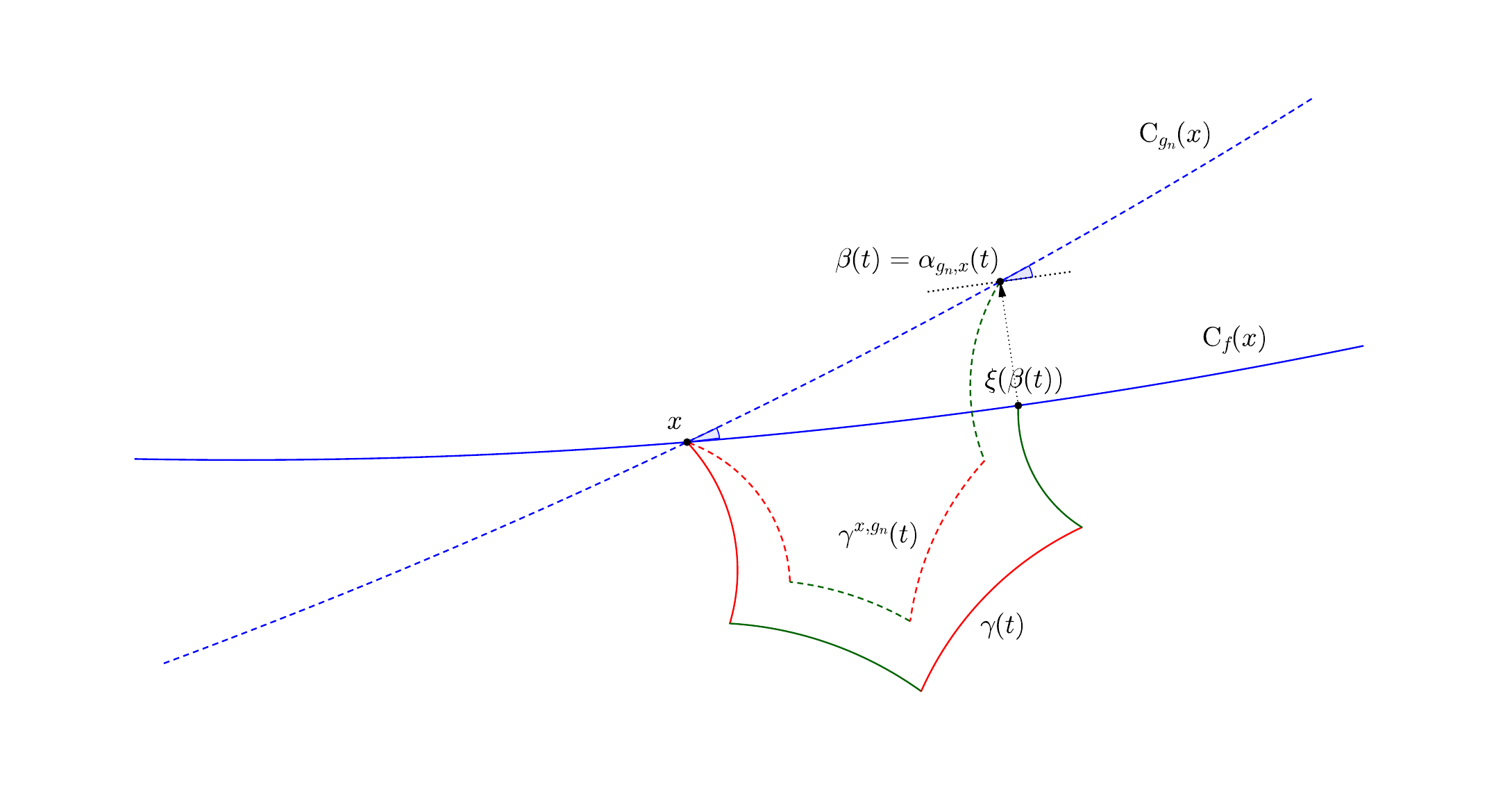}
			\caption{Tangent spaces to $\cac_f(x)$, resp. $\cac_{g_n}(x)$ at $\xi(\beta(t))$, resp. $\beta(t)$.}
		\end{center}
	\end{figure}
	
	The proof of the second part of the proposition is similar. We know from the previous part that given $\theta>0$ there is a $C^1$ neighbourhood $\cU$ of $f$ such that for every $g\in \cU^\mathscr{F}_{1}(x)$, it holds $\angle(T_{x}\cac_{f}(x),T_{x}\cac_{g}(x))<\theta$. Now we want to see the variation of the leaves at uniform (small) scale. Let us then suppose by contradiction that there are sequences $g_{n}\to f$ and $x_{n}\to x$ such that $x_{n}\in \cac_{g_{n}}(x,\frac 1n)\setminus \mathscr{C}_f(x,\theta,\frac 1n)$. By Lagrange Mean Value Theorem, this implies that there is $y_{n}\in \cac_{g_{n}}(x,\frac 1n)$ such that $\angle(T_{y_{n}}\cac_{g_{n}}(x),T_{x}\cac_{f}(x))>\theta$. Take  $\delta_0=\delta_0(\theta)>0$, $\epsilon_0>0$ sufficiently small, and $n>0$ sufficiently large such that  $g_n \in \mathcal{U}_{\delta_0}^\mathscr{F}$, $y_{n}\in B(x,\rho(\epsilon_0))$,  and such that the curve $\beta_{1}\colon t\mapsto H_{g_{n}}^{t}(y_{n})$ crosses $\mathscr{N}_{\alpha_{f,x},\epsilon_0}$ from the left to the right side. Now we argue as above, the only difference being that the role of the ``big angle'' is played by $\angle( T_{y_{n}}\cac_{g_{n}}(x),T_{x}\cac_{f}(x))$ instead of $\angle(T_{x}\cac_{g_{n}}(x),T_{x}\cac_{f}(x))$: for any $t \in [0,1]$, it holds 
	\begin{align*}
		&\angle(\dot{\beta_{1}}(t),T_{\xi(\beta_{1}(t))}\cac_{f}(x)) = \angle(D_{y_n}H_{g_n}^{t}(T_{y_n}\cac_{g_{n}}(x)),T_{\xi(H_{g_n}^{t}(y_n))}\cac_{f}(x)) \\
		&> \angle(D_{y_n}H_{g_n}^{t}(T_{y_n}\cac_{g_{n}}(x)),T_{H_{f}^{t}(x)}\cac_{f}(x))
		- \angle(T_{H_{f}^{t}(x)}\cac_{f}(x), T_{\xi(H_{g_n}^{t}(y_n))}\cac_{f}(x)) \\
		&= \angle(D_{y_n}H_{g_n}^{t}(T_{y_n}\cac_{g_{n}}(x)),D_{x} H_{f}^{t}(T_{x}\cac_{f}(x))) 
		-\angle(T_{H_{f}^{t}(x)}\cac_{f}(x), T_{\xi(H_{g_n}^{t}(x))}\cac_{f}(x)) \\
		&>\delta_{0}-\frac{\delta_{0}}{2}=\delta,
	\end{align*}
	which again contradicts Lemma \ref{elelemHS}. This concludes the proof.
\end{proof}

 As it will be used in the proof, let us recall the following result of \cite{HS}:
\begin{prop}[Corollary 2.20, \cite{HS}]\label{prop lamination}
	Let $\mathcal{C}$ be a center disk  of $f$ such that $\mathcal{C}\cap  \Gamma_f^0=\emptyset$. Then the set $\Gamma_f^1(\mathcal{C})$ of points with $1$-dimensional center accessibility classes in $\mathcal{C}$ admits a $C^1$ lamination whose leaves are the manifolds $\cac_f(y)\cap \cC$, $y \in \Gamma_f^1(\mathcal{C})$.  
\end{prop}

\section{Construction of adapted accessibility loops}\label{section adapt loop}

Let $r \geq 2$, and let us consider a partially hyperbolic diffeomorphism $f \in \cPH^r(M)$ with $\dim E_f^c \geq 2$ that is center bunched, dynamically coherent, and plaque expansive.
In this section, given a point $x\in M$, we build suitable loops starting at $x$ which will later be used to construct perturbations to break non-open accessibility classes. The loops which we construct will depend on whether the accessibility class of the point $x$ is already open or not. In fact, although the accessibility class of $x$ is a homogeneous set, when working with specific families of loops with a prescribed number of legs of a certain size, the set of points which we can reach from $x$ moving along these loops may not exhibit the global structure of the accessibility class (for example, if the class of $x$ is open, to be able to reach any point in a neighborhood of $x$, we may need to consider very long accessibility paths instead of local ones), which leads us to the following definitions. 

Fix a subset $\mathscr{S}\subset M$. 
For any $\sigma>0$, we let $\tilde \Gamma_f^0(\mathscr{S},\sigma)$ be the set of all points $x \in \mathscr{S}$ whose center accessibility class is \emph{locally trivial} %, i.e., such that there exists
in the following sense: for any $4$ us-loop   $\gamma = [x,x_1,x_2,x_3,x_{4}]$ at $(f,x)$ such that $\ell(\gamma)<10^{-2}\sigma$, we have $x_{4}=x$. We also set $\tilde \Gamma_f^0(\mathscr{S}):=\cup_{\sigma\in (0,1)}\tilde \Gamma_f^0(\mathscr{S},\sigma)$. When $\mathscr{S}=M$, we abbreviate $\tilde \Gamma_f^0(\mathscr{S},\sigma),\tilde \Gamma_f^0(\mathscr{S})$ respectively as $\tilde \Gamma_f^0(\sigma),\tilde \Gamma_f^0$. 

The next lemma explains how to construct \emph{closed} us-loops at points $x$ whose center accessibility class is (locally) one-dimensional; it will be useful later  to show that after a $C^r$-small perturbation, the accessibility class of $x$ can be made open. 
\begin{lemma}\label{eleme lemma}
	There exist $C^2$-uniform constants $\sigma_0=\sigma_0(f)>0$, $K_0=K_0(f)>0$  such that for any $\sigma\in (0,\sigma_0)$,  for any  point $x_0 \in \Gamma_f^1 \setminus \tilde \Gamma_f^0(\sigma)$,  if $\phi=\phi_{x_0}$ is the chart given by Lemma \ref{lemma constr phi}, then for any point $x \in \Gamma_f^1\cap \phi(B(0_{\R^d},\frac{K_0}{10} \sigma))$, there exists a non-degenerate  closed us-loop $\gamma_x=[x,x_1,\dots,x_{9},x]$ at $(f,x)$ such that 
	\begin{enumerate}
		\item $\ell(\gamma_x)<\sigma$;
		\item\label{deuxieme prop} $B(z_1,K_0 \sigma)\cap \{z,z_2,\dots,z_{9}\}=\emptyset$, where $z=\phi^{-1}(x)$, and $z_i=\phi^{-1}(x_i)$, for each integer $i=2,\dots,9$;
		\item   the map $\Gamma_f\cap \phi(B(0_{\R^d},\frac{K_0}{10} \sigma))\ni x \mapsto \gamma_x$ is continuous. 
	\end{enumerate} 
	%	Moreover, 
	%	\begin{itemize}
	%\item 
	%	we can take $n=2$ when  $x \in  \Gamma_f^0(\mathcal{V})$.
	%		\item 
	%	\end{itemize}
\end{lemma}

\begin{proof}
	%$(a)$ Fix $x\in M$. Assume first that  the center accessibility class $\cac_{f}(x)$ is trivial, i.e., $x \in \Gamma_f^0$.  Fix some small $\sigma'\in (\frac{\sigma}{20},\frac{\sigma}{10})$. We choose $x_1\in \cWu_f(x,\sigma')\setminus \cWu_f(x,\frac{\sigma'}{2})$, $x_2\in \cWs_f(x_1,\sigma')\setminus \cWs_f(x_1,\frac{\sigma'}{2})$, and we set $\{x_3\}:=\cWu_{f,\mathrm{loc}}(x_2)\cap \cWcs_{f,\mathrm{loc}}(x)$. As  $\cac_{f}(x)$ is trivial, we have  $\cWs_{f,\mathrm{loc}}(x_3)\cap \cWc_{f,\mathrm{loc}}(x)=\{x\}$, provided that $\sigma'$ is sufficiently small.  Let us consider the closed $4$ us-loop $\gamma=[x,x_1,x_2,x_3,x]$. By construction, $\ell(\gamma)<\sigma$, provided that $\sigma'$ is sufficiently small, and  $\gamma$ is non-degenerate; more precisely,  $B(x_1,\frac{\sigma}{1000})\cap \{x,x_2,x_{3}\}=\emptyset$. 
	
	Fix some small $\sigma>0$, let $x_0 \in \Gamma_f^1\setminus \tilde \Gamma_f^0(\sigma)$, and let $\sigma'\in (\frac{\sigma}{30},\frac{\sigma}{20})$. By definition, there exists a non-degenerate $4$ us-loop $\gamma=[x_0,x_1,x_2,x_3,x_4]$ such that
	\begin{itemize}
		\item $x_1\in \cWu_f(x_0)$, with $\frac{\sigma'}{2}<d_{ \cWu_f}(x_0,x_1)<\sigma'$;
		\item $x_2\in \cWs_f(x_1)$, with $\frac{\sigma'}{2}<d_{ \cWs_f}(x_1,x_2)<\sigma'$;
		\item $x_3:=H_{f,x_2,x_0}^u(x_2)\in \cWu_{f}(x_2,\sigma')\cap \cWcs_{f}(x_0,\sigma')$;
		\item $x_4:=H_{f,x_3,x_0}^s(x_3)\in \cWs_{f}(x_3,\sigma')\cap \cWc_{f}(x_0,\sigma')$, with $x_4\in \cac_f(x)\setminus \{x_0\}$. 
	\end{itemize}   
	\begin{figure}[H]
		\begin{center}
			\includegraphics [width=13.5cm]{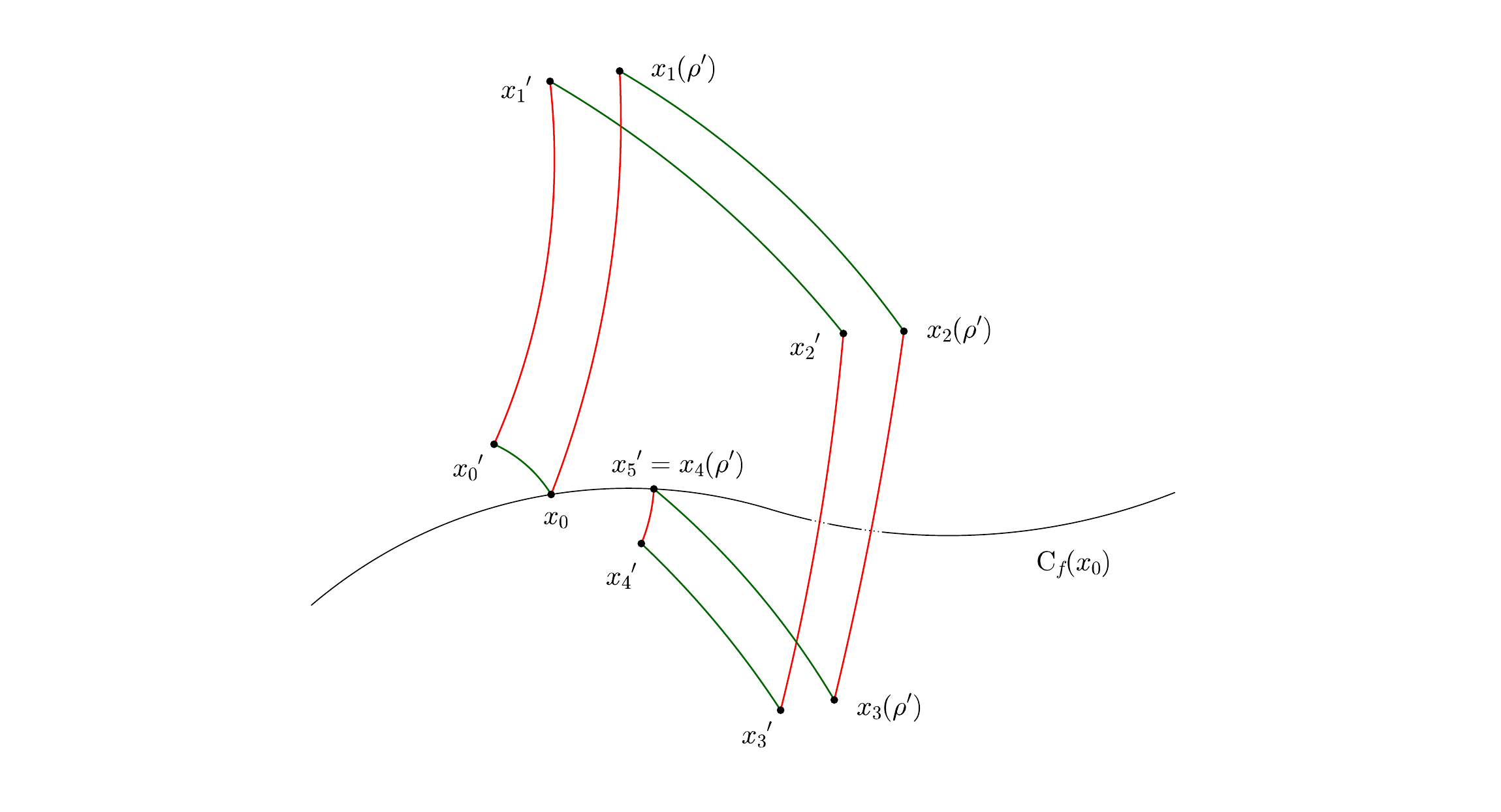}
			\caption{Construction of a non-degenerate closed us-loop.}
		\end{center}
	\end{figure}
	
	%	$\bullet\ \mathbf{Case\ 1:}$  $x_4=x$. Then $\gamma=[x,x_1,x_2,x_3,x]$ is a  closed us-loop. Moreover, $\ell(\gamma)<\sigma$ provided that $\sigma'$ is sufficiently small, and $B(x_1,\frac{\sigma}{1000})\cap \{x,x_2,\dots,x_{2n-1}\}=\emptyset$. 
	
	%	$\bullet\ \mathbf{Case\ 2:}$
	%	otherwise, we have $ x_4\in \cac_f(x)\setminus \{x\}$. 
	As $\cac_f(x_0)$ is one-dimensional, for the chart $\phi=\phi_{x_0}$ given by Lemma \ref{lemma constr phi}, we can assume that $\phi^{-1}(\cac_f(x_0,\sigma))=(-\rho_1,\rho_2)\times \{0_{\R^{d-1}}\}\simeq (-\rho_1,\rho_2)$, with $\rho_1,\rho_2>0$, $x_0\simeq 0$, and $x_4\simeq \rho\in (0,\rho_2)$. By varying the size of the legs, we can construct a   continuous family $\{\gamma(t)=[x_0,x_1(t),x_2(t),x_3(t),x_4(t)]\}_{t\in [\frac \rho 2,\rho]}$ of non-degenerate $4$ us-loops at $(f,x_0)$ such that $x_4(t)\simeq t \in [ \frac \rho 2,\rho]$.  %for some open interval $[\frac{\rho}{2},\rho]\subset I \subset  (0,\rho_2)$. 
	
	Let us take $x_0'\in \cWs_f(x_0,\frac{\sigma'}{10})\setminus \cWs_f(x_0,\frac{\sigma'}{20})$ and $t_0 \in [\frac \rho 2,\rho)$ close to $\rho$. 
	%, and let $x_1'' \in B(x_1(\rho'),\sigma'')$ for some $\rho' \in I$. We  define successively
	%	\begin{align*}
	%	\{x_1'\}&:=\cWu_{f,\mathrm{loc}}(x_0')\cap \cWcs_{f,\mathrm{loc}}(x_1''),\\
	%	\{x_2'\}&:=\cWs_{f,\mathrm{loc}}(x_1')\cap \cWcu_{f,\mathrm{loc}}(x_2(\rho')),\\
	%	\{x_3'\}&:=\cWu_{f,\mathrm{loc}}(x_2')\cap \cWcs_{f,\mathrm{loc}}(x_3(\rho')),\\
	%	\{x_4'\}&:=\cWs_{f,\mathrm{loc}}(x_3')\cap \cWcu_{f,\mathrm{loc}}(x_4(\rho')). 
	%	\end{align*}
	As in Lemma \ref{lemma continuation}, we let $\gamma^{x_0',f}(t_0)=[x_0',x_1',x_2',x_3',x_4']$ be the natural continuation of $\gamma(t_0)$ starting at $x_0'$ in place of $x_0$.  
	Since $\cWcu_{f}(x_4')=\cWcu_{f}(x_4)=\cWcu_{f}(x_0)$, we can also define 
	$
	\{x_5'\}:=H_{f,x_4',x_0}^u(x_4')\in\cWu_{f,\mathrm{loc}}(x_4')\cap \cWc_{f,\mathrm{loc}}(x_0)
	$, and we set $\gamma':=[x_0,x_0',x_1',\dots,x_5']$.
	In particular, $x_5'\in \cac_f(x,\sigma)$, and % $x_5'$ %can be made arbitrarily 
	%is close to $x_4(\rho')$ since $\sigma''$ is small, hence 
	$x_5'\simeq \rho'$ for some $\rho' \in (0,\rho)$.  As $x_5'=x_4(\rho')$, we can concatenate the $4$ us-loop  $\gamma(\rho')$ at $(f,x)$ with the $6$ us-loop $\overline{\gamma'}$ at $(f,x_5')$ to produce a closed $10$ us-loop $\gamma_{x_0}:=\overline{\gamma'}\gamma(\rho')$ at $(f,x_0)$. By construction, $\gamma_{x_0}$ is non-degenerate, and we have $\ell(\gamma_{x_0})< \sigma$. %Moreover, by a suitable choice of $x_1''$ above, we can ensure that  $B(x_1,\frac{\sigma}{1000})\cap \{x,x_2,\dots,x_{2n-1}\}=\emptyset$.
	
	Let us check that $d(x_1',x_1(\rho'))>\frac{\sigma}{800}$ provided that $\sigma$ is taken sufficiently small. By definition, we have $d_{\cWs_f}(x_0,x_0')\in [\frac{\sigma}{600},\frac{\sigma}{200}]$. Since we work in a $\sigma$-neighbourhood of $x_0$, 	and as the map $z \mapsto E_f^u(z)$ is Hölder continuous (see \cite{PSW1}), we deduce that the distance between the unstable bundles at any two points  $z_1\in \cW_f^u(x_0',\sigma)$, $z_2\in \cW_f^u(x_0,\sigma)$ is at most $\tilde c_1\sigma^{\theta}$, for two $C^2$-uniform constants $\theta=\theta(f)>0$,  $\tilde c_1=\tilde c_1(f)>0$. 
	Integrating the discrepancy along the unstable arcs from $x_0'$ to $x_1'$ and from $x_0$ to $x_1(\rho')$ yields 
	$$
	d(x_1',x_1(\rho'))\geq d(x_0',x_0)- \tilde c_2 \sigma^\theta \times \sigma \geq \frac{\sigma}{600} - \tilde c_2 \sigma^{1+\theta},
	$$
	for some constant $\tilde c_2>0$. We conclude that $d(x_1',x_1(\rho'))\geq \frac{\sigma}{800}$ 
	provided that $\sigma$ is chosen sufficiently small, i.e., $\sigma \in (0,\sigma_0)$,  for some $C^2$-uniform constant $\sigma_0=\sigma_0(f)>0$.  
	Moreover, by construction, $B(x_1,\frac{\sigma}{100})\cap \{x_0,x_2,x_3,x_4\}=\emptyset$. Similarly, we have  $B(x_1',\frac{\sigma}{200})\cap \{x_0,x_0',x_2',x_3',x_4',x_5'\}=\emptyset$ and $B(x_1(\rho'),\frac{\sigma}{200})\cap \{x_0,x_2(\rho'),x_3(\rho'),x_4(\rho')\}=\emptyset$. 
	
	Let us now explain how this construction can be performed for points $x$ near $x_0$ whose center accessibility class is also one-dimensional. By Lemma \ref{lemma continuation}, for any point $x\in M$ which is sufficiently close to $x_0$, and for any $t \in [\frac \rho 2,\rho]$, the us-loop $\gamma(t)$ admits a natural continuation $(\gamma(t))^{x,f}=:\check\gamma^{x}(t)$ that is a $4$ us-loop at $(f,x)$. Moreover, the map $t \mapsto \check\gamma^{x}(t)$ is continuous.  %; indeed, we let 
	%\begin{align*}
	%	\{y_0'\}&:= \cWs_{f,\mathrm{loc}}(y)\cap \cWcu_{f,\mathrm{loc}}(x_0'),\\
	%	\{y_1'\}&:=\cWu_{f,\mathrm{loc}}(y_0')\cap \cWcs_{f,\mathrm{loc}}(x_1(\rho'))\dots
	%	\end{align*}
	Similarly, the $6$ su-loop $\gamma'$ has a natural continuation  $(\gamma')^{x,f}=[x,(x_0')^{x,f},\dots,(x_5')^{x,f}]$. The point $(x_5')^{x,f}$  depends continuously on $x$, hence we can choose a continuous map $\rho'(\cdot)$ such that $\rho'(x_0)=\rho'$ and such that the endpoint of $\check\gamma^{x}(\rho'(x))$ coincides with the endpoint $(x_5')^{x,f}$ of $(\gamma')^{x,f}$. In particular, the continuations $(\gamma')^{x,f}$, $\check\gamma^{x}(\rho'(x))$ depend continuously on $x$. We conclude that the closed $10$ us-loop $\gamma_{x}:=\overline{(\gamma')^{x,f}}\check\gamma^{x}(\rho'(x))$ at $(f,x)$ depends continuously on the point $x$ in a small neighbourhood of $x_0$.  In particular, for $x$ sufficiently close to $x_0$, we have $\ell(\gamma_x)<\sigma$. 
\end{proof}

Actually, given a small center disk $\mathcal{D}$, we will need to construct \emph{closed} us-loops at points $x\in \mathcal{D}$ whose center accessibility class is not open, i.e., either zero or one-dimensional. %Since we will only consider local us-loops, in the same spirit as above, this motivates the following definitions. For any $\sigma>0$, we let $\tilde\Gamma_f(\sigma):=\tilde\Gamma_f^0(\sigma) \cup \Gamma_f^1$. 
Let us introduce some notation. Fix some small $\sigma>0$. For any $x \in \Gamma_f=\Gamma_f^0 \cup \Gamma_f^1$, we let %\tilde\Gamma_f(\sigma):=\tilde\Gamma_f^0(\sigma) \cup \Gamma_f^1$, 
\begin{itemize}
	\item let $\overline{\Gamma}_f(x):=\tilde \Gamma_f^0(\sigma)$, and $n(x):=2$,  if $x \in \tilde \Gamma_f^0(\sigma)$;
	\item otherwise, let $\overline{\Gamma}_f(x):=\Gamma_f^1 \setminus  \tilde\Gamma_f^0(\sigma) $, and $n(x):=5$, if $x\in \Gamma_f^1 \setminus  \tilde\Gamma_f^0(\sigma)$. 
\end{itemize}

\begin{lemma}\label{lem construction families of loops}
	There exist $C^2$-uniform constant $\tilde{K}=\tilde{K}(f)\in (0,1)$, $\tilde{\sigma}=\tilde{\sigma}(f)>0$ such that for any %non-periodic point $x_0 \in \Gamma_f$, for any 
	$R_0>0$, for any integer $k_0\geq 1$,\footnote{We will apply this lemma with $k_0=1$ or $2$ in the following.}  %,  $\tilde{\sigma}=\tilde{\sigma}(f,x_0,R_0)\in (0,\tilde h)$ 
	%such that 
	for any $\sigma\in (0,\tilde{\sigma})$, and for any point $x_0 \in \Gamma_f$ %satisfying $R_{\pm}(f,B(x_0,10\sigma))>R_0$,\footnote{\label{footnote recall }Recall Definition \ref{support of deformation} and Definition \ref{def rec funtion}.}  
	there exists a continuous map $\overline{\Gamma}_f(x_0)\cap \cWc_f(x_0,\tilde{K}\sigma)\ni x \mapsto \gamma^x$ such that $\gamma^x=\{\gamma^x(t) = [x,x_1^x(t),\dots,x_{2n}^x(t)]\}_{t \in [0,1]}$ is a continuous family of $2n$ us-loops at $(f,x)$, with $n:=n(x_0)\in \{2,5\}$,  $\ell(\gamma^x)< \sigma$,  such that $\gamma^x(0)$ is trivial, and for any integer $k \in \{1,\dots,k_0\}$,  $\gamma^x(\frac{k}{k_0})$ is a non-degenerate closed us-loop.  
\end{lemma}  

\begin{proof} 
	Let $R_0>0$, and let $k_0\geq 1$ be  some integer. Let $\sigma_0=\sigma_0(f)>0$,   $K_0=K_0(f)>0$ be as in Lemma \ref{eleme lemma}, %set $\tilde \sigma=\tilde \sigma(f):=\min(\overline{h},\sigma_0)>0$, 
	and take some small $\sigma \in (0,\min(\overline{h},\sigma_0))$. 
	
	We consider a point $x_0 \in \Gamma_f$ %such that $R_{\pm}(f,B(x_0,10\sigma))>R_0$, a
	and set $n:=n(x_0)\in \{2,5\}$. Let $\overline{h}=\overline{h}(f)>0$ and $\phi =  \phi_{x_0} \colon (-\overline{h},\overline{h})^{d} \to M$ be given by Lemma \ref{lemma constr phi}. 
	We distinguish between two cases.\\
	
$(1)$ If $x_0 \in \tilde \Gamma_f^0(\sigma)$, then there exists a non-degenerate closed $2n$ us-loop $\tilde \gamma=[x_0,x_1,x_2,x_3,x_0]$ at $(f,x_0)$ with $n=2$, $\ell(\tilde\gamma)<\frac \sigma 2$ and   $B(z_1,K_0\sigma)\cap \{0_{\R^d},z_2,z_3\}=\emptyset$,  where $z_i:=\phi^{-1}(x_i)$, for $i=1,2,3$.  By decreasing continuously the size of the legs of $\tilde\gamma$, we obtain a family of $2n$ us-loops $\{\gamma(t) = [x_0,x_1(t),x_2(t),x_3(t),x_0]\}_{t \in [0,1]}$  at $(f,x_0)$ such that   $\gamma(0)$ is trivial and $\gamma(1)=\tilde \gamma$. Moreover, by choosing the map $t \mapsto \gamma(t)$ carefully, we can ensure %ithout loss of generality, we can assume 
		that  for any  $k\in \{1,\dots,k_0\}$, it holds $B(z_1(\frac{k}{k_0}),\frac{K_0}{2}\sigma)\cap \{0_{\R^d},z_2(\frac{k}{k_0}),z_3(\frac{k}{k_0})\}=\emptyset$,  where  $z_i(\frac{k}{k_0}):=\phi^{-1}\big((x_i)(\frac{k}{k_0})\big)$, for $i=1,2,3$, and $d\big(z_1(\frac{k}{k_0}),z_1(\frac{k'}{k_0})\big)\geq \frac{K_0}{2k_0}\sigma$,\footnote{For instance, we choose the map $t \mapsto x_1(t)\in \mathcal{W}_{f,\mathrm{loc}}^u(x_0)$ in such a way that $d(z_1(t),z_1(t'))= d(0_{\R^d},z_1)\cdot |t-t'|$, for all $t,t'\in [0,1]$.} for all $k'\in \{1,\dots,k_0\}\setminus \{k\}$.
		
		%Let $\tilde K >0$ be a small constant to be determined.  
		For any point $x \in \tilde \Gamma_f^0(\sigma)\cap \cWc_{f,\mathrm{loc}}(x_0)$ with $d(0_{\R^d},\phi^{-1}(x)) \leq \frac{K_0}{10}\sigma$, and for  $t \in [0,1]$, let $\gamma^x(t)=[x,x_1^x(t),x_2^x(t),x_3^x(t),x]$ be the closed $2n$ us-loop whose corners are:
		\begin{itemize}
			\item $x_1^x(t):=H_{f,x,x_1(t)}^u(x)\in\cW_{f,\mathrm{loc}}^{u}(x)\cap \cW_{f,\mathrm{loc}}^{cs}(x_1(t))$;
			\item $x_2^x(t):=H_{f,x_1^x(t),x_2(t)}^s(x_1^x(t))\in\cW_{f,\mathrm{loc}}^{s}(x_1^x(t))\cap \cW_{f,\mathrm{loc}}^{cu}(x_2(t))$;
			\item $x_3^x(t):=H_{f,x_2^x(t),x}^u(x_2^x(t))\in\cW_{f,\mathrm{loc}}^{u}(x_2^x(t))\cap \cW_{f,\mathrm{loc}}^{cs}(x)$.
		\end{itemize}
		We let $\gamma^x$ be the continuous family $\gamma^x:=\{\gamma^{x}(t)\}_{t \in [0,1]}$. If $\sigma$ is sufficiently small, then $\ell(\gamma^x)<\sigma$, and for any   $k\in \{1,\dots,k_0\}$, $\gamma^x(\frac{k}{k_0})$ is a non-degenerate closed us-loop at $(f,x)$. Let $z_0^x:=\phi^{-1}(x)$, and  $z_i^x(\frac{k}{k_0}):=\phi^{-1}\big((x_i^x)(\frac{k}{k_0})\big)$, for $i=1,2,3$.  % If $\tilde K >0$ is sufficiently small, then, a
		Arguing as in the proof of Lemma \ref{eleme lemma},  we  have $B\big(z_1^x(\frac{k}{k_0}),\frac{K_0}{5}\sigma\big)\cap \{z_0^x,z_2^x(\frac{k}{k_0}),z_3^x(\frac{k}{k_0})\}=\emptyset$, and $d\big(z_1^x(\frac{k}{k_0}),z_1^x(\frac{k'}{k_0})\big)\geq \frac{K_0}{5 k_0}\sigma$, for all $k'\in \{1,\dots,k_0\}\setminus \{k\}$, provided that $\sigma$ is sufficiently small. \\
		
$(2)$ Otherwise, we have $x_0 \in \Gamma_f^1 \setminus  \tilde \Gamma_f^0(\sigma)$.  By Lemma \ref{eleme lemma}, after possibly taking $K_0$ smaller, then for any point  $x \in \Gamma_f\cap \cWc_{f,\mathrm{loc}}(x_0)$ such that $d(0_{\R^d},\phi^{-1}(x)) \leq \frac{K_0}{10}\sigma$, there exists a  non-degenerate  closed $2n$ us-loop $\gamma_x=[x,x_1,\dots,x_{2n-1},x]$ at $(f,x)$ with $n=5$,  $\ell(\gamma_x)<\frac \sigma 2$, such that the map $\Gamma_f\cap \phi(B(0_{\R^d},\frac{K_0}{10} \sigma)) \ni x \mapsto \gamma_x$ is continuous, and such that $B(z_1^x,K_0\sigma)\cap \{z_0^x,z_2^x,\dots,z_{2n-1}^x\}=\emptyset$, where $z_0^x:=\phi^{-1}(x)$, and $z_i^x:=\phi^{-1}\big(x_i^x\big)$, for each integer $i=1,\dots,2n-1$.

		By decreasing continuously the size of the legs of $\gamma_x$, keeping $x_{2n-1}(t)\in \cWcs_f(x)$ and letting $x_{2n}(t) :=H_{f,x_{2n-1}(t),x}^s(x_{2n-1}(t))$, we obtain a continuous family $\gamma^x=\{\gamma^x(t) = [x,x_1^x(t),\dots,x_{2n}^x(t)]\}_{t \in [0,1]}$ of $2n$ us-loops at $(f,x)$ such that $\gamma^x(0)$ is trivial, $\gamma^x(1)=\gamma_x$, and $\ell(\gamma^x)<\sigma$.  %	If $\sigma$ is sufficiently small, then $\ell(\gamma^x)<\sigma$. 
		
		Moreover, by choosing carefully  the map $t \mapsto\gamma^x (t)$,  we can ensure that for any integer $k\in \{1,\dots,k_0\}$, $\gamma^x(\frac{k}{k_0})$ is a non-degenerate closed us-loop at $(f,x)$. Indeed, as in the proof of  Lemma \ref{eleme lemma}, we consider a one-parameter family $(\check\gamma^{x}(t))_{t \in [0,1]}$ of $4$ us-loops at $(f,x)$ such that $\check\gamma^{x}(0)$ is the trivial loop and such that the first corners of $\check\gamma^{x}(t)$ and $\check\gamma^{x}(t')$ are distinct for $t\neq t' \in [0,1]$. We can also perform the same  construction as in Lemma \ref{eleme lemma} in order to  obtain a closed $10$-us loop $\gamma^x(t)$ at the times $t=1,\frac{k_0-1}{k_0},\frac{k_0-2}{k_0},\dots,\frac{1}{k_0}$, and such that $B(z_1^x(\frac{k}{k_0}),\frac{K_0}{5}\sigma)\cap \{z_0^x,z_2^x(\frac{k}{k_0}),\dots,z_{2n-1}^x(\frac{k}{k_0})\}=\emptyset$, where we let   $z_i^x(\frac{k}{k_0}):=\phi^{-1}\big((x_i^x)(\frac{k}{k_0})\big)$, for  $i=1,\dots,2n-1$,   and such that $d\big(z_1^x(\frac{k}{k_0}),z_1^x(\frac{k'}{k_0})\big)\geq \frac{K_0}{5 k_0}\sigma$, for all $k'\in \{1,\dots,k_0\}\setminus \{k\}$. 
\end{proof}

 We will also need to construct certain us/su-paths for all points in a small center disk. 
 Take $f\in \mathscr{F}$ and let $\sigma>0$ be small. We assume that for some point $x_0\in M$, and some constant $K>0$, it holds $x \notin\tilde\Gamma_{f}^0(\sigma)$, for all $x \in \cWc_f(x_0,K\sigma)$. Fix $\theta>0$ small. By Proposition \ref{prop lamination} and Proposition \ref{lemma var des classes}, there exists a $C^1$ neighbourhood $\cU$ of $f$ such that for any $g \in \cU^\mathscr{F}$ and for any $x \in \Gamma_{g}^1\cap\cWc_f(x_0,K\sigma)$, it holds
 \begin{align}\label{var clases duex}
 	\Pi_x^c\cac_{g}(x,10\sigma)\subset \mathscr{C}_1,
 \end{align} 
 where $\mathscr{C}_{1}\subset \R^2$ is the cone of angle $\theta$ centered at $0_{\R^2}$, and $\Pi_x^c\colon M \to \R^2$ is the map in Lemma \ref{lemma constr phi} for $f$. 
 In the following, we let $\mathscr{C}:=\big(\R^2  \setminus \mathscr{C}_1\big) \cup \{0_{\R^2}\}$, we denote by $\mathscr{C}_*^+,\mathscr{C}_*^-$ the two components of the set $\mathscr{C}\setminus \{0_{\R^2}\}$, and let $\mathscr{C}^+:=\mathscr{C}_*^+ \cup \{0_{\R^2}\}$, $\mathscr{C}^-:=\mathscr{C}_*^- \cup \{0_{\R^2}\}$. Assume that $\mathscr{C}^+$, resp. $\mathscr{C}^-$ is the top, resp. bottom component in Figure \ref{fig constr loops}. 
 \begin{lemma}\label{lemme de construction des loops}
 	Take $f$, $x_0$, $\sigma$, $\theta$, $\mathcal{U}$ as above, and let $\mathscr{C}$, $\mathscr{C}^+$, and $\mathscr{C}^-$ as defined above. After possibly taking $K$ smaller, there exist continuous maps $\cWc_f(x_0,K\sigma)\ni x \mapsto \gamma_1^x,\gamma_2^x$ such that for any  $x \in \cWc_f(x_0,K\sigma)$, $\gamma_1^x=[x,\alpha_{1}^x,\dots,\omega_{1}^x]$, resp. $\gamma_2^x=[x,\alpha_{2}^x,\dots,\omega_{2}^x]$, is a non-degenerate closed $10$ us-loop, resp. $10$ su-loop at $(f,x)$ such that $\ell(\gamma_1^x),\ell(\gamma_2^x)< \sigma$, such that the endpoints $\omega_1^x=H_{\gamma^x_1}(x)$, $\omega_2^x=H_{\gamma^x_2}(x)$ satisfy 
 	%	$$
 	%	\Pi_x^c\omega_\star^x\in \mathscr{C}.
 	%	$$
 	%	More precisely, it holds 
 	\begin{equation*}%\label{good cone plus moins}
 		\big(\Pi_x^c\omega_1^x,\Pi_x^c\omega_2^x\big)\in \big(\mathscr{C}^+ \times \mathscr{C}^-\big)\cup \big(\mathscr{C}^- \times \mathscr{C}^+\big),
 	\end{equation*}
 	and such that for $\star=1,2$, for some $C^2$-uniform constant $\hat K_\star>0$, we have
 	\begin{equation}\label{distance etre points}
 	B(\alpha_\star^x,\hat K_\star\sigma)\cap \{z\}=\emptyset,\text{ for any corner }z \neq \alpha_\star^x\text{ of }\gamma_\star^x.
 	\end{equation}
 \end{lemma}
 
 \begin{proof}
 	As $x_0 \in \Gamma_{f}^1\setminus \tilde\Gamma_{f}^0(\sigma)$, there exists a non-degenerate $4$ us-loop $\gamma=[x_0,x_1,x_2,x_3,x_4]$ with  $\ell(\gamma)<\sigma$ and $x_4\in \cac_{f}(x_0)\setminus \{x_0\}$. By shrinking the size of the legs, we construct a continuous family $\{\gamma(t)=[x_0,x_1(t),x_2(t),x_3(t),x_4(t)]\}_{t\in [0,1]}$ of non-degenerate $4$ us-loops at $(f,x_0)$ such that $\gamma(0)$ is trivial and $\gamma(1)=\gamma$.  
 	%We can easily extend this family to a continuous family $\{\gamma(t)=[x_0,x_1(t),x_2(t),x_3(t),x_4(t)]\}_{t\in [-1,1]}$ of non-degenerate $4$ us-loops at $(f,x_0)$ such that $x_0$ lies in the interior of the segment connecting $x_4(-1)$ and $x_4(1)$ within the center accessibility class of $x_0$. In other terms, 
 	
 	Assuming that $K>0$ is sufficiently small, the family $\{\gamma(t)\}_{t\in [0,1]}$ extends to a continuous map $\cWc_f(x_0,K\sigma)\ni x\mapsto \gamma^x=\{\gamma^x(t)\}_{t \in [0,1]}$ such that for each $x \in \cWc_f(x_0,K\sigma)$, and for each $t\in [0,1]$, $\gamma^x(t)=[x,x_1^x(t),x_2^x(t),x_3^x(t),x_4^x(t)]$ is a $4$ us-loop at $(f,x)$, and $\gamma^x(0)$ is trivial.  %$\big\{\Pi_c\phi_x^{-1}\big(x_4^x\big(\frac{1}{2}\big)\big)\big\}\subset \mathscr{C}_0$ and 
 	Moreover, up to reparametrization, there exists $\vartheta>0$ such that for each $x \in \cWc_f(x_0,K\sigma)$, it holds
 	%	If $\sigma$ is sufficiently small, then for each $x \in \cWc_f(x_0,K\sigma)$, the endpoints $x_4^x(1)$, $x_4^x(-1)$ satisfy $\Pi_x^c\big(x_4^x(1)\big)\in \mathscr{C}_1$ and $\Pi_x^c\big(x_4^x(-1)\big)\in \mathscr{C}_1$. Actually, by compactness, there exists $\vartheta>0$ such that  for each $x \in \cWc_f(x_0,K\sigma)$, we have 
 	\begin{equation}\label{cond encadrement}
 	\left\{
 	\begin{array}{l}
 	\big\{\Pi_x^c\big(x_4^x(t)\big)\big\}_{t \in [\frac 14, \frac 13]}\subset B\Big(\Pi_x^c \Big(x_4\Big(\frac 13\Big)\Big),\frac 12\vartheta\Big)%\subset B\Big(\Pi_x^c \Big(x_4\Big(\frac 13\Big)\Big),2\vartheta\Big)
 	\subset B\big(0_{\R^2},3\vartheta\big),\\
 	\big\{\Pi_x^c\big(x_4^x(t)\big)\big\}_{t \in [\frac 12, \frac 23]}\subset B\Big(\Pi_x^c \Big(x_4\Big(\frac 23\Big)\Big),\frac 12\vartheta\Big)%\subset B\Big(\Pi_x^c \Big(x_4\Big(\frac 13\Big)\Big),2\vartheta\Big)
 	\subset \mathscr{C}_1^r\cap \Big(B\big(0_{\R^2},7\vartheta\big)\setminus B\big(0_{\R^2},4\vartheta\big)\Big),\\
 	\big\{\Pi_x^c\big(x_4^x(t)\big)\big\}_{t \in [\frac 34, 1]}\subset B\big(\Pi_x^c \big(x_4(1)\big),\frac 12\vartheta\big)\subset \mathscr{C}_1^r\cap \Big(B\big(0_{\R^2},10\vartheta\big)\setminus B\big(0_{\R^2},8\vartheta\big)\Big),
 	%&\text{with }B(x_4(1),3\vartheta)\subset \mathscr{C}_1^r. %\big\{\Pi_x^c\big(x_4^x(t)\big)\big\}_{t \in [-1,-1+\vartheta]}&\subset  B(x_4(-1),\vartheta), &\text{with }B(x_4(-1),3\vartheta)\subset \mathscr{C}_1^\ell.
 	\end{array}
 	\right. 
 	\end{equation} 
 	denoting by $\mathscr{C}_1^r$ the connected component of $\mathscr{C}_1\setminus \{0_{\R^2}\}$ containing $\Pi_{x_0}^c (x_4)$.  Moreover, after possibly changing the parametrization by $t$, we can also assume that for all $x \in \cWc_f(x_0,K\sigma)$, we have
 	\begin{equation}\label{borne inff}
 	d_{\mathcal{W}_{f}^u}(x_1^x(t),x)\geq \frac{1}{200}\sigma,\quad \text{for all }t\in \Big[\frac 14,1\Big].
 	\end{equation}
 	%, we have 
 	%$$
 	%\Pi_x^c\big(x_4(1)\big)\in\mathscr{C}_1^r,\qquad \Pi_x^c\big(x_4(-1)\big)\in\mathscr{C}_1^\ell. 
 	%$$
 	
 	%For every point $x \in \cWc_f(x_0,K\sigma)$ and every $t \in [0,1]$, we let $\gamma^x(t)=(\gamma(t))^{x,f_0}=[x,x_1^x(t),x_2^x(t),x_3^x(t),x_4^x(t)]$ be the continuation of $\gamma(t)$ starting at $x$ given by Lemma \ref{lemma continuation}. 
 	Now, as in Lemma \ref{eleme lemma}, we take $x_0'\in \cWs_{f}(x_0)$ such that $\frac{1}{200}\sigma \leq d_{\mathcal{W}_{f}^s}(x_0,x_0') \leq \frac{1}{100}\sigma  %\frac{1}{2}d_{\mathcal{W}_{f_1}^s}(x_0,y_2)
 	%\asymp \sigma
 	$. For any $t\in [0,1]$, let $\tilde\gamma(t)=[y_0(t),\dots,y_4(t)]$ be the natural continuation of $\gamma(t)$ starting at $y_0(t)=x_0'$ in place of $x_0$.  
 	As $\cWcu_{f}(y_4(t))=\cWcu_{f}(x_4)=\cWcu_{f}(x_0)$, we may also define 
 	$
 	\{y_5(t)\}:=\cWu_{f,\mathrm{loc}}(y_4(t))\cap \cWc_{f,\mathrm{loc}}(x_0)
 	$, and set $\gamma_*(t):=[x_0,y_0(t),\dots,y_4(t),y_5(t)]$. In the same way, for each point $x \in \cWc_f(x_0,K\sigma)$, we let $\gamma_*^x(t)=[x,y_0^x(t),\dots,y_5^x(t)]$ be the continuation of $\gamma_*(t)$ starting at $x$ given by Lemma \ref{lemma continuation}. 
 	%$$
 	%\Pi_x^c(y_5^x(1))\in B(x_4(1),2\vartheta) \subset \mathscr{C}_1^r,\quad \Pi_x^c(y_5^x(-1))\in B(x_4(-1),2\vartheta) \subset \mathscr{C}_1^\ell,
 	%$$ 
 	%for all $x \in \cWc_f(x_0,K\sigma)$. 
 	
 	%By the fact that $x_0 \in\Gamma_{f_0}^1\setminus \tilde\Gamma_{f_0}^0(\sigma)$, for each $t \in [0,1]$, we have
 	%$$
 	%\Pi_c\phi_x^{-1}(x_4(t))\in
 	%\Pi_c\phi_x^{-1} \cac_{f_0}(x_0)\subset \mathscr{C}_0.
 	%$$
 	For each $(x,t) \in \cWc_f(x_0,K\sigma)\times [0,1]$, % we denote by $\mathcal{M}_x\subset \cWc_f(x_0)$ the set of points $\{x_4^x(t)\}_{t\in [-1,1]}$. We distinguish between two cases: 
 	%\begin{enumerate}
 	%	\item either $y_5^x(1) \in \mathcal{M}_x$; %for any $t'$ in an open interval $(t_0-\varepsilon,t_0+\varepsilon)\subset [1-\vartheta, 1]$, $\varepsilon>0$; 
 	%		in particular, $y_5^x(1)=x_4^x(t^x)$ for some $t^x \in [0, 1]$ close to $1$ (assuming again that $y$ is chosen sufficiently close to $x$); in this case, we write $x \in\tilde\Gamma_{f}^1(\sigma)$. Then, in the same way as in Lemma \ref{eleme lemma}, we let $\gamma^{x,1}:=\gamma^x(t^x)$, $\gamma^{x,2}:=\overline{\gamma_*^x(1)}$, and we obtain a closed $10$ us-loop at $(f,x)$ by considering the concatenation $\gamma_1^x(1):=\gamma^{x,2} \gamma^{x,1}$;
 	%	\item\label{case deux} 
 	%otherwise, $y_5^x(1) \notin \mathcal{M}_x$; in this case, we write $x \in\tilde\Gamma_{f}^{2}(\sigma)$. For any 
 	%and for any $t\in [0,1]$, 
 	we denote by $\check\gamma^{x}(t)$ the continuation of $\overline{\gamma_*^x(\frac 23)}$ starting at $x_4^x(t)$ as in Lemma \ref{lemma continuation}, and by concatenation, we obtain the $10$ us-loop $\gamma_1^x(t):=\gamma^{x}(t)\check\gamma^{x}(t)=[x,\alpha_1^x(t),\dots,\omega_1^x(t)]$. %If the endpoint of $\gamma_2^x$ is below the line spanned by $\mathcal{M}_x$, 
 	If $\sigma$ is sufficiently small, $x_0'$ is very close to $x_0$, %we will have $y$ is chosen sufficiently close to $x$, for $s$ close to $1$, 
 	and by \eqref{cond encadrement}, for any $x \in \cWc_f(x_0,K\sigma)$, it holds 
 	$$
 	\Pi_x^c\Big(\omega_1^x\Big(\frac 14\Big)\Big)\notin \mathscr{C}_1^r,\qquad \Pi_x^c(\omega_1^x(1))\in \mathscr{C}_1^r. 
 	$$ 
 	%is in  $\mathscr{C}_1^r$, while for $s$ close to $-1$, the endpoint of $\hat\gamma^{x}(t)\gamma^{x}(t)$ satisfies  is in $\mathscr{C}_1^\ell$. 
 	Since the set $\{\omega_1^x(t)\}_{t \in [0,1]}$ of endpoints is connected, its image under $\Pi_x^c$ has to cross the cone $\mathscr{C}=\mathscr{C}^+ \cup \mathscr{C}^-$. We then let $t^x\in [0,1]$ be the smallest $t \in [0,1]$ such that $\Pi_x^c(\omega_1^x(t))\in \mathscr{C}_1^r$; we also denote by $\gamma_1^x=[x,\alpha_{1}^x,\dots,\omega_{1}^x]$ the $10$ us-loop $\gamma_1^x(t^x)=\gamma^{x}(t^x)\check\gamma^{x}(t^x)$, with $\alpha_1^x:=\alpha_1^x(t^x)$ and  $\omega_1^x:=\omega_1^x(t^x)$. In particular, we have $\Pi_x^c(\omega_1^x)\in \mathscr{C}$; without loss of generality, we assume that $\Pi_x^c(\omega_1^x)\in \mathscr{C}^+$.   %Otherwise, the endpoint of $\gamma_2^x$ lies above the line spanned by $\mathcal{M}_x$, and we let $\gamma_x$ be the $10$ us-loop $\gamma_x:=\gamma_x^1\gamma_x^2 $.
 	%	\end{enumerate}
 	
 	\begin{figure}[H]
 		\begin{center}
 			\includegraphics [width=14cm]{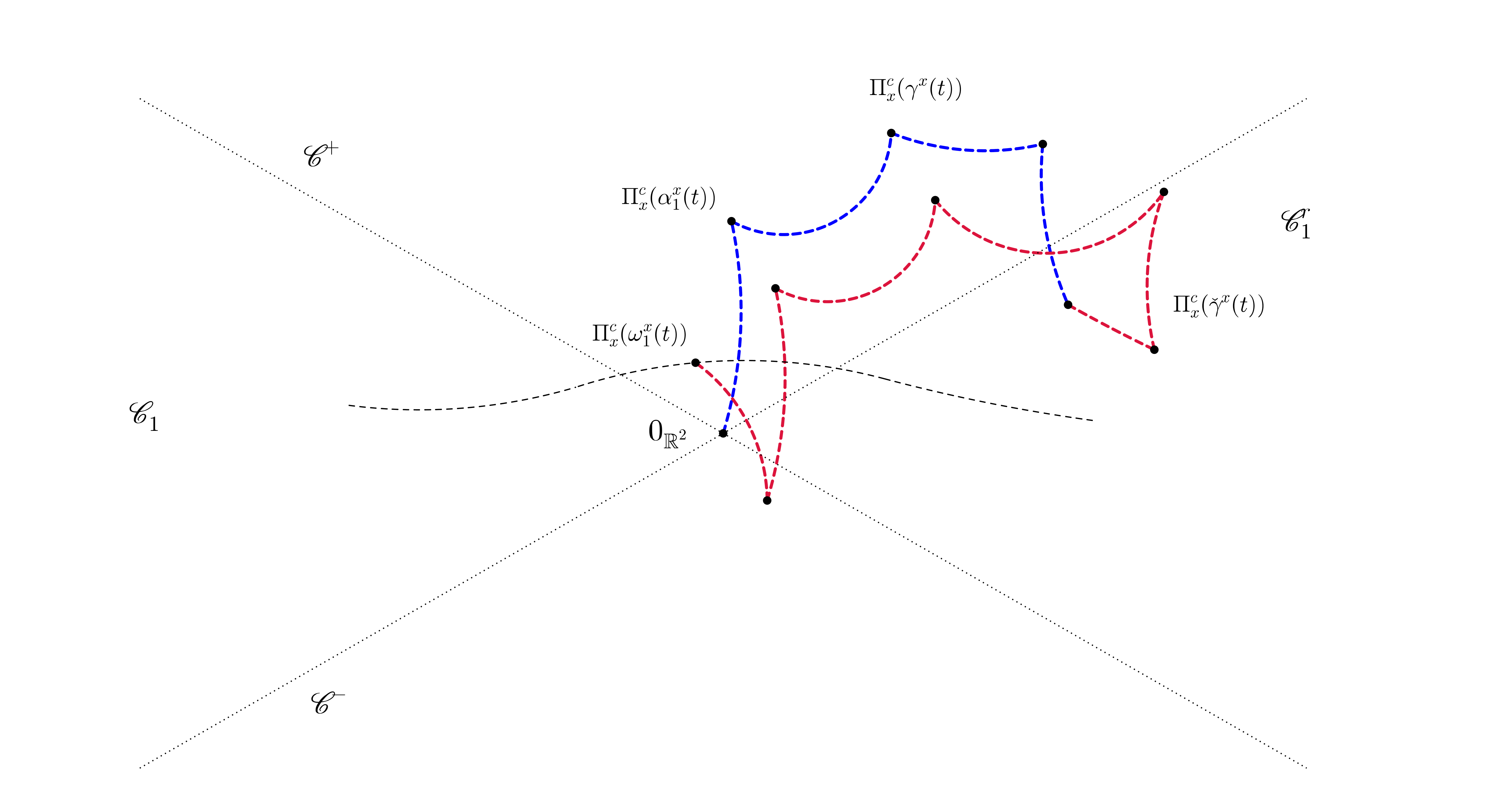}
 			\caption{Construction of the loop $\gamma_1^x$.}\label{fig constr loops}
 		\end{center}
 	\end{figure} 
 	
 	%In either case, we denote respectively by $\alpha_1^x$, $\omega_1^x$ the first and last corners of  $\gamma_1^x(1)=[x,w_{1}^x,\dots,w_{9}^x]$, i.e., $\alpha_1^x:=w_1^x$, $\omega_1^x:=w_9^x$. 
 	For each $s \in [0,1]$, we also denote by $\gamma_2^x(s)=[x,\alpha_2^x(s),\dots,\omega_2^x(s)]$ the $10$ su-loop obtained by taking the continuation of $\overline{\gamma_1^x(s)}$ starting at $x$ in place of $\omega_1^x(s)$. In this case, arguing as above, we see that for certain values $s\in [0,1]$, it holds $\Pi_x^c(\omega_2^x(s))\in \mathscr{C}^-$; we then let $s^x\in [0,1]$ be the largest $s \in [0,1]$ with that property, %such that $\Pi_x^c(\omega_2^x(s))\in \mathscr{C}_1^r$; then, for the
 	and we define the $10$ su-loop $\gamma_2^x:=\gamma_2^x(s^x)$, with $\gamma_2^x=[x,\alpha_{2}^x,\dots,\omega_{2}^x]$, and $\Pi_x^c(\omega_2^x)\in \mathscr{C}^-$.
 	
 	Besides, \eqref{distance etre points} follows from arguments similar to those in Lemma \ref{eleme lemma}, using \eqref{borne inff}, and since $x_0'$ was chosen such that $d_{\mathcal{W}_{f}^s}(x_0,x_0')\geq \frac{1}{200}\sigma $.   
 	%Actually, $\Pi_c \phi_x^{-1} \omega_\star^x=0_{\R^2}$ in the first case, that is, when $x \in\tilde\Gamma_{f_1}^1(\sigma)$. Moreover, in the second case (i.e., when $x \in\tilde\Gamma_{f_1}^2(\sigma)$), if $\Pi_c \phi_x^{-1} \omega_1^x\in \mathscr{C}^+$, then $\Pi_c \phi_x^{-1} \omega_2^x\in \mathscr{C}^-$; conversely, if $\Pi_c \phi_x^{-1} \omega_1^x\in \mathscr{C}^-$, then $\Pi_c \phi_x^{-1} \omega_2^x\in \mathscr{C}^+$, which concludes. 
 \end{proof}

\section{A submersion from the space of perturbations to the phase space}\label{section sumb}

As above, we consider a partially hyperbolic diffeomorphism $f \in \cPH^r(M)$, $r \geq 2$, with $\dim E_f^c \geq 2$ that is center bunched, dynamically coherent, and plaque expansive. In Subsection \ref{random pert}, we recall some general results from \cite{LZ} about random perturbations and the changes those perturbations induce on certain holonomy maps.  In Subsection \ref{construct }, we construct a family of perturbations and show how the results of the previous part can be applied to the particular setting we are interested in. 

\subsection{Random perturbations}\label{random pert}

As in \cite{LZ}, we will use the following suspension construction to show that certain holonomy maps are differentiable with respect to the perturbation parameter. The idea is to incorporate the perturbation parameter into a higher dimensional partially hyperbolic diffeomorphism, which, under some assumptions, is still dynamically coherent and center bunched. 

\begin{defi}[$C^r$ deformation] \label{smooth deform}
	Let $I\geq 1$ be some integer, and let  $\mathcal{U}$ be an open neighbourhood of $\{0\}$ in $\R^{I}$.
	A $C^{r}$ map $\hat{f} \colon \mathcal{U} \times M \to M$ satisfying $\hat{f}(0,\cdot) = f$ and $\hat f(b,\cdot)\in \mathcal{PH}^r(M)$ for all $b \in \mathcal{U}$ is called a \emph{$C^r$ deformation at $f$ with $I$-parameters}.
	We associate with  $\hat{f}$ the suspension map $T(\hat{f})$ defined by 
	\begin{equation}\label{def T}
	T= T(\hat{f}) \colon
	\mathcal{U} \times M \to \mathcal{U} \times M,\quad 
	(b,x) \mapsto (b,\hat{f}(b,x)), 
	\end{equation}
	and we denote $f_b:=\hat{f}(b,\cdot)$. 
	If in addition $f_b \in \mathcal{PH}^{r}(M,\mathrm{Vol})$ for all $ b \in \mathcal{U}$,  then   $\hat{f}$ is said to be \emph{volume preserving}. 
	%
	%In the following, \marginpar{I shortened a little}
	%$U$ will be taken sufficiently small so that various conditions are satisfied.  
	%Given some conditions associated to parameters $\{c_i\}_i$, the notation $U=U(\{c_i\}_i)$ means that $U$ can be chosen depending only on $\{c_i\}_i$. Given another set of parameters $\{c_j'\}_j$ associated to  more restrictive conditions, the notation  $U(\{c_j'\}_j)\subset U(\{c_i\}_i)$ means that $U(\{c_j'\}_j)$ is chosen as a subset of $U(\{c_i\}_i)$ depending on $\{c_j'\}_j$. 
\end{defi}
%
%
%The following lemma will be used to estimate the derivatives of holonomy maps in Proposition \ref{prop a priori est} when $f$ is a dynamically coherent, plaque expansive partially hyperbolic diffeomorphism. We defer its proof to Appendix A.

\begin{defi}[Infinitesimal $C^r$ deformation]\label{def infini deform}
	
	Let  $I\geq 1$ be an integer. A $C^r$ map $V \colon \R^{I} \times M \to TM$ is called an \emph{infinitesimal $C^{r}$ deformation with $I$-parameters} if 
	
	\begin{enumerate}
		%\item $V(0, \cdot)  \equiv 0$, 
		\item for each $B \in \R^{I}$, $V(B,\cdot)$ is a $C^{r}$ vector field on $M$;
		\item for each $x \in M$, $B \mapsto V(B,x)$ is a linear map from $\R^{I}$ to $T_{x}M$.
	\end{enumerate}

\end{defi}
	
	\begin{remark}
			Given $I \geq 1$, an infinitesimal $C^{r}$ deformation $V$ with $I$-parameters, and some small $\epsilon > 0$, we associate with $V$ a $C^{r}$ deformation at $f$ with $I$-parameters, denoted by $\hat{f}$,  which is defined by 
		\begin{equation*}
		\hat{f}(b, x):= \mathcal{F}_{V(b, \cdot)}(1,  f(x)), \quad \forall\, (b,x) \in \mathcal{U} \times M,
		\end{equation*} 
		where $\mathcal{U} = B(0, \epsilon) \subset \R^{I}$ and for any $B \in \mathbb{R}^I$, $\mathcal{F}_{V(B,\cdot)} \colon \R \times  M  \to M$ denotes the $C^r$ flow generated by the vector field $V(B,\cdot)$.
		In this case, we say that $\hat{f}$ is \emph{generated by} $V$. 
		If in addition  $V(B,\cdot)$ is divergence-free for each $B \in \R^{I}$,  then $\hat f$ is  volume preserving  as in Definition \ref{smooth deform}, and we say that $V$ is  \emph{volume preserving}. 
	\end{remark}

\begin{lemma}[Lemma 4.11 in \cite{LZ}]\label{lemma T deformation}
	\label{ph for T} %\footnote{ write it better. \cb I tried to\clb} 
	Let $I \geq  1$ be some integer,  let $\mathcal{U} \subset \R^I$ be an open neighbourhood  of $\{0\}$, and let $\hat{f} \colon \mathcal{U} \times M \to M$ be a $C^{r}$ deformation at $f$ with $I$-parameters. If $\mathcal{U}$ is chosen sufficiently small, then 
	the map $T=T(\hat f)$ is a $C^r$ dynamically coherent partially hyperbolic system for some $T$-invariant splitting 
	\begin{equation*}
	T_{(b,x)}(\mathcal{U} \times M)\simeq T_{b}\mathcal{U} \oplus T_{x}M = E^{s}_{T}(b,x) \oplus E^{c}_{T}(b,x) \oplus E^{u}_{T}(b,x),
	\end{equation*}
	for all $(b,x) \in \mathcal{U} \times M$. 
	Moreover, for any $(b,x) \in \mathcal{U} \times M$,   we have
	\begin{equation*}
	E^{*}_{T}(b,x) = \{0\} \oplus E^{*}_{f_b}(x),\qquad  \cW_T^{*}(b,x)=\{b\} \times \cW_{f_b}^{*}(x),\qquad \text{for } *=u,s,
	\end{equation*}
	and
	%\end{equation}
	\begin{equation}\label{prop subspace def}
	E^{c}_{T}(b,x) = Graph(\nu_b(x,\cdot)) \oplus E^{c}_{f_b}(x),
	\end{equation}
	for a unique linear map $\nu_b(x,\cdot) \colon T_b \mathcal{U} \to E^{su}_{f_b}(x):=E^{s}_{f_b}(x)\oplus E^{u}_{f_b}(x)$.
	
	If in addition $f$ is center bunched, then, after reducing the size of $\mathcal{U}$, 
	$u/s$-holonomy maps between local center leaves of $T$ (within distance $1$)  are $C^{1}$ %for some $C^1$-uniform constant $\epsilon>0$.
	when restricted to some $cu/cs$-leaf, with uniformly continuous, uniformly bounded derivatives. 
\end{lemma}

Let $I\geq 1$ be some integer, let $\mathcal{U}\subset \R^I$ be some small neighbourhood of $\{0\}$ in $\R^I$, let $\hat{f} \colon \mathcal{U} \times M \to M$ be a $C^1$ deformation at $f$ with $I$-parameters, and let $T = T(\hat{f})$.  

\begin{defi}[Lift of a us/su-loop]\label{defii liftt}
	For any point  $x\in M$, for any integer $n \geq 2$, and for any $2n$ us/su-loop $\gamma = [x,x_1, \dots,x_{2n}]$ at $(f,x)$,  we define the \emph{lift} of $\gamma$ as  
	$$
	\hat{\gamma}:= [(0,x),(0,x_1),\dots,(0,x_{2n})].
	$$ 
	In particular, by Lemma \ref{lemma T deformation}, $\hat{\gamma}$ is a $2n$ us/su-loop at $(T,(0,x))$. 
\end{defi}

\begin{remark}
In the following, we will mostly consider us-loops; for that reason, we will state the technical lemmas needed for the proof only for us-loops, but similar results hold for su-loops as well. 
\end{remark}

Similarly to Lemma \ref{lemma continuation}, given a point $x \in M$ and a us-loop at $(f,x)$, we can define a natural continuation for the $C^1$ deformation $\hat f$ with $I$-parameters we consider:
\begin{defi} \label{lift of loop for deformation}
	Let $x \in M$, let $n \geq 2$, we say that $\gamma=\{\gamma(t) = [x,x_1(t),\dots,x_{2n}(t)]\}_{t \in [0,1]}$ is a continuous family of $2n$ us-loops at $(f,x)$ if for each $t \in [0,1]$, $\gamma(t)$ is a $2n$ us-loop, and for each $i=1,\dots,2n$, the map $t \mapsto x_i(t)$ is continuous. Given such a family, for any $t \in [0,1]$, we let $\hat{\gamma}(t)$ be the lift of $\gamma(t)$ as above. Then by continuity, there exists a $C^2$-uniform constant $\hat{\delta}=\hat{\delta}(T,\gamma)> 0$ %\footnote{ also depend on $\norm{T}_{C^1}$}, 
	such that $B(0,\hat{\delta}) \subset \mathcal{U}$, and for any $(b,y,t) \in \cWc_T((0,x),\hat{\delta})\times [0,1]$,   for some constant $\hat{h}=\hat{h}(T,\gamma)> 0$,  the following intersections exist and are unique:
	\begin{itemize}
	\item $\{(b, \hat{x}_1^{b,y}(t))\}:=\cWu_{T,\mathrm{loc}}((b,y), \hat{h})\cap\cWcs_{T,\mathrm{loc}}((0,x_1(t)), \hat{h})$;
	\item $\{(b,  \hat{x}_2^{b,y}(t))\}:=\cWs_{T,\mathrm{loc}}((b,  \hat{x}_1^{b,y}(t)),\hat{h})\cap\cWcu_{T,\mathrm{loc}}((0,x_2(t)),\hat{h})$\dots
	\item \dots $\{(b,  \hat{x}_{2n-1}^{b,y}(t))\}:=\cWu_{T,\mathrm{loc}}((b,  \hat{x}_{2n-2}^{b,y}(t)),\hat{h})\cap\cWcs_{T,\mathrm{loc}}((0,x),\hat{h})$;
	\item $\{(b, \hat{x}_{2n}^{b,y}(t))\}:=\cWs_{T,\mathrm{loc}}((b,  \hat{x}_{2n-1}^{b,y}(t)),\hat{h})\cap\cWc_{T,\mathrm{loc}}((0,x),\hat{h})$.
	\end{itemize}
	%Thus for each $(b,y) \in \cWc_T((a,x),\delta_f)$, w
	We thus have a continuous family of $2n$ us-loops at $(f_b,y)$, denoted by $\{\hat \gamma^{b,y}(t)\}_{t\in [0,1]}$:
	\begin{equation*}
	\hat \gamma^{b,y}(t) := [y,\hat{x}_1^{b,y}(t), \dots, \hat{x}_{2n}^{b,y}(t)], \quad \forall\, t\in [0,1].
	\end{equation*} 
	We define the map
	\begin{equation} \label{widehatpsidefi}
	\widehat{\psi} = \widehat{\psi}(T,x,\gamma) \colon\left\{
	\begin{array}{rcl}
	\cWc_{T}((0,x), \hat{\delta}) \times [0,1] &\to& \cWc_T(0,x),  \\
	(b,y,t)&\mapsto&  H_{T, \hat \gamma %_{b,y}\clb
		(t)}(b,y)=(b, \hat{x}_{2n}^{b,y}(t)).
	\end{array}
	\right.
	\end{equation}
	For any $(b,y) \in \cWc_{T}((0,x), \hat{\delta})$, we thus get a map $\psi = \psi(T,x,\gamma)$: 
	\begin{equation}\label{relationwidehatpsipsifxgamma}
	\psi(b,y,\cdot):=\pi_M\widehat{\psi}(b,y,\cdot)\colon
	[0,1] \to \cWc_{f_b}(y),
	\end{equation}
	where $\pi_M \colon \mathcal{U} \times M \to M$ denotes the canonical projection. 
\end{defi}

%\begin{lemma}[Continuation of us-loops]\label{extension continuous family floops}
%	Let $x \in M$. There exist $\mathcal{V}$, a $C^1$-open neighbourhood of $f$, as well as $ \varsigma > 0$ such that  the following is true. Let $\gamma$ be a continuous family of us-loops at $(f,x)$ satisfying $\ell(\gamma)<\frac{\varsigma}{2}$. Then
%	for any $g \in \mathcal{V}$ and $y \in B(x, \varsigma)$, we can define $\gamma^{g,y}$, a continuous family of us-loops at $(g,y)$, such that  $\gamma^{f,x}= \gamma$ and each coordinate of $\gamma^{g,y}(t)$ depends continuously on $(g,y,t)$.
%\end{lemma}

\begin{defi} \label{support of deformation} 
	Let $I\geq 1$ be some integer. 
	For any infinitesimal $C^{r}$ deformation with $I$-parameters $V \colon \R^{I}\times M \to T M$, we define
	\begin{align*}
	\mathrm{supp}(V):= \{ x \in M\ \vert\ \exists\, B \in \R^{I} \mbox{ such that } V(B,x) \neq 0 \}.
	\end{align*} 
Given  an open neighbourhood $\mathcal{U}$ of $\{0\}$ in $\R^I$, and a $C^{r}$ deformation at $f$ with $I$-parameters $\hat{f} \colon \mathcal{U} \times M \to M$, we define
\begin{align*}
\mathrm{supp}(\hat{f}):= \{x \in M\ \vert\ \exists\, b \in \mathcal{U} \text{  such that } \hat{f}(b,x) \neq f(x) \}.
\end{align*}
\end{defi}
%It is clear from Definitions \ref{def infini deform} and \ref{support of deformation} that for any infinitesimal  $C^{r}$ deformation $V$, if $\hat{f}$ is the $C^{r}$ deformation of $f$ generated by $V$, then we have 
%\begin{align}\label{supp V supp hat f}
%\mathrm{supp} (\hat{f}) \subset f^{-1}(\mathrm{supp}(V)).
%\end{align}

We introduce the following definitions  in order to control return times of a map to the support of a deformation; they are motivated by the fact that  for very large return times, it is possible to achieve a good control on how certain holonomies change after perturbation. 
%We introduce the following definitions.

\begin{defi} \label{def rec funtion}
For any subsets $A, B \subset M$, and for $*\in \{+,-\}$, we define
%\aryst
%\cb\marginpar{missing definitions}
\begin{align*}
R(f, A, B) &:= \inf \{ n \geq 0 \ \vert \ f^{n}(A) \cap B \neq \emptyset \mbox{ or } f^{-n}(A) \cap B \neq \emptyset\};\\
%\color{red} R_{\geq 0}(f, A, B) &:= \inf \{n \geq 0 \ \vert\ f^{n}(A) \cap B \neq \emptyset  \}; \\
%\color{black}
R_{*}(f, A, B) &:= \inf \{ n \geq 1 \ \vert\ f^{* n}(A) \cap B \neq \emptyset \}.
%R_{-}(f, A, B) &:= \inf \{ n \geq 1 \ \vert\ f^{-n}(A) \cap B \neq \emptyset \}.
\end{align*}
We abbreviate $R(f, A, A)$, $R_{*}(f, A, A)$ respectively as $R(f, A)$, $R_{*}(f, A)$. 
%\earyst
%For any subset $A \subset M$, we use the abbreviation $R_{\pm}(f,A):= R_{\pm}(f,A,A)$. 
Similarly, for a $C^1$ deformation $\hat{f} \colon \mathcal{U} \times M \to M$ of $f$,  and for $*\in \{+,-\}$,  we set
%\aryst
\begin{align*}
%R_{\geq 0}(\hat{f}, A, B) &:= \inf \{ n \geq 0 \ \vert \ \exists b \in U \mbox{ such that } \hat{f}(b,\cdot)^{n}(A) \cap B \neq \emptyset \},  \\
R(\hat{f}, A,B)&:=\inf \{ n \geq 0 \ \vert\ \exists\, b \in \mathcal{U} \mbox{ s.t. } \hat{f}(b,\cdot)^{ n}(A)\cap B\neq \emptyset \mbox{ or } \hat{f}(b,\cdot)^{- n}(A)\cap B\neq \emptyset\},\\
R_{*}(\hat{f}, A, B)&:= \inf \{ n \geq 1 \ \vert\ \exists\, b \in \mathcal{U} \mbox{ s.t. } \hat{f}(b,\cdot)^{* n}(A) \cap B \neq \emptyset \},
%R_{-}(\hat{f}, A, B) &:= \inf \{ n \geq 1 \ \vert\ \exists b \in U \mbox{ such that } \hat{f}(b,\cdot)^{-n}(A) \cap B \neq \emptyset \}.
\end{align*}
and we abbreviate $R(\hat{f}, A, A)$, $R_{*}(\hat{f}, A, A)$ respectively as $R(\hat{f}, A)$, $R_{*}(\hat{f}, A)$. 
%\earyst
%Moreover, it is clear that $R(\hat{f},A,B) = \min(R_{\geq 0}(\hat{f}, A, B), R_{-}(\hat{f}, A, B))$.
\end{defi}

In the following, most of the time\footnote{Except in Subsection \ref{break trivial global} where deformations  with $4$-parameters are needed.}, we restrict ourselves to the case of  deformations  with $2$-parameters, i.e., we take a small neighbourhood  $\mathcal{U}\subset \R^2$  of $\{0_{\R^2}\}$, we let $\hat{f} \colon \mathcal{U} \times M \to M$ be a $C^1$ deformation at $f$ with $2$-parameters generated by an infinitesimal $C^{1}$ deformation with $2$-parameters $V \colon \R^{2}\times M \to T M$, and we set  $T = T(\hat{f})$.  

\begin{defi}[Adapted deformation]\label{defiadaptedtosth}
	Let $x \in M$, let $n\geq 2$ be some integer, and let  $\gamma= [x,x_{1},\dots,x_{2n}]$ be a $2n$ us-loop or su-loop at $(f,x)$  with $\ell(\gamma)<\sigma$ for some small $\sigma>0$.  Given two constants $C,R_0 > 0$, we say that an infinitesimal  $C^{r}$ deformation $V$  is \emph{adapted to $(\gamma, \sigma, C,R_0)$} if 
	\begin{enumerate}
	\item  $\sigma\|\partial_b\partial_xV\|_{M} + \|\partial_bV\|_{M} < C$;
	\item $R(f, \{z\}, \mathrm{supp}(V)) > R_0$ for $z= x, x_{2}, \dots,x_{2n}$;
	\item $R_{\pm}(f, \{x_1\},  \mathrm{supp}(V)) > R_0$.
	\end{enumerate}
\end{defi}

\begin{prop}[see Proposition 5.6, \cite{LZ}]\label{determinant for smooth deformations} %\marginpar{I changed the name to proposition}
	For any  $C, \kappa > 0$,  there exist $C^2$-uniform constants $R_0  = R_0(f, C, \kappa)> 0$ and $\kappa_0 = \kappa_0(f,  C, \kappa) > 0$ such that the following is true. 
	
	Let $x \in M$, let $n\geq 2$ be some integer, and let  $\gamma= [x,x_{1},\dots,x_{2n}]$ be a $2n$ us-loop at $(f,x)$ of length $\sigma>0$ such that there exists an infinitesimal  $C^{r}$ deformation $V$  that is adapted to $(\gamma, \sigma,C, R_0)$. In the following, we denote by  $B = (B_1,B_2)$ an element of $T_{0}\mathcal{U} \simeq \R^{2}$. Assume that 
	for all $z \in \{x,x_2,\dots,x_{2n}\}$, we have
	\begin{equation} \label{determinant 1}
	D_{B}(\pi_cV(B, z)) = 0, %\quad \forall z \in \cWc_f(x_{i,j,1}, \cg \Lambda_f^{4c}\sigma\clb),
	\end{equation}
	while %for some integer $j_0\in \{1,2\}$, and   indices $\{\alpha_{j}\}_{j=1,2} \subset \{1,\dots, I\}$, we have that if  $j\in \{1,2\}\setminus \{j_0\}$, then for all $z \in \cWc_f(x^{j}_1, K %2^cC_f^{c}\Lambda_f^{4c}
	%\sigma)$, it holds
	%\begin{equation} \label{determinant 1}
	%D_{B_{\alpha_{1}}, B_{\alpha_{2}} }(\pi_cV(B, z)) = 0, %\quad \forall z \in \cWc_f(x_{i,j,1}, \cg \Lambda_f^{4c}\sigma\clb),
	%\end{equation}
	%while 
	%for any $z \in \cWc_f(x_1, K %2^cC_f^{c}\Lambda_f^{4c} 
	%\sigma)$, we have
	\begin{equation} \label{determinant 2}
	\big|\det(B \mapsto D_{B}(\pi_cV(B, x_1)) )\big| > \kappa,
	\end{equation}
	where $\pi_c \colon T M \to E^{c}_{f}$ denotes  the canonical projection.
	
	Then, the map 
	$$
	\Xi \colon\left\{\begin{array}{rcl}
	T_0 \mathcal{U} &\to&  E^c_{f}(x_{2n}),\\
	B &\mapsto& \hat\pi_c D H_{T, \hat{\gamma}}(B+\nu_0(x,B)),
	\end{array} 
	\right.
	$$
	satisfies
	\begin{equation*}
	\det \Xi \geq \kappa_0,
	\end{equation*}
	where $\hat{\gamma}$ is the lift of $\gamma$ for $T$, and $\hat \pi_c\colon E^{c}_{T}(0,x_{2n}) = Graph(\nu_0(x_{2n},\cdot)) \oplus E^{c}_{f}(x_{2n})\to E^{c}_{f}(x_{2n})$ denotes the canonical projection. 
	%\end{equation*}
\end{prop}

\subsection{Construction of $C^r$ deformations at $f$}\label{construct }

In the following, we assume  that $\dim E_f^c=2$. 
Recall that $\Pi^c \colon \R^{d} \simeq \R^{2} \times \R^{d_u+d_s} \to \R^{2}$ is the canonical projection, and that $\Pi^c_x$ is the map $\Pi_x^c:= \Pi^c \circ \phi_x^{-1} \colon M \to \R^2$.

\begin{lemma}\label{lem def regular family and chart}
	Let $\tilde{K}=\tilde{K}(f)\in (0,1)$, $\tilde{\sigma}=\tilde{\sigma}(f)>0$ be as in Lemma \ref{lem construction families of loops}. Then, for any %non-periodic point $x_0 \in \Gamma_f$, for any 
	$R_0>0$, for any integer $k_0\geq 1$,  %,  $\tilde{\sigma}=\tilde{\sigma}(f,x_0,R_0)\in (0,\tilde h)$ 
	%such that 
	for any $\sigma\in (0,\tilde{\sigma})$, and for any point $x_0 \in \Gamma_f$ satisfying $R_{\pm}(f,B(x_0,10\sigma))>R_0$,\footnote{\label{footnote recall }Recall Definition \ref{support of deformation} and Definition \ref{def rec funtion}.}  there exists an infinitesimal $C^r$ deformation at $f$ with $2k_0$-parameters $V \colon \R^{2k_0}\times M \to T M$ such that $\mathrm{supp}(V)\subset B(x_0,10\sigma)$,\footref{footnote recall } and there exists a continuous map $\overline{\Gamma}_f(x_0)\cap \cWc_f(x_0,\tilde{K}\sigma)\ni x \mapsto \gamma^x$ such that $\gamma^x=\{\gamma^x(t) = [x,x_1^x(t),\dots,x_{2n}^x(t)]\}_{t \in [0,1]}$ is a continuous family of $2n$ us-loops at $(f,x)$, with $n:=n(x_0)\in \{2,5\}$,  $\ell(\gamma^x)< \sigma$,  such that $\gamma^x(0)$ is trivial, and for any integer $k \in \{1,\dots,k_0\}$, we have:
	\begin{enumerate}
		%\item $n = 2$, when  $x \in \tilde \Gamma_f^0$; otherwise $n=5$, when $x \in  \Gamma_f^1 \setminus  \tilde\Gamma_f^0$;
		%\item 
		\item $\gamma^x(\frac{k}{k_0})$ is a non-degenerate closed us-loop;
		%\item $\mathrm{supp}_X(V)\subset B(x_1, \sigma)\setminus \{x,x_2,\dots,x_{2n}\}$; 
		\item\label{point un du lemme} %for any $(b,y) \in \cW_T^c((0,x),\tilde \delta)$, 
		$V$ is adapted to $(\gamma^x(\frac{k}{k_0}), \sigma,\widetilde C,R_0)$, for some $C^2$-uniform constant $\widetilde C=\widetilde C(f,k_0)>0$; %, for any integer $k \in \{1,\dots,k_0\}$;
		\item\label{point deux du lemme}
		%for any integer $k \in \{1,\dots,k_0\}$, 
		for any $z \in \big\{x,x_2^x(\frac{k}{k_0}),\dots,x_{2n-1}^x(\frac{k}{k_0})\big\}$, it holds 
		\begin{equation*} 
		D_{B}(\pi_cV(B, z)) = 0, %\quad \forall z \in \cWc_f(x_{i,j,1}, \cg \Lambda_f^{4c}\sigma\clb),
		\end{equation*}
		and there exists a $2$-dimensional vector space $E_k\subset \R^{2k_0}$ such that
		\begin{equation*} 
		\Big|\det\Big(E_k \ni B \mapsto D_{B}\big(\pi_cV\big(B, x_1^x\big(\frac{k}{k_0}\big)\big)\big)\Big)\Big| > \tilde \kappa,
		\end{equation*}
		for some $C^2$-uniform constant   $\tilde \kappa=\tilde \kappa(f)>0$, 
		where $\pi_c \colon TM  \to E^{c}_{f}$ denotes  the canonical projection.
	\end{enumerate} 
%Moreover, when $x \in \Gamma_f^0$, we can take $n=2$, and $x_{4}(t)=x$, for all  $t \in [0,1]$. 
\end{lemma}

\begin{figure}[H]
	\begin{center}
		\includegraphics [width=13.5cm]{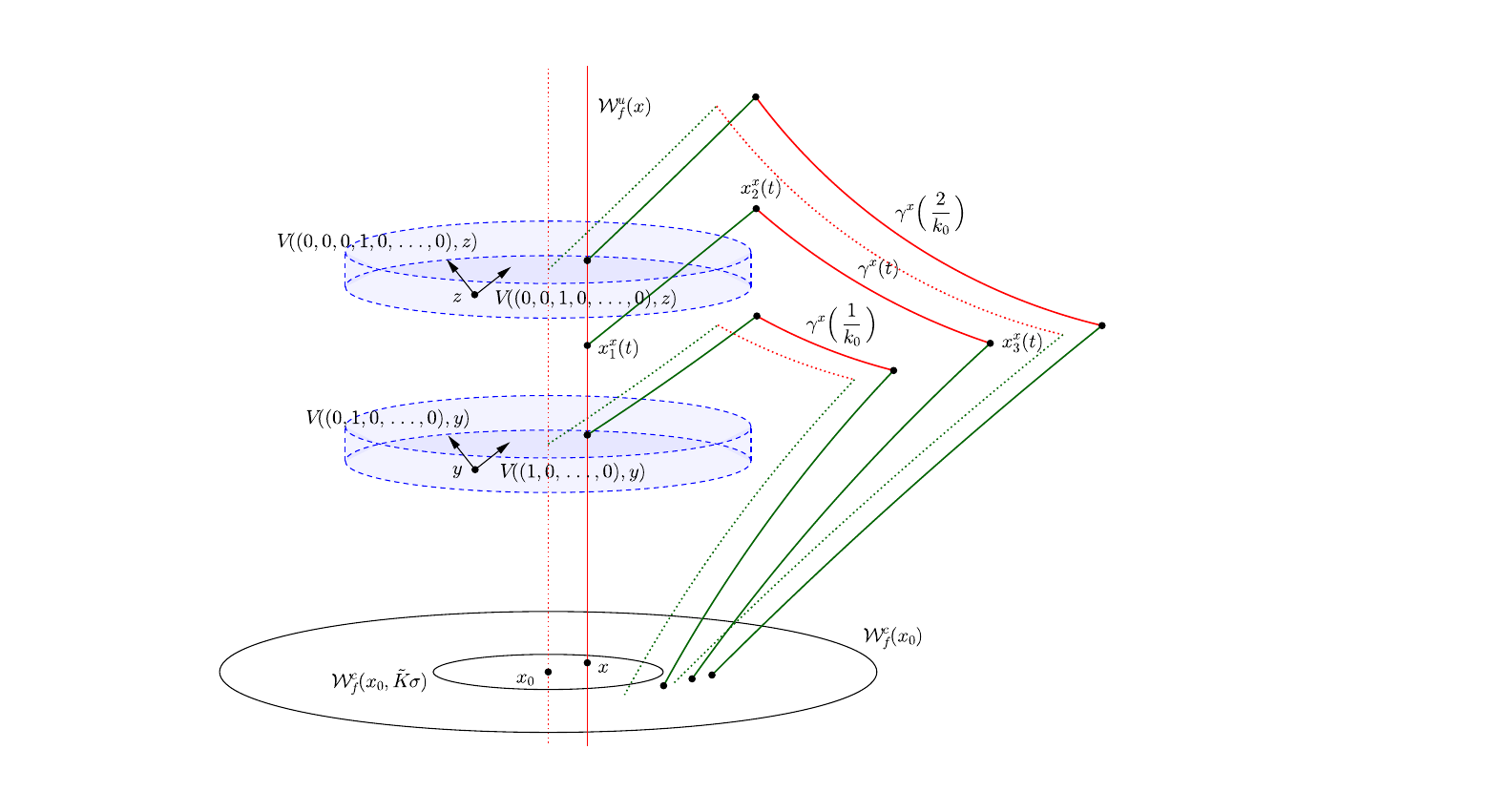}
		\caption{Localization of the perturbations.}
	\end{center}
\end{figure}

%For the proof of Lemma \ref{lem def regular family and chart}, we will need the following lemma. 

Actually, we will apply Lemma \ref{lem def regular family and chart} with $k_0=1$ or $2$ in the following. The construction in the proof of Lemma \ref{lem def regular family and chart} is adapted from \cite{LZ}. 

\begin{proof}[Proof of Lemma \ref{lem def regular family and chart}]
	Let $R_0>0$, and let $k_0\geq 1$ be  some integer. Let $\tilde{K}=\tilde{K}(f)\in (0,1)$, $\tilde{\sigma}=\tilde{\sigma}(f)>0$ be as in Lemma \ref{lem construction families of loops}, %Let $\sigma_0=\sigma_0(f)>0$,   $K_0=K_0(f)>0$ be as in Lemma \ref{eleme lemma}, %set $\tilde \sigma=\tilde \sigma(f):=\min(\overline{h},\sigma_0)>0$, 
	and take some small $\sigma \in (0,\tilde{\sigma})$. %\min(\overline{h},\sigma_0))$. 
	
	We consider a point $x_0 \in \Gamma_f$ such that $R_{\pm}(f,B(x_0,10\sigma))>R_0$, and set $n:=n(x_0)\in \{2,5\}$. We let $\overline{h}=\overline{h}(f)>0$ and $\phi =  \phi_{x_0} \colon (-\overline{h},\overline{h})^{d} \to M$ be given by Lemma \ref{lemma constr phi}. 
	 
	For  each $x\in \overline{\Gamma}_f(x_0)\cap \cWc_f(x_0,\tilde{K}\sigma)$ we let $\gamma^x=\{\gamma^x(t) = [x,x_1^x(t),\dots,x_{2n}^x(t)]\}_{t \in [0,1]}$ be the continuous family of $2n$ us-loops at $(f,x)$ constructed in Lemma \ref{lem construction families of loops}.

	For each integer $k\in \{1,\dots,k_0\}$, we let $z_k=(z_k^c,z_k^u,z_k^s):=z_1^{x_0}(\frac{k}{k_0})\in (-\overline{h},\overline{h})^d=(-\overline{h},\overline{h})^2 \times (-\overline{h},\overline{h})^{d_u}\times (-\overline{h},\overline{h})^{d_s}$. % be such that $z_k=z_1^{x_0}(\frac{k}{k_0})$. %By point  \eqref{deuxieme prop} in Lemma \ref{eleme lemma}, for some $C^2$-uniform constant $K=K(f)>1$, and  $\sigma'\in  [K^{-1} \sigma,K \sigma]$, we have $(-\overline{h},\overline{h})^2 \times \big(z_1^u+(\sigma')^{d_u} \big)\times \big(z_1^s+(\sigma')^{d_s} \big)\cap \phi^{-1}\{x,x_2(1),\dots,x_{2n-1}(1)\}=\emptyset$,  %; note that it is possible  choose $\gamma$ and $\sigma'$ in such a way that $\sigma'\in (\frac{1}{100}\sigma,\sigma)$. 
	%Without loss of generality, we will thus  assume in the following that $\sigma'$ is of the same order as $\sigma$. 
	 We define a collection of functions and vector fields as follows.
	 
%	$\bullet$ For $*=u,s$, let $\rho^{*, \sigma'} \colon (-\overline{h},\overline{h})^{d_*} \to \R_+$ be a $C^{\infty}$ function such that
%	\begin{enumerate}
%		\item $\mathrm{supp}( \rho^{*, \sigma'} )\subset (-\frac{\sigma'}{3}, \frac {\sigma'}{3})^{d_*}$ and $\rho^{*, \sigma'}|_{(-\frac{\sigma'}{5}, \frac{\sigma'}{5})^{d_*}} \equiv 1$;
%		\item $\|\rho^{*, \sigma'}\|_{C^1} < 10^4(\sigma')^{-1}$.
%	\end{enumerate}

   For each $1\leq j\leq 2$, let $U_j \colon (-\frac{1}{5}, \frac{1}{5})^{2} \times (- \frac{1}{5}, \frac{1}{5})^{d_u} \times (-\frac 13, \frac 13)^{d_s} \to \R^{d}$ be a compactly supported $C^{\infty}$ divergence-free vector field such that $U_j$ restricted to $(-\frac{1}{10}, \frac{1}{10})^{2} \times (- \frac{1}{10}, \frac{1}{10})^{d_u}\times (-\frac 15, \frac 15)^{d_s}$ is equal to the constant vector  %$e_j:=(0,\dots,\underset{\substack{j-1}}{0},\underset{\substack{j}}{1},\underset{\substack{j+1}}{0},\dots,0)$. 
   $e_j$, where $e_1:=(1,0,0,\dots,0)\in \R^d$ and $e_2:=(0,1,0,\dots,0)\in \R^d$. 
%Such $U_j$ always exists since $d \geq 3$.
Moreover, we can assume that $U_j$ satisfies $\|U_j\|_{C^1} < C_{*}$ for some  constant $C_{*} = C_{*}(d) > 0$.

For any $x_c \in \R^2$, $x_u \in \R^{d_u}$, $x_s \in \R^{d_s}$, for any $\lambda_c, \lambda_u, \lambda_s > 0$, and for any $z_c \in \R^2, z_u \in \R^{d_u}, z_s \in \R^{d_s}$, we set:
\begin{equation*}
\Lambda_{x_c, x_u, x_s}^{\lambda_c,\lambda_u,\lambda_s}(z_c, z_u, z_s) = (x_c + \lambda_c z_c, x_u + \lambda_u z_u, x_s + \lambda_s z_s).
\end{equation*}

For any $j \in \{1,2\}$, any $x_u \in \R^{d_u}$, we let $U^{\sigma}_{j,x_u} \colon (-\overline{h}, \overline{h})^d \to \R^d$ be the vector field %defined as follows:
\begin{align*}
U^{\sigma}_{j,x_u} = U_j \big(\Lambda_{0_{\R^2},x_u,0_{\R^{d_s}}}^{\tilde{K}\sigma,\frac{\tilde{K}}{k_0}\sigma,\tilde{K}\sigma}\big)^{-1}.
\end{align*}
The support of $U^{\sigma}_{j,x_u}$ is contained in
\begin{equation*}
\Big(-\frac{\tilde{K}}{5}\sigma, \frac{\tilde{K}}{5}\sigma\Big)^2 \times \Big(x_u +\big(- \frac{\tilde{K}}{5k_0}\sigma,  \frac{\tilde{K}}{5k_0}\sigma\big)^{d_u}\Big) \times \Big(-\frac{\tilde{K}}{3}\sigma,\frac{\tilde{K}}{3}\sigma\Big)^{d_s}.
\end{equation*}
Moreover, for any $z_c \in (-\frac{\tilde{K}}{10}\sigma,\frac{\tilde{K}}{10}\sigma)^2$, for any $z_u \in x_u +(- \frac{\tilde{K}}{10k_0}\sigma,  \frac{\tilde{K}}{10k_0}\sigma)^{d_u}$ and for any $z_s \in \big(- \frac{\tilde{K}}{5}\sigma, \frac{\tilde{K}}{5}\sigma\big)^{d_s}$,  it holds 
\begin{equation*}
U^{\sigma}_{j,x_u}(z_c, z_u, z_s) = e_j.
\end{equation*}
We set
\begin{equation*}
V^{\sigma}_{j,x_u}  := D\phi(U^{\sigma}_{j,x_u}).
\end{equation*} 
The vector field $V_{j,x_u}^\sigma$ is divergence-free and satisfies:
\begin{equation}\label{eq Vitj adapt C1}
\sigma \| \partial_x V^{\sigma}_{j,x_u}\|_{M} + \|V^{\sigma}_{j,x_u}\|_{M} < \widetilde{C}_0,
\end{equation}
for some $C^2$-uniform constant $\widetilde{C}_0=\widetilde{C}_0(f,k_0)>0$.

%
%	
%$\bullet$	For $j=1,2$, we let $U_j \colon (-\overline{h},\overline{h})^{2} \to \R^{2}$ be a $C^{\infty}$ divergence-free vector field such that
%	\begin{enumerate}
%		\item $U_j$ restricted to $(-\frac{\overline{h}}{2},\frac{\overline{h}}{2})^{2}$ is equal to the constant vector  $e_j$, with $e_1:=(1,0)$ and $e_2:=(0,1)$;
%		\item $U_j$ vanishes outside $(-\frac{2\overline{h}}{3},\frac{2\overline{h}}{3})^{2}$.
%	\end{enumerate}
%	We can also assume that $U_j$ satisfies $\|U_j\|_{C^1} < C_{*}\overline{h}^{-1}$ for some  constant $C_{*}> 0$. 
%	
%	Now,   we let $U_j^{\sigma'} \colon (-\overline{h},\overline{h})^d \to \R^d$, $V_j^{\sigma'}$ be the vector fields %defined as follows:
%	\begin{align*}
%	U^{\sigma'}_j(z) &:=U_j(z^c) \rho^{u, \sigma'}(z^u-z_1^u) \rho^{s, \sigma'}(z^s-z_1^s) ,\quad \forall\, z = (z^c, z^u, z^s)\in (-\overline{h},\overline{h})^d,\\
%	V^{\sigma'}_j  &:= D\phi(U^{\sigma'}_{j}).
%	\end{align*}
%
%By Lemma \ref{lemma constr phi}, $C^1$ bounds on $\rho^{u, \sigma'},\rho^{s,\sigma'},U_j$ above, and since $\sigma'\in  [K^{-1} \sigma,K \sigma]$, we see that the vector field $V_{j}^{\sigma'}$ is divergence-free and satisfies
%\begin{equation}\label{eq Vitj adapt C1}
%\sigma \| \partial_x V_{j}^{\sigma'}\|_{M} + \| V_{j}^{\sigma'}\|_{M} < \widetilde{C},
%\end{equation}
%for some $C^2$-uniform constant $\widetilde{C}=\widetilde{C}(f)>0$.

Let $V \colon \R^{2k_0} \times M \to T M$ be the infinitesimal $C^r$ deformation defined as 
\begin{align}
V(B,\cdot)&:=  \big(B_{1,1} V^\sigma_{1,z_1^u}+B_{2,1} V^{\sigma}_{2,z_1^u}\big)+\big(B_{1,2} V^\sigma_{1,z_2^u}+B_{2,2} V^{\sigma}_{2,z_2^u}\big)+\dots+\nonumber\\ 
&\quad +\big(B_{1,k_0-1} V^\sigma_{1,z_{k_0-1}^u}+B_{2,k_0-1} V^{\sigma}_{2,z_{k_0-1}^u}\big)
+\big(B_{1,k_0} V^\sigma_{1,z_{k_0}^u}+B_{2,k_0} V^{\sigma}_{2,z_{k_0}^u}\big),\label{term 2030}  
\end{align}
for all $B =\sum_{k=1}^{k_0} B_{1,k} u_{2k-1} +B_{2,k} u_{2k} \in \R^{2k_0}$, where $(u_i)_{i=1}^{2k_0}$ denotes the canonical basis of $\R^{2k_0}$. 

By definition, the map $V$ is linear in $B$. Moreover, by \eqref{eq Vitj adapt C1}, \eqref{term 2030}, it holds 
\begin{equation} \label{term lip v}
\sigma \|\partial_b\partial_xV\|_{M} + \|\partial_bV\|_{M} <  \widetilde{C},
\end{equation} 
with $\widetilde{C}:=2k_0\widetilde{C}_0>0$. 

As $\ell(\gamma^x)<\sigma$, we have $\mathrm{supp}(V)\subset   B(x_0,10\sigma)$, and for any integer $k \in \{1,\dots,k_0\}$, %for $\tilde K>0$ sufficiently small, 
it holds  $x_2^x(\frac{k}{k_0}),x_3^x(\frac{k}{k_0}),\dots,x_{2n-1}^x(\frac{k}{k_0})\in B(x_0,10\sigma)$, for  all $x \in\overline{\Gamma}_f(x_0)\cap \cWc_f(x_0,\tilde{K}\sigma)$. %Since $x_0$ is non-periodic, then by choosing $\sigma$ sufficiently small, we can ensure that\footnote{Recall Definition \ref{def rec funtion}.} 
Recall that by assumption, we have 
\begin{equation}\label{rec temps}
R_{\pm}(f,B(x_0,10\sigma))>R_0. 
\end{equation}
By \eqref{term lip v} and \eqref{rec temps}, we conclude that $V$ is adapted to $(\gamma^x(\frac{k}{k_0}), \sigma,\widetilde C,R_0)$.

By construction,   for any $z \in \big\{x,x_2^x(\frac{k}{k_0}),\dots,x_{2n-1}^x(\frac{k}{k_0})\big\}$, it holds 
\begin{equation*} 
D_{B}(\pi_cV(B, z)) = 0, %\quad \forall z \in \cWc_f(x_{i,j,1}, \cg \Lambda_f^{4c}\sigma\clb),
\end{equation*}
and 
\begin{equation*} 
\Big|\det\Big(E_k \ni B \mapsto D_{B}\big(\pi_cV\big(B, x_1^x\big(\frac{k}{k_0}\big)\big)\big)\Big)\Big| > \tilde \kappa,
\end{equation*}
for some $C^2$-uniform constant   $\tilde \kappa=\tilde \kappa(f)>0$, where $E_k:=\mathrm{Span}(u_{2k-1},u_{2k})\subset \R^{2k_0}$, and   $\pi_c \colon TM  \to E^{c}_{f}$ denotes  the canonical projection.
\end{proof}

\begin{corollary}\label{coro un peu chiant}
	For any integers $k_0\geq 1$, $r \geq 2$, for any $\delta>0$, there exist $C^2$-uniform constants $\tilde{K}_0=\tilde{K}_0(f)\in (0,1)$, $\tilde{\sigma}_0=\tilde{\sigma}_0(f,k_0)>0$, $\tilde R_0=\tilde R_0(f,k_0)>0$ and  $\tilde{\delta}_0=\tilde{\delta}_0(f,r,\delta)>0$ such that for any $\sigma\in (0,\tilde{\sigma}_0)$, for any point $x_0 \in \Gamma_f$ satisfying $R_{\pm}(f,B(x_0,10\sigma))>\tilde R_0$, there exists a $C^r$ deformation $\hat f \colon B(0_{\R^{2k_0}},\tilde{\delta}_0) \times M \to M$ at $f$ with $2k_0$-parameters generated by an infinitesimal $C^r$ deformation $V \colon \R^{2k_0}\times M \to T M$, such that  $\mathrm{supp}(\hat f)\subset B(x_0,10\sigma)$,\footnote{Recall Definition \ref{support of deformation} .} and there exists a continuous map $\overline{\Gamma}_f(x_0)\cap \cWc_f(x_0,\tilde{K}_0\sigma)\ni x \mapsto \gamma^x$, such that
		\begin{enumerate}
		\item 
		 $\gamma^x=\{\gamma^x(t) = [x,x_1^x(t),\dots,x_{2n}^x(t)]\}_{t \in [0,1]}$ is a continuous family of $2n$ us-loops at $(f,x)$ as in Lemma \ref{lem def regular family and chart}, with $n:=n(x_0)\in \{2,5\}$,  $\ell(\gamma^x)< \sigma$,  such that $\gamma^x(0)$ is trivial, and for any integer $k \in \{1,\dots,k_0\}$,  $\gamma^x(\frac{k}{k_0})$ is a non-degenerate closed us-loop;
		\item\label{point deux de la prop} let  $T=T(\hat f)$, let $\psi_x:=\psi(T,x,\gamma^x)$ be the map defined in  %\eqref{widehatpsidefi}-
		\eqref{relationwidehatpsipsifxgamma}, let $\Pi_x^c \colon M \to \R^2$ be the map given by Lemma \ref{lemma constr phi}, and for $k \in \{1,\dots,k_0\}$, let $\Phi^{(k)}\colon (b,x)\mapsto\Pi_x^c\psi_x(b,x,\frac{k}{k_0})$; then,  %for any integer $k \in \{1,\dots,k_0\}$,  there exists a $2$-dimensional vector space  $E_k\subset \R^{2k_0}$ as in Lemma \ref{lem def regular family and chart} such that %for %$\R^{2k_0}\simeq E:=\oplus_{k=1}^{k_0} E_k$, and
		 %$B(0_{\R^{2k_0}},\tilde{\delta}_0):=\{b\in \R^{2k_0}: %^1,\dots,b^{k_0})\in E: \sum_{k=1}^{k_0}|b^k|
		 %|b|\leq \tilde{\delta}_0 \}$,
		  the map 
		\begin{equation*}
		\Phi\colon\left\{
		\begin{array}{rcl}
		B(0_{\R^{2k_0}},\tilde{\delta}_0) \times \Big(\overline{\Gamma}_f(x_0)\cap \cWc_f(x_0,\tilde{K}_0\sigma)\Big) & \to & \R^{2k_0}\\
		(b,x) & \mapsto &\big(\Phi^{(k)}(b,x)\big)_{k=1,\dots,k_0}
		\end{array}
		\right.
		\end{equation*}
		 is continuous; besides, for any $x \in \overline{\Gamma}_f(x_0)\cap \cWc_f(x_0,\tilde{K}_0\sigma)$, $\Phi(\cdot,x)$ is $C^1$, and 
			$$
			|\det D_b|_{b=0}\big(\Phi(\cdot,x)\big)| > \tilde \kappa_0,
			$$
			for some $C^2$-uniform constant   $\tilde \kappa_0=\tilde \kappa_0(f,k_0)>0$; 
		\item\label{point un de la prop}
		$d_{C^r}(f,f_b)< \delta$, for all $b \in B(0_{\R^{2k_0}},\tilde{\delta}_0)$, where $f_b:=\hat f(b,\cdot)\in \mathcal{PH}^r(M)$.
	\end{enumerate} 
\end{corollary}

\begin{proof}
	Fix two integers  $k_0\geq 1$ and $r \geq 2$. Let $\tilde K_0:=\tilde K(f)>0$, $\tilde \sigma:=\tilde \sigma(f)>0$, $\widetilde C=\widetilde C(f,k_0)>0$  and   $\tilde \kappa=\tilde \kappa(f)>0$ be the constants given by Lemma \ref{lem def regular family and chart}, and let $\tilde R_0:=R_0(f,\widetilde C,\tilde \kappa)>0$, $\kappa_0:=\kappa_0(f,\widetilde C,\tilde \kappa)>0$ be the constants given by Proposition \ref{determinant for smooth deformations}. 
	Given $\sigma \in (0,\tilde \sigma)$, we consider a point $x_0 \in  \Gamma_f$ such that $R_{\pm}(f,B(x_0,10\sigma),B(x_0,10\sigma))>\tilde R_0$, and set $n:=n(x_0)\in \{2,5\}$.  
	  
	 We let $V \colon \R^{2k_0}\times M \to T M$ be the infinitesimal $C^r$ deformation at $f$ with $2k_0$-parameters given by Lemma \ref{lem def regular family and chart}. Take a point $x \in \overline{\Gamma}_f(x_0)\cap \cWc_f(x_0,\tilde{K}_0\sigma)$ and let $\gamma^x=\{\gamma^x(t) = [x,x_1^x(t),\dots,x_{2n}^x(t)]\}_{t \in [0,1]}$  be the  continuous family of $2n$ us-loops at $(f,x)$ given by Lemma \ref{lem def regular family and chart}. Recall that the map $x \mapsto \gamma^x$ is continuous. Let $\hat f$ be the $C^r$ deformation at $f$ with $2$-parameters generated by $V$, and let $T=T(\hat f)$. By the properties of $V$ in Lemma \ref{lem def regular family and chart}, we have $\mathrm{supp}(\hat f)\subset B(x_0,10\sigma)$.

		For any $t \in [0,1]$, let $\hat{\gamma}^x(t)$ be the lift of $\gamma^x(t)$ for $T$, and  let us denote by $\hat \pi_c\colon E^{c}_{T}(0,x) = Graph(\nu_0(x,\cdot)) \oplus E^{c}_{f}(x)\to E^{c}_{f}(x)$ the canonical projection. Fix an integer $k \in \{1,\dots,k_0\}$. We let $\Xi_x^{(k)}$ be the map defined as
		$$
		\Xi_x^{(k)} \colon\left\{\begin{array}{rcl}
		\R^2\simeq E_k &\to&  E^c_{f}(x),\\
		B &\mapsto& \hat\pi_c D H_{T, \hat{\gamma}^x(\frac{k}{k_0})}(B+\nu_0(x,B)). 
		\end{array} 
		\right.
		$$
		By points \eqref{point un du lemme}-\eqref{point deux du lemme} of Lemma \ref{lem def regular family and chart}, and by Proposition \ref{determinant for smooth deformations},  it holds 
		\begin{equation}\label{eqaution controlant det}
		|\det \Xi_x^{(k)} |\geq \kappa_0.
		\end{equation}
		%\end{equation*}
		
		Let $\widehat{\psi}_x:=\widehat{\psi}(T,x,\gamma^x)$ and $\psi_x:=\pi_M \widehat{\psi}_x$ be the maps defined in \eqref{widehatpsidefi}-\eqref{relationwidehatpsipsifxgamma}, and let $\Pi_x^c \colon M \to \R^2$ be the map given by Lemma \ref{lemma constr phi}. Let $\hat \delta>0$ be such that $\hat \delta<\hat \delta(T,\gamma^x)$ for all $x \in \overline{\Gamma}_f(x_0)\cap \cWc_f(x_0,\tilde{K}_0\sigma)$, with $\hat \delta(T,\gamma^x)>0$ as in  Definition \ref{lift of loop for deformation}.  
		%Fix an integer $k \in \{1,\dots,k_0\}$, and l
		Let 
		\begin{equation*}%\label{def hat Phi}
		\Phi^{(k)} \colon \left\{
		\begin{array}{rcl}
		\cWc_T((0,x_0),\hat{\delta}) &\to&  \R^{2}, \\
		(b,x)&\mapsto& \Pi_x^c\psi_x(b,x,\frac{k}{k_0}).
		\end{array}
		\right.
		\end{equation*}
		As $x \mapsto \gamma^x$ is continuous, the maps $x \mapsto \psi_x$ and $\Phi^{(k)}$ are continuous as well. 
		
		For each $x \in \overline{\Gamma}_f(x_0)\cap \cWc_f(x_0,\tilde{K}_0\sigma)$, and for each $B \in \R^2\simeq E_k$, we have
		\begin{align}
		D\Phi_x^{(k)}(0,B+\nu_{0}(x, B)) &= D\Pi_x^c\pi_M D  H_{T, \hat{\gamma}^x(\frac{k}{k_0})}(B + \nu_{0}(x, B)) \nonumber\\
		&= D\Pi_x^c\Big[\hat \pi_c D   H_{T, \hat{\gamma}^x(\frac{k}{k_0})}(B + \nu_{0}(x, B)) + \nu_{0}(x, B)\Big],\label{je ne sais pas}
		\end{align}
		where $\Phi_x^{(k)}:=\Phi^{(k)}(\cdot,x)$, and $\pi_M \colon \R^2 \times M \to M$ denotes the projection onto the second coordinate. 
		
		By Lemma \ref{lemma constr phi} 
		 there exists a constant $D>0$   such that for $\zeta > 0$ small, if $\sigma\in (0,\overline{h}_{\zeta}(f))$,  then  for any $x \in \overline{\Gamma}_f(x_0)\cap \cWc_f(x_0,\tilde{K}_0\sigma)$ and for any $B \in \R^{2}$, it holds 
		\begin{equation}\label{term je ne sais}
		\|D\Pi_x^c\nu_{0}(x, B)\| \leq D\zeta  \|B\|.
		\end{equation}
		If $\zeta>0$ is sufficiently small (depending only on $\kappa_0$),  then for any $\sigma\in (0,\tilde \sigma_0)$, with $\tilde \sigma_0:=\min(\tilde \sigma,\overline{h}_{\zeta}(f))>0$, and for any $x \in \overline{\Gamma}_f(x_0)\cap \cWc_f(x_0,\tilde{K}_0\sigma)$, by \eqref{eqaution controlant det}-\eqref{je ne sais pas}-\eqref{term je ne sais}, we deduce that
		\begin{equation*}
		\big|\det D_b|_{b=0}\big( \Phi_x^{(k)}|_{E_k}\big)\big|> \frac 12 \kappa_0,
		\end{equation*}
		which concludes the proof of point \eqref{point deux de la prop}, for $\tilde \kappa_0:=\Big(\frac 12 \kappa_0\Big)^{k_0}>0$. 
		
		Finally, point \eqref{point un de la prop} is a direct observation. 
\end{proof}

\section{Local accessibility}

Let us fix an integer $r \geq 2$, and let us consider $f \in \mathscr{F}$, where as before, $\mathscr{F}\subset \mathcal{PH}_*^r(M)$ is the set of  $C^r$ dynamically coherent, plaque expansive, partially hyperbolic diffeomorphisms with two-dimensional center, which satisfy some strong bunching condition as in Definition \ref{stron bunch}. 

In this part, we show that  it is possible to make the accessibility class of any non-periodic point open by a $C^r$-small perturbation. First, we explain how to break trivial accessibility classes, and then, we show how to open one-dimensional accessibility classes, based on some transversality arguments.

\subsection{Breaking trivial accessibility classes}\label{break trivial}

\begin{prop}\label{prop break zero}
	For any non-periodic point $x_0 \in M$,  for any  $\delta>0$, and for any $\sigma>0$,  there exists a partially hyperbolic diffeomorphism $g \in \mathscr{F}$ such that $d_{C^r}(f,g)< \delta$ and such that $x_0 \notin\tilde \Gamma_g^0(\sigma)$; in particular, the center accessibility class $\mathrm{C}_g(x_0)$ is at least one-dimensional. 
\end{prop}

\begin{proof}
	Take a non-periodic point $x_0 \in M$. Fix some small $\delta>0$, let $k_0:=1$, and let %$\tilde{K}_0=\tilde{K}_0(f)\in (0,1)$, 
	$\tilde{\sigma}_0=\tilde{\sigma}_0(f,1)>0$, $\tilde R_0=\tilde R_0(f,1)>0$ and  $\tilde{\delta}_0=\tilde{\delta}_0(f,r,\delta)>0$ be the constants given by Corollary \ref{coro un peu chiant}. As $x_0$ is non-periodic, then for $\sigma\in (0,\tilde \sigma_0)$ sufficiently small, it holds $R_{\pm}(f,B(x_0,10\sigma))>\tilde R_0$. Assume that $x_0 \in \tilde \Gamma_f^0(\sigma)$ (otherwise there is nothing to prove). 
	
	By Corollary \ref{coro un peu chiant}, for $n:=n(x_0)=2$, there exist a continuous family 
	\begin{equation}\label{famille gamma zero}
	\gamma=\gamma^{x_0}=\{\gamma(t) = [x_0,x_1(t),x_2(t),x_3(t),x_0]\}_{t \in [0,1]}
	\end{equation} 
	of $4$ us-loops at $(f,x_0)$ such that $\ell(\gamma)< \sigma$, $\gamma(0)$ is trivial, $\gamma(1)$ is a non-degenerate closed $4$ us-loop, a $C^r$ deformation $\hat f \colon B(0_{\R^{2}},\tilde{\delta}_0) \times M \to M$ at $f$ with $2$-parameters, so that  $\mathrm{supp}(\hat f)\subset B(x_0,10\sigma)$, and such that
	 % as in Corollary \ref{coro un peu chiant},
	 the map 
	\begin{equation}\label{def Phi}
	\Phi_{x_0}\colon  B(0_{\R^{2}},\tilde{\delta}_0)\ni b\mapsto  \Pi_x^c\psi(b,x_0,1)
	\end{equation}
	is $C^1$ and satisfies
	\begin{equation}\label{submerssion}
	|\det D_b|_{b=0}\Phi_{x_0}| > \tilde \kappa_0,
	\end{equation}
	for some $C^2$-uniform constant   $\tilde \kappa_0=\tilde \kappa_0(f,1)>0$. 
	Recall that in \eqref{def Phi}, $\Pi^c_x\colon M \to \R^2$ is the map defined in Lemma \ref{lemma constr phi}, $T=T(\hat f)$,  and $\psi=\psi(T,x_0,\gamma^{x_0})$. 
	
	Moreover, by Definition \ref{lift of loop for deformation}, for all $b \in B(0_{\R^{2}},\tilde{\delta}_0)$ and all $t \in [0,1]$, we have $\psi(b,x_0,t) \in \cW_{f_b}^c(x_0)\cap \ac_{f_b}(x_0)=\cac_{f_b}(x_0)$, where $f_b:=\hat f(b,\cdot)$. 
	Besides,  \eqref{submerssion} implies that the map $\psi(\cdot,x_0,1)$ is a submersion in a neighbourhood of $0_{\R^2}$, hence 
	$$
	\psi(\cdot,x_0,1)^{-1}\{x_0\}\cap B(0_{\R^2},\delta_1)=\{0_{\R_2}\},
	$$
	for some sufficiently small $\delta_1\in (0,\tilde \delta_0)$. Fix $b \in B(0_{\R^2},\delta_1) \setminus \{0_{\R^2}\}$ such that $g:=f_b \in\mathscr{F}$; we have
	$
	\psi(b,x_0,1)\in \cac_{g}(x_0)\setminus \{x_0\}
	$, and $[0,1]\ni t \mapsto \psi(b,x_0,t)\in \cac_{g}(x_0)$ is a non-trivial $g$-path connecting $x_0$ to $\psi(b,x_0,1)\neq x_0$ within $\cac_{g}(x_0)$. In particular, $x_0 \notin\tilde \Gamma_g^0(\sigma)$; in fact, by Theorem \ref{theorem structure}, $\cac_{g}(x_0)$ is at least one-dimensional. 
	Moreover, by point \eqref{point un de la prop}  of Corollary \ref{coro un peu chiant}, we have $d_{C^r}(f,g)< \delta$, which concludes the proof. 
\end{proof}

\subsection{Opening one-dimensional accessibility classes}\label{opening one dim}

%Let us consider a partially hyperbolic diffeomorphism $f \in \cPH^r_2(M)$ of class $C^r$ for some $r \geq 2$, that is center bunched, dynamically coherent and plaque expansive, and such that $\dim E_f^c=2$. 

The following result  shows that local accessibility can be achieved near non-periodic points after a $C^r$-small perturbation.

\begin{prop}\label{main prop bis}
	For any non-periodic point $x_0 \in M$,  and for any  $\delta>0$,   there exists a partially hyperbolic diffeomorphism $g \in \mathscr{F}$ such that $d_{C^r}(f,g)< \delta$ and such that the accessibility class $\ac_g(x_0)$ is open. 
\end{prop}

\begin{proof}
	Let us  consider  a non-periodic point $x_0 \in M$. Let $k_0:=1$, and let $\tilde{\sigma}_0=\tilde{\sigma}_0(f,1)>0$, $\tilde R_0=\tilde R_0(f,1)>0$ be the constants given by Corollary \ref{coro un peu chiant}. As $x_0$ is non-periodic, we can fix $\sigma\in (0,\tilde \sigma_0)$ such that $R_{\pm}(f,B(x_0,10\sigma))>\tilde R_0$. 
	
	Assume by contradiction that $\ac_f(x_0)$ is stably not open in the $C^r$ topology.  In other words, by Theorem \ref{theorem structure}, for some  $\delta_1>0$, and for all diffeomorphism $g \in \cPH_*^r(M)$  such that $d_{C^r}(f,g)<\delta_1$, we have  $x_0 \in \Gamma_g$.  Fix some small $\delta \in (0,\delta_1)$. By Proposition \ref{prop break zero}, there exists a  $C^r$ diffeomorphism $f_0 \in \mathscr{F}$ with  $d_{C^r}(f,f_{0})< \frac{\delta}{2}$ such that  $x_0 \not\in \tilde\Gamma_{f_0}^0(\frac{\sigma}{2})$. In particular, by our choice of $\delta$, we have $x_0 \in \Gamma_{f_0}^1\setminus \tilde\Gamma_{f_0}^0(\frac{\sigma}{2})$. 
	Besides, there exists $\delta_2 \in (0,\frac{ \delta}{2})$ such that for any diffeomorphism $g$ satisfying $d_{C^r}(f_0,g)<\delta_2$, we have $g \in \mathscr{F}$, and $x_0 \in \Gamma_{g}^1\setminus \tilde\Gamma_{g}^0(\sigma)$. In particular, with the notations in Section \ref{section one dim},   $\mathcal{U}^\mathscr{F}=\mathcal{U}_1^\mathscr{F}(x_0)$, where $\mathcal{U}$ is a $\delta_2$-neighbourhood of $f_0$ in the $C^r$ topology. 
	
	By Corollary \ref{coro un peu chiant}, for $n:=n(x_0)=5$, there exist a continuous family 
	\begin{equation}\label{famille gamma zero longu}
	\gamma=\gamma^{x_0}=\{\gamma(t) = [x_0,x_1(t),x_2(t),\dots,x_{10}(t)]\}_{t \in [0,1]}
	\end{equation} 
	of $10$ us-loops at $(f_0,x_0)$ such that $\ell(\gamma)< \sigma$, $\gamma(0)$ is trivial, $\gamma(1)$ is a non-degenerate closed $10$ us-loop, a $C^r$ deformation $\hat f \colon B(0_{\R^{2}},\tilde{\delta}_0) \times M \to M$ at $f_0$ with $2$-parameters,  with $\tilde \delta_0=\tilde \delta_0(f_0,r,\delta)>0$,  so that for the map $\Pi^c_{x_0}\colon M \to \R^2$ defined in Lemma \ref{lemma constr phi},  $T=T(\hat f)$,  and $\psi=\psi(T,x_0,\gamma^{x_0})$,  the map $\Phi_{x_0}\colon  B(0_{\R^{2}},\tilde{\delta}_0)\ni b\mapsto  \Pi_{x_0}^c\psi(b,x_0,1)$ is $C^1$, and for some constant $\tilde \kappa_0=\tilde \kappa_0(f,1)>0$, it holds 
	\begin{equation}\label{submerssion bis i}
	|\det D_b|_{b=0}\Phi_{x_0}| > \tilde \kappa_0.
	\end{equation}
	
	Fix some small $\theta>0$. It follows from the previous discussion and Proposition \ref{lemma var des classes} that for  $\delta_0\in (0,\tilde \delta_0)$, $\varepsilon_0>0$ sufficiently small, then for all $b \in B(0_{\R^2},\delta_0)$, the diffeomorphism $f_b:=\hat f(b,\cdot)$ satisfies $f_b\in \mathcal{U}_1^\mathscr{F}(x_0)$, and 
	\begin{align}\label{var clases}
	\cac_{f_b}(x_0,\varepsilon_0)\subset \mathscr{C}_{f_0}(x_0,\theta,\varepsilon_0),
	\end{align}
	where $\mathscr{C}_{f_0}(x_0,\theta,\varepsilon_0)$ is as in \eqref{def connne}. Let us set %$\hat \delta_0:=\frac{1}{2}\min(\delta_0,\varepsilon_0,\hat \delta)$, and
	$$
	\mathscr{C}(x_0,\theta):=\Pi_{x_0}^c \big( \mathscr{C}_{f_0}(x_0,\theta,\varepsilon_0)\big). %\cap \pi_M \cWc_T((0,x),\hat \delta_0)\big).
	$$ 
	
%	\begin{figure}[H]
%		\begin{center}
%			\includegraphics [width=14cm]{Centerleaf2.pdf}
%			\caption{Variation of one-dimensional center accessibility classes.}
%		\end{center}
%	\end{figure}

	By Definition \ref{lift of loop for deformation}, and since $\psi(0,x_0,1)=x_0$,\footnote{Recall that $\gamma(1)$ is a closed $10$ us-loop at $(f_0,x_0)$.} for $\tilde\delta_1\in (0,\delta_0)$ sufficiently small, %after taking possibly $\delta_0>0$ even smaller,  
	we have $\psi(b,x_0,1)\in \mathrm{C}_{f_b}(x_0,\varepsilon_0)$, for all $b \in B(0_{\R^{2}},\tilde\delta_1)$, and by \eqref{var clases}, we deduce that %it holds 
	$$
	\Phi_{x_0}\big(B(0_{\R^{2}},\tilde{\delta}_1)\big)\subset \Pi_x^c\big(\mathrm{C}_{f_b}(x_0,\varepsilon_0)\big)\subset \mathscr{C}(x_0,\theta).
	$$
	%a contradiction. Indeed, on the one hand, b
	On the one hand, by the definition of the cone $\mathscr{C}(x_0,\theta)$, 
	we have $\R^2\setminus \Phi_{x_0}\big(B(0_{\R^{2}},\tilde{\delta}_1)\big)\supset \Delta_0$ for some straight line $\Delta_0$ through the origin $0_{\R^2}$. But on the other hand, it follows from \eqref{submerssion bis i} that $\Phi_{x_0}\big(B(0_{\R^{2}},\tilde{\delta}_1)\big)$ contains an open neighbourhood of $0_{\R^2}$, a contradiction. By Theorem \ref{theorem structure}, we conclude that for some $b \in B(0_{\R^2},\tilde{\delta}_0)$, $\ac_{f_b}(x_0)$ is open; moreover, by construction, $g:=f_b\in \mathscr{F}$ satisfies   
	$$
	d_{C^r}(f,g) \leq d_{C^r}(f,f_0)+d_{C^r}(f_0,f_b)<\frac \delta 2+\frac \delta 2=\delta,
	$$ 
	which concludes the proof. 
\end{proof}

\section{$C^r$-density of accessibility}\label{setion dens}

In this part, we conclude the proof of our main results stated in Section \ref{section main resul}. As above, we fix an integer $r \geq 2$, and let $f \in \mathscr{F}$, where  $\mathscr{F}\subset \mathcal{PH}_*^r(M)$ is the set of  $C^r$ dynamically coherent, plaque expansive, partially hyperbolic diffeomorphisms with two-dimensional center, which satisfy some strong bunching condition as in Definition \ref{stron bunch}. Our  goal is to conclude the proof of our main result (Theorem \ref{premier theo}):

\begin{prop}\label{main prop}
	For any $\delta>0$, there exists a partially hyperbolic diffeomorphism $g\in\mathscr{F}$ with $d_{C^r}(f,g)< \delta$ such that $g$ is stably accessible.
\end{prop}

\subsection{Spanning $c$-families}\label{spanning famm}

For the proof of Proposition \ref{main prop}, we combine ideas from the last section with some global argument; this is done by means of spanning families of center-disks; this notion was already present in the work of Dolgopyat-Wilkinson \cite{DW} and is also used in \cite{LZ}.  

\begin{defi}[$c$-disk] %\footnote{ write it better }
	For each $x \in M$ and $\sigma>0$, $\cC = \cWc_f(x,\sigma)$ is called the \emph{center disk} of $f$ (or $c$-disk of $f$ for short) centered at $x$ with radius $\sigma$. We set $\varrho(\cC):=\sigma$, and for any $\theta \in (0, 1 ]$, we also define $\theta \cC:= \cWc_f(x, \theta \sigma)$.
\end{defi}

\begin{defi}
	A collection of disjoint center disks $\mathcal{D} = \{\cC_1, \dots, \cC_J \}$ is called a \emph{family of center disks for $f$} (or \emph{$c$-family for $f$} for short). In addition, we set 
	$$
	\underline{r}(\mathcal{D}):=\inf_{\cC\in \mathcal{D}}\{\varrho(\cC)\},\quad \overline{r}(\mathcal{D}):=\sup_{\cC\in \mathcal{D}}\{\varrho(\cC)\}. 
	$$
	Given $\theta \in (0,1)$  and %an integer 
	$k \geq 1$, we say that $\mathcal{D}$ is a \emph{$(\theta, k)$-spanning} $c$-family for $f$ if 
	\begin{equation*}
	M = \bigcup_{\cC \in \mathcal{D}} \bigcup_{x \in \theta\cC} Acc_{f}(x, k),
	\end{equation*}
	where $Acc_{f}(x, k)$ denotes the set of all points $y \in M$ which can connected to $x$ by a $f$-accessibility sequence with at most $k$ legs of length less than one.
	
	Given any subset $\mathcal{C} \subset M$, and $\sigma\geq 0$, we set $(\mathcal{C},\sigma):=\{x \in M\, \vert\, d(x,\mathcal{C})\leq \sigma\}$. 
	Given $\sigma\geq 0$ and  a $c$-family  $\mathcal{D} = \{\cC_1, \dots, \cC_J\}$ for $f$, we set %\marginpar{to avoid a new line} 
	%$\cD = \{\cC_1, \dots, \cC_K\}$, a collection of subsets of $X$,  \{1,\dots,K\}$:
	\begin{equation*}
	(\mathcal{D}, \sigma):= \bigcup_{j=1}^J (\cC_j, \sigma).
	\end{equation*}
	
	We say that  $\mathcal{D}$ is  \emph{$\sigma$-sparse} if for any two distinct $\cC, \cC' \in \mathcal{D}$, $(\cC, \sigma), (\cC', \sigma)$ are disjoint. Any $c$-family for $f$ is $\sigma$-sparse for some $\sigma>0$. 
\end{defi}

\begin{prop}[Corollary 6.2, \cite{LZ}]\label{prop acc mod den dis}  
	Assume that $f \in \cPH^1(M)$ is dynamically coherent, plaque expansive, and that the fixed points of $f^{k}$ are isolated for all $k\geq 1$.  Then for every $\overline{R} > 1$, there exist $C^1$-uniform constants $\overline{N} = \overline{N}(f,\overline{R})>0$, $\overline{\rho} = \overline{\rho}(f, \overline{R})\in(0,\overline{R}^{-1})$ and $\overline{\sigma} = \overline{\sigma}(f,\overline{R})>0$ such that the following is true. For all diffeomorphism $g$ sufficiently $C^1$-close to $f$, there exists   a $(\frac{1}{40},  4)$-spanning $c$-family $\cD_g$ for $g$ with at most $\overline{N}$ elements such that 
	\begin{enumerate}
		\item $\overline{\rho}<\underline{r}(\cD_g)\leq \overline{r}(\cD_g) < \overline{R}^{-1}$;
		\item $\cD_g$ is $\overline{\sigma}$-sparse;
		\item $R_\pm(g, (\cD_g, \overline{\sigma})) > \overline{R}$.
	\end{enumerate}
	Moreover, the map $g \mapsto \mathcal{D}_g$ can be chosen to be continuous. 
\end{prop}

\subsection{Density of diffeomorphisms with no trivial accessibility class}\label{break trivial global}

The following result strengthens Proposition \ref{prop break zero}. 
\begin{prop}\label{prop break zero global}
	There exist $C^2$-uniform constants  $\tilde{\sigma}_1=\tilde{\sigma}_1(f)>0$, $\tilde K_1=\tilde K_1(f)\in (0,1)$ and $\tilde R_1=\tilde R_1(f)>0$ such that for any $\delta>0$, for any $\sigma\in (0,\tilde{\sigma}_1)$, for any point $x_0\in M$ satisfying $R_{\pm}(f,B(x_0,10\sigma))>\tilde R_1$, there exists a partially hyperbolic diffeomorphism $g \in \mathscr{F}$ such that $d_{C^r}(f,g)< \delta$ and such that for some $\delta'=\delta'(x_0,g)>0$, we have $x \notin\tilde \Gamma_h^0(\sigma)$, for all $x \in \mathcal{W}_f^c(x_0,\tilde K_1\sigma)$ and for all $h \in \mathscr{F}$ with $d_{C^1}(g,h)<\delta'$. In particular, the center accessibility class $\mathrm{C}_h(x)$ of each point $x \in \mathcal{W}_f^c(x_0,\tilde K_1\sigma)$ is at least one-dimensional. 
\end{prop}

\begin{remark}
	In order to deal with all the points in a given center disk, the idea is to increase the codimension of ``bad'' configurations; this is done by considering two $4$ us-loops at each point in the center disk, and show that we can construct a perturbation in such a way that for each of those points, at least one of the endpoints of the $4$ us-loops is not the original point. 
\end{remark}

\begin{proof}

Fix some small $\delta>0$, let $k_0:=2$, and let $\tilde{K}_1:=\tilde{K}_0(f)\in (0,1)$, 
$\tilde{\sigma}_1:=\tilde{\sigma}_0(f,2)>0$, $\tilde R_1:=\tilde R_0(f,2)>0$ and  $\tilde{\delta}_1:=\tilde{\delta}_0(f,r,\delta)>0$ be the constants given by Corollary \ref{coro un peu chiant}. Let us take $\sigma\in (0,\tilde{\sigma}_1)$, and let us consider a point $x_0\in \tilde \Gamma_f^0(\sigma)$ satisfying $R_{\pm}(f,B(x_0,10\sigma))>\tilde R_1$. 

By Corollary \ref{coro un peu chiant}, for $n:=n(x_0)=2$, there exists a continuous map $\tilde\Gamma_f^0(\sigma)\cap \cWc_f(x_0,\tilde{K}_1\sigma)\ni x \mapsto \gamma^x$ such that $\gamma^x=\{\gamma^x(t) = [x,x_1^x(t),x_{2}^x(t),x_3^x(t),x]\}_{t \in [0,1]}$ is a continuous family of $4$ us-loops at $(f,x)$, with $\ell(\gamma^x)< \sigma$,  such that $\gamma^x(0)$ is trivial, for $k=1,2$,  $\gamma^x(\frac{k}{2})$ is a non-degenerate closed us-loop, and there exists a $C^r$ deformation $\hat f \colon B(0_{\R^{4}},\tilde{\delta}_1) \times M \to M$ at $f$ with $4$-parameters, so that  $\mathrm{supp}(\hat f)\subset B(x_0,10\sigma)$, and %for  $T=T(\hat f)$, for $\widehat{\psi}_x:=\widehat{\psi}(T,x,\gamma^x)$, $\psi_x:=\pi_M \widehat{\psi}_x$ as in \eqref{widehatpsidefi}-\eqref{relationwidehatpsipsifxgamma},  and for any $b\in \subset \R^{4}$, 
such that the map 
\begin{equation*}
\Phi\colon\left\{
\begin{array}{rcl}
B(0_{\R^{4}},\tilde{\delta}_1) \times \Big(\tilde\Gamma_f^0(\sigma)\cap \cWc_f(x_0,\tilde{K}_1\sigma)\Big) & \to & \R^{4}=\R^2\times \R^2\\
(b,x) & \mapsto & (\Phi^{(1)}(b,x),\Phi^{(2)}(b,x))
\end{array}
\right.
\end{equation*}
is continuous and satisfies
	$$
	|\det D_b|_{b=0}\big(\Phi(\cdot,x)\big)| > \tilde \kappa_0,
	$$
	for some $C^2$-uniform constant   $\tilde \kappa_0=\tilde \kappa_0(f,2)>0$. 
Recall that   $\Pi_x^c\colon M \to \R^2$ is the map given by Lemma \ref{lemma constr phi}, $T=T(\hat f)$, $\psi_x=\psi(T,x,\gamma^{x})$, and 
$\Phi^{(k)}(\cdot,x):=\Pi_x^c\psi_x(\cdot,x,\frac{k}{2})$, for $k=1,2$. 

 By Lemma \ref{lemma continuation}, we can extend the map $x \mapsto \gamma^x=\{\gamma^x(t)\}_{t \in [0,1]}$ to all the points $x \in \cWc_f(x_0,\tilde{K}_1\sigma)$ (note that for $x \in  \cWc_f(x_0,\tilde{K}_1\sigma)\setminus \tilde\Gamma_f^0(\sigma)$, the us-loops $\gamma^x(\frac 12),\gamma^x(1)$ may not be closed). Considering the associated maps $\psi_x=\psi(T,x,\gamma^{x})$ and 
 $\Phi^{(k)}(\cdot,x)=\Pi_x^c\psi_x(\cdot,x,\frac{k}{2})$, for $k=1,2$, we can thus extend $\Phi$  to a map 
 \begin{equation*}
 \Phi\colon\left\{
 \begin{array}{rcl}
 B(0_{\R^{4}},\tilde{\delta}_1) \times   \cWc_f(x_0,\tilde{K}_1\sigma)  & \to & \R^{4}=\R^2\times \R^2\\
 (b,x) & \mapsto & (\Phi^{(1)}(b,x),\Phi^{(2)}(b,x))
 \end{array}
 \right.
 \end{equation*}
 such that
 %$\Phi\colon B(0_{\R^{4}},\tilde{\delta}_1) \times \cWc_f(x_0,\tilde{K}_1\sigma)  \to  \R^{4}$, 
% $(b,x)  \mapsto (\Phi^{(1)}(b^1,x),\Phi^{(2)}(b^2,x))$, such that 
$$
|\det D_b|_{b=0}\big(\Phi(\cdot,x)\big)| >\frac 12  \tilde \kappa_0. 
$$

	Take $\hat \delta>0$ suitably small, and let 
	\begin{equation*}
	\Psi \colon\left\{
	\begin{array}{rcl}
	\cW_T^c((0,x_0),\hat \delta) &\to& \R^6=\R^2\times \R^2\times \R^2,\\
	(b,x)  &\mapsto& \big(\Pi_x^c(x),\Phi^{(1)}(b,x),\Phi^{(2)}(b,x)\big).
	\end{array}
	\right.
	\end{equation*}
	%If $\overline{R}$ is chosen sufficiently large, then f
	For any point $x \in\cWc_f(x_0,\tilde{K}_1\sigma)$, the map $\Phi(\cdot,x)$ is a submersion, and thus, the map $\Psi$ is uniformly transverse to the diagonal 
	$$
	\Sigma_0:= \{(z,z,z):z \in \R^2\}\subset \R^6.
	$$
	Therefore, $\Psi^{-1}(\Sigma_0)$ is a submanifold of codimension $4$. Let $\pi_B \colon \cW_T^c((0,x_0),\hat \delta) \to \R^4$, $(b,x)\mapsto b$. Let $b\in B(0_{\R^{4}},\tilde{\delta}_1)\setminus \pi_B \big(\Psi^{-1}(\Sigma_0)\big)$, and let $g:=f_b:=\hat f(b,\cdot)\in \mathscr{F}$. Then, for any $x \in\cWc_f(x_0,\tilde{K}_1\sigma)$, we have $\Psi((b,x))\notin \Sigma_0$, i.e., $\psi_x(b,x,\frac 12)\in \cac_{g}(x)\setminus \{x\}$ or $\psi_x(b,x,1)\in \cac_{g}(x)\setminus \{x\}$. We conclude that $x \notin \tilde \Gamma_g^0(\sigma)$.
	
	Actually, the same holds for any diffeomorphism $h$ that is sufficiently $C^1$-close to $g$. Indeed, for any $x \in \cWc_f(x_0,\tilde{K}_1\sigma)$, let $\gamma_{1}^x$, $\gamma_{2}^x$ be the $4$ us-loops at $(g,x)$ coming from $\gamma^x(\frac 12)$, $\gamma^x(1)$,  with respective endpoints $\psi_x(b,x,\frac 12)$, $\psi_x(b,x,1)$. For any diffeomorphism $h$ which is $C^1$-close to $g$, we let $\gamma_1^{x,h}$, $\gamma_2^{x,h}$ be the respective continuations of $\gamma_{1}^x$, $\gamma_{2}^x$ given by Lemma \ref{lemma continuation}, and we set 
	$$
	\widetilde \Psi(h,x):=\big(\Pi_x^c (x),\Pi_x^c H_{h,\gamma_1^{x,h}}(x),\Pi_x^c H_{h,\gamma_2^{x,h}}(x)\big).
	$$
	By our choice of $b$, and by compactness, there exists $\varepsilon_0>0$ such that 
	$$
	d(\widetilde \Psi(g,x),\Sigma_0)=d(\Psi(b,x),\Sigma_0)> \varepsilon_0,
	$$
	for all $x \in\cWc_f(x_0,\tilde{K}_1\sigma)$. Thus, there exists $\delta'>0$ such that for any diffeomorphism $h$ with $d_{C^1}(g,h)<\delta'$, and for any $x \in\cWc_f(x_0,\tilde{K}_1\sigma)$, it holds 
	$$
	d(\widetilde \Psi(h,x),\Sigma_0)> \frac{\varepsilon_0}{2}>0.
	$$
	Therefore, $H_{h,\gamma_1^{x,h}}(x)\in \cac_{h}(x)\setminus \{x\}$ or $H_{h,\gamma_2^{x,h}}(x)\in \cac_{h}(x)\setminus \{x\}$, so that $x \notin \tilde \Gamma_h^0(\sigma)$, which concludes the proof. 
%	Arguing similarly for the other $c$-disks in $\mathcal{D}$, and as $g \mapsto \mathcal{D}_g$ is continuous, we deduce that there exists a diffeomorphism $g \in \cPH^r(M)$ such that 
%	$
%	d_{C^r}(f,g)< \delta
%	$
%	and such that for any $\mathcal{C}\in \mathcal{D}_g$ and for any $y \in \mathcal{C}$, $\cac_{g}(y)$ is non-trivial.  As the $c$-family $\mathcal{D}_g$ is $(\frac{1}{40},4)$-spanning, this concludes the proof. 
\end{proof}

We can now give the proof of Theorem \ref{deuxieme theo}. 

\begin{corollary}\label{prop break zero global bis}
	There exists a $C^2$-uniform constant $\hat{\sigma}_1=\hat{\sigma}_1(f)>0$ such that for any $\sigma \in (0,\hat{\sigma}_1)$, and for any $\delta>0$,  %there exist $C^2$-uniform constants  $\tilde{\sigma}_1=\tilde{\sigma}_1(f)>0$, $\tilde K_1=\tilde K_1(f)\in (0,1)$, $\tilde R_1=\tilde R_1(f)>0$ and  $\tilde{\delta}_1=\tilde{\delta}_1(f,r,\delta)>0$ such that for any $\sigma\in (0,\tilde{\sigma}_1)$, for any point $x_0\in M$ satisfying $R_{\pm}(f,B(x_0,10\sigma))>\tilde R_1$, 
	there exists a partially hyperbolic diffeomorphism $g \in \mathscr{F}$ such that $d_{C^r}(f,g)< \delta$ and such that   %for any point $x \in B(x_0,\tilde K_1\sigma)$, we have $x \notin\tilde \Gamma_g^0(\sigma)$; in particular, 
	for some $(\frac{1}{40},  4)$-spanning $c$-family  $\cD_g$ for $g$, it holds $x \notin \tilde \Gamma_g^0(\sigma)$, for all  $\mathcal{C}\in \mathcal{D}_g$, and for all $x \in \frac{1}{20}\mathcal{C}$. 
	In particular, the center accessibility class $\mathrm{C}_g(x)$ of each point $x \in M$ is non-trivial. 
\end{corollary}

\begin{proof}
Fix some small  $\delta>0$. By Kupka-Smale's Theorem (see for instance \cite{K}), $C^r$-generically, periodic points are hyperbolic. Therefore, without loss of generality, we can assume that the fixed points of $f^{k}$ are isolated, for all $k\geq 1$. 

Let $\tilde{\sigma}_1=\tilde{\sigma}_1(f)>0$, $\tilde K_1=\tilde K_1(f)\in (0,1)$ and $\tilde R_1=\tilde R_1(f)>0$ be the constants given by Proposition \ref{prop break zero global}.  
For $\overline{R}>\max(\tilde R_1,1)$, let $\overline{N} = \overline{N}(f,\overline{R})>0$, $\overline{\rho} = \overline{\rho}(f,\overline{R})\in(0,\overline{R}^{-1})$ and $\overline{\sigma} = \overline{\sigma}(f,\overline{R})>0$ be the constants given by Proposition \ref{prop acc mod den dis}. Then,  there exists a constant $\delta_0'\in (0,\delta)$ such that for any diffeomorphism $g$ with $d_{C^1}(f,g)<\delta_0'$, there exists a $(\frac{1}{40},  4)$-spanning $c$-family $\cD_g$ for $g$ with at most $\overline{N}$ elements such that the map $g \mapsto \mathcal{D}_g$ is continuous, and 
\begin{equation*}
\overline{\rho}<\underline{r}(\cD_g)\leq \overline{r}(\cD_g) < \overline{R}^{-1};\qquad
\cD_g\text{ is }\overline{\sigma}\text{-sparse};\qquad
R_\pm(g, (\cD_g, \overline{\sigma})) > \overline{R}.
\end{equation*}
Take $\sigma \in \big(0,\min\big(\tilde{\sigma}_1,\frac{\overline{\sigma}}{10}\big)\big)$, and let $z_1,z_2,\dots,z_\ell$, $\ell \geq 1$, be a finite collection of points such that for any diffeomorphism $g$ with $d_{C^1}(f,g)<\delta_0'$, we have $g \in \mathscr{F}$, and  
\begin{equation}\label{slow rec fam}
\bigcup_{\mathcal{C}\in \mathcal{D}_g} \frac{1}{20} \cC\subset \bigcup_{i=1}^\ell \mathcal{W}_f^c(z_i,\tilde K_1\sigma)\subset (\cD_g, 10 \sigma). 
\end{equation}
As $\sigma\in (0,\tilde{\sigma}_1)$ and $R_{\pm}(f,B(z_1,10\sigma))>\tilde R_1$, we can apply Proposition \ref{prop break zero global} to get a diffeomorphism $f_1 \in \mathscr{F}$ such that for some $\delta_1'\in (0,\delta'(z_1,f_1))$, it holds $B_{C^r}(f_1,\delta_1')\subset B_{C^r}(f,\delta_0')$, %such that $d_{C^r}(f,f_1)<\delta_1'$ and such that for some $\delta_1'\in (0, \frac{\delta_1}{\ell})$, we have
and $x \notin\tilde \Gamma_{h}^0(\sigma)$, for all $x \in \mathcal{W}_f^c(z_1,\tilde K_1\sigma)$ and for all $h \in B_{C^r}(f_1,\delta_1')$. %\mathscr{F}$ with $d_{C^1}(f_1,h)<\delta_1'$.  

Similarly, as $R_{\pm}(f_1,B(z_2,10\sigma))>\tilde R_1$, we can apply Proposition \ref{prop break zero global} to get a diffeomorphism $f_2 \in \mathscr{F}$ %such that $d_{C^r}(f_1,f_2)< \min \big(\frac{\delta_1}{\ell},\delta_1'\big)$ and 
such that for some $\delta_2'>0$, it holds $B_{C^r}(f_2,\delta_2')\subset B_{C^r}(f_1,\delta_1')\subset B_{C^r}(f,\delta_0')$, and $x \notin\tilde \Gamma_h^0(\sigma)$, for all $x \in \mathcal{W}_f^c(z_2,\tilde K_1\sigma)$ and for all $h \in B_{C^r}(f_2,\delta_2')$; %$h \in \mathscr{F}$ with $d_{C^1}(f_2,h)<\delta_2'$;
in fact, as $B_{C^r}(f_2,\delta_2')\subset B_{C^r}(f_1,\delta_1')$, we have $x \notin\tilde \Gamma_h^0(\sigma)$, for all $x \in \mathcal{W}_f^c(z_1,\tilde K_1\sigma)\cup \mathcal{W}_f^c(z_2,\tilde K_1\sigma)$.  

Recursively, we thus obtain a diffeomorphism $g=f_\ell \in \mathscr{F}$ such that $d_{C^r}(f,g)<\delta_0'<\delta$ and such that $x \notin\tilde \Gamma_g^0(\sigma)$, for all $x \in \mathcal{W}_f^c(z_1,\tilde K_1\sigma)\cup \dots \cup \mathcal{W}_f^c(z_\ell,\tilde K_1\sigma)$.  By \eqref{slow rec fam}, we conclude that for each $\mathcal{C}\in \mathcal{D}_g$, and for each $x \in \frac{1}{20}\mathcal{C}$, we have  $x \notin \tilde\Gamma_g^0(\sigma)$. In particular, as $\cD_g$ is a $(\frac{1}{40},  4)$-spanning $c$-family for $g$, the center accessibility class $\mathrm{C}_g(x)$ of each point $x \in M$ is non-trivial. 
\end{proof}

\begin{remark}
	In fact, Corollary \ref{prop break zero global bis} also holds when the center dimension $\dim E_f^c$ is larger than $2$. Indeed, the proof  relies on the submersion from the space of perturbations to the phase space -- here, some center leaf -- constructed in Lemma \ref{lem def regular family and chart}  and Corollary \ref{coro un peu chiant}, which can be carried out also when $\dim E_f^c>2$. 
\end{remark}

\subsection{Density of accessibility}\label{opening one dim global}

In this part, we conclude the proof of Proposition \ref{main prop}. Let us start with the following result,  which strengthens Proposition \ref{main prop bis}. 
\begin{prop}\label{prop open global}
	There exist $C^2$-uniform constants  $\tilde{\sigma}_2=\tilde{\sigma}_2(f)>0$, $\tilde K_2=\tilde K_2(f)\in (0,1)$ and $\tilde R_2=\tilde R_2(f)>0$ such that for any $\delta>0$, for any $\sigma\in (0,\tilde{\sigma}_2)$, for any point $x_0\in M$ satisfying $R_{\pm}(f,B(x_0,10\sigma))>\tilde R_2$, there exists a partially hyperbolic diffeomorphism $g \in \mathscr{F}$ such that $d_{C^r}(f,g)< \delta$ and such that for some $\delta''=\delta''(x_0,g)>0$, it holds $\ac_h(x_0)\supset B(x_0,\tilde K_2\sigma)$, for all $h \in \mathscr{F}$ with $d_{C^1}(g,h)<\delta''$.  
\end{prop}

%Let us now conclude the proof of Proposition \ref{prop open global}. 
\begin{proof}%[Proof of Proposition \ref{prop open global}]
	Fix some small $\delta>0$. 
	%Let $\tilde{K}_0=\tilde{K}_0(f)\in (0,1)$, $\tilde{\sigma}_0=\tilde{\sigma}_0(f,1)>0$ and $\tilde R_0=\tilde R_0(f,1)>0$ be the constants given by Corollary \ref{coro un peu chiant} for $k_0=1$. 
	Let $\tilde{\sigma}_1=\tilde{\sigma}_1(f)>0$, $\tilde K_1=\tilde K_1(f)\in (0,1)$ and $\tilde R_1=\tilde R_1(f)>0$ be the constants in Proposition \ref{prop break zero global}. Let $x_0\in M$ be a point satisfying $R_{\pm}(f,B(x_0,10\sigma))>\widetilde{R}$, for some  $\widetilde{R}>\tilde R_1$ and $\sigma \in (0,\tilde{\sigma}_1)$, and take $\widetilde{K} \in (0,\tilde K_1)$. Then, by Proposition \ref{prop break zero global}, there exists a partially hyperbolic diffeomorphism $f_1 \in \mathscr{F}$ such that $d_{C^r}(f,f_1)< \frac{\delta}{2}$ and such that for some $\delta'\in (0,\frac\delta 2)$, we have $x \notin\tilde \Gamma_g^0(\sigma)$, for all $x \in B(x_0,\widetilde{K}\sigma)$ and for all $g \in \mathscr{F}$ with $d_{C^1}(f_1,g)<\delta'$. 
	
	%Without loss of generality, we assume that $x_0 \notin\tilde \Gamma_{f_1}^0(\sigma)$. 
	
	In the following, for any $x \in \Gamma_{g}^1\cap\cWc_{f_1}(x_0,\widetilde{K}\sigma)$, we denote by
	$\Pi_x^c\colon M \to \R^2$ the map in Lemma \ref{lemma constr phi} for the diffeomorphism $f_1$. 
	By Proposition \ref{prop lamination} and Proposition \ref{lemma var des classes}, if $\delta'$ is sufficiently small, then for any $g \in \mathscr{F}$ with $d_{C^1}(f_1,g)<\delta'$  and for any $x \in \Gamma_{g}^1\cap\cWc_{f_1}(x_0,\widetilde{K}\sigma)$, it holds
	\begin{equation}\label{cone cone}
	\Pi_x^c\cac_{g}(x,10\sigma)\subset \mathscr{C}_1,
	\end{equation}
	for some cone $\mathscr{C}_1$ centered at $0_{\R^2}$; as in Section \ref{section adapt loop}, we let $\mathscr{C}:=\big(\R^2  \setminus \mathscr{C}_1\big) \cup \{0_{\R^2}\}$, and let $\mathscr{C}^+$, $\mathscr{C}^-$ be the closures of the two connected components of $\mathscr{C}\setminus \{0_{\R^2}\}$. For any $x \in \cWc_{f_1}(x_0,\widetilde{K}\sigma)$, we let $\gamma_1^x=[x,\alpha_{1}^x,\dots,\omega_{1}^x]$, resp. $\gamma_2^x=[x,\alpha_{2}^x,\dots,\omega_{2}^x]$  be the non-degenerate closed $10$ us-loop, resp. non-degenerate closed $10$ su-loop at $(f_1,x)$ given by Lemma \ref{lemme de construction des loops} for $f_1$ in place of $f$, with 
%	$$
%	\Pi_x^c\omega_\star^x\in \mathscr{C}.
%	$$
%	More precisely, it holds 
\begin{equation}\label{good cone plus moins}
\big(\Pi_x^c\omega_1^x,\Pi_x^c\omega_2^x\big)\in \big(\mathscr{C}^+ \times \mathscr{C}^-\big)\cup \big(\mathscr{C}^- \times \mathscr{C}^+\big).
\end{equation}
	
	In the following,  we will define a new deformation $\hat f$ obtained by considering infinitesimal deformations localized near the points $\alpha_1^x$ and $\alpha_2^x$ for $x \in \cWc_{f_1}(x_0,\widetilde{K}\sigma)$. %, in such a way that we have a submersion property for the perturbed holonomies along $\gamma_1^x$ and $\gamma_2^x$. 
	%Similarly to Definition \ref{defiadaptedtosth}, for any $x' \in M$, $\sigma'>0$, $n \geq 2$, any $2n$ su-loop $\gamma'=[x',x_1',x_2',\dots,x_{2n}']$ at $(f_1,x')$, with  $\ell(\gamma')<\sigma'$, and given two constants $C',R_0' > 0$, we say that an infinitesimal  $C^{r}$ deformation $V'$  is \emph{adapted to} $(\gamma', \sigma', C',R_0')$ if 
	%\begin{enumerate}
	%\item  $\sigma'\|\partial_b\partial_xV'\|_{M} + \|\partial_bV'\|_{M} < C'$;
	%\item $R(f_1, \{z\}, \mathrm{supp}(V')) > R_0'$ for $z= x', x_1',x_3', \dots,x_{2n}'$;
	%\item $R_{\pm}(f_1, \{x_2'\},  \mathrm{supp}(V')) > R_0'$.
	%\end{enumerate}
	Arguing as in Lemma \ref{lem def regular family and chart}, for $\star=1,2$, we can construct an infinitesimal $C^r$ deformation at $f_1$ with $2$-parameters $V_\star \colon \R^{2}\times M \to T M$ such that $\mathrm{supp}(V_\star)\subset B(x_0,10\sigma)$, and such that for some constants $\widetilde C>0$, $\tilde \kappa>0$, we have: for any $x\in \cWc_{f_1}(x_0,\widetilde{K}\sigma)$, 
	\begin{enumerate}
		\item $V_\star$ is adapted to $(\gamma^x_\star, \sigma,\widetilde C,\widetilde{R})$; %, for any integer $k \in \{1,\dots,k_0\}$;
		\item\label{point deux du lemme}
		%for any integer $k \in \{1,\dots,k_0\}$, 
		for any corner $z\neq \alpha_\star^x$ of $\gamma_\star^x$,  it holds 
		\begin{equation*} 
		D_{B}(\pi_cV_\star(B, z)) = 0, %\quad \forall z \in \cWc_f(x_{i,j,1}, \cg \Lambda_f^{4c}\sigma\clb),
		\end{equation*}
		where $\pi_c \colon TM  \to E^{c}_{f}$ denotes  the canonical projection, and  
		\begin{equation*} 
		\Big|\det D_{B}\big(\pi_cV_\star\big(B, x_1^x\big(\frac{k}{k_0}\big)\big)\big)\Big| > \tilde \kappa.
		\end{equation*}
	\end{enumerate} 
Indeed, for $\star=1,2$, as the map $\cWc_{f_1}(x_0,\widetilde{K}\sigma) \ni x \mapsto \gamma^x_\star$ is continuous, and by \eqref{distance etre points}, we can construct the infinitesimal deformation $V_\star$ such that the $\mathrm{supp}(V)$ is localized around the set $\{\alpha_\star^x\}_x$ of the first corners of the loops $\gamma^x_\star$.
	
Let then $V\colon \R^4 \times M \to TM$ be the infinitesimal $C^r$ deformation defined as 
$$
V(B,\cdot):=B^1 V_1(\cdot) + B^2 V_2(\cdot),\quad \forall\, B=(B^1,B^2)\in \R^2 \times \R^2.
$$
In particular, $V$ satisfies $\mathrm{supp}(V)\subset B(x_0,10\sigma)$,  for any $x\in \cWc_{f_1}(x_0,\widetilde{K}\sigma)$, 
$V$ is adapted to $(\gamma_1^x, \sigma,\widetilde C,\widetilde{R})$ and $(\gamma_2^x, \sigma,\widetilde C,\widetilde{R})$, and for any corner $z\neq \alpha_1^x,\alpha_2^x$ of $\gamma_1^x$, $\gamma_2^x$,  
\begin{equation*} 
D_{B}(\pi_cV(B, z)) = 0, %\quad \forall z \in \cWc_f(x_{i,j,1}, \cg \Lambda_f^{4c}\sigma\clb),
\end{equation*}
while for $E_1:=\R^2 \times \{0_{\R^2}\}$, $E_2:=\{0_{\R^2}\} \times \R^2$, we have
\begin{equation}\label{champ de vecteur def} 
\Big|\det\big(E_\star \ni B \mapsto D_{B}\big(\pi_cV\big(B, \alpha_\star^x\big)\big)\big)\Big| > \tilde \kappa, \quad \star=1,2.
\end{equation}

	For some small $\delta_1>0$, let us consider the $C^r$ deformation $\hat f \colon B(0_{\R^{4}},\delta_1) \times M \to M$ at $f_1$ with $4$-parameters generated by the infinitesimal $C^r$ deformation $V$.  As before, for any $b \in B(0_{\R^{4}},\delta_1)$, we set $f_b:=\hat f(b,\cdot)$. 
	By \eqref{cone cone}, if $\delta_1$ and $\sigma$ are sufficiently small, then for all $b \in B(0_{\R^4},\delta_1)$, and for all $x \in \Sigma_b(\sigma):=\Gamma_{f_b}^1\cap\cWc_{f_1}(x_0,\widetilde{K}\sigma)$, it holds 
	%such that
	\begin{align}\label{var clases duex}
	\Pi_x^c\cac_{f_b}(x,10\sigma)\subset \mathscr{C}_1.%{f_0}(x,\theta,10\sigma).
	\end{align}

Let $T=T(\hat f)$ be as in \eqref{def T}. We denote by $\hat \gamma_1^x$, $\hat \gamma_2^x$ the respective lifts of $\gamma_1^x$ and $\gamma_2^x$ for $T$ according to Definition \ref{defii liftt}.  By \eqref{champ de vecteur def}, thanks to Proposition \ref{determinant for smooth deformations}, and arguing as in Corollary \ref{coro un peu chiant}, we obtain:
	\begin{lemma}\label{lemma end lemma}
		The map 
		\begin{equation*}
		\Phi\colon\left\{
		\begin{array}{rcl}
		B(0_{\R^{4}},\delta_1) \times \cWc_{f_1}(x_0,\widetilde{K}\sigma)  & \to & \R^{4}=\R^2 \times \R^2\\
		(b,x) & \mapsto & \Big(\Pi_x^c H_{T,\hat \gamma_1^x}(b,x),\Pi_x^c  H_{T,\hat \gamma_2^x}(b,x)\Big)
		\end{array}
		\right.
		\end{equation*} 
		satisfies 
		\begin{equation}\label{closeness deriv}
		\big|D_b|_{b=0} \Phi(\cdot,x)-D_b|_{b=0} \Phi(\cdot,y)\big| \leq \rho(\sigma),\qquad \forall\, x,y \in \cWc_{f_1}(x_0,\widetilde{K}\sigma),
		\end{equation}
		for some function $\rho\colon \R_+ \to \R_+$ such that $\lim\limits_{\sigma \to 0}\rho(\sigma)=0$, and  there exists  $\kappa>0$ such that  for any $x \in\cWc_{f_1}(x_0,\widetilde{K}\sigma)$, it holds 
		\begin{equation}\label{determin fin fin bisbi}
		|\det D_b|_{b=0}\big(\Phi(\cdot,x)\big)| > \kappa.	
		\end{equation} 
	\end{lemma}

Indeed, for $\star=1,2$, since the map $\cWc_{f_1}(x_0,\widetilde{K}\sigma) \ni x \mapsto \gamma^x_\star$ is continuous, it follows from Lemma \ref{lemme Davi} and Corollary \ref{lemma dependence family} that the partial derivatives of the holonomies $H_{T,\hat \gamma_\star^x}$ with respect to $b$ are uniformly close, for all $x \in\cWc_{f_1}(x_0,\widetilde{K}\sigma)$. Hence, by the definition of $\Phi$, and by Proposition \ref{determinant for smooth deformations}, the maps $\{\Phi(\cdot,x)\}_{x \in \cWc_{f_1}(x_0,\widetilde{K}\sigma)}$ are uniform submersions, which gives \eqref{closeness deriv} and \eqref{determin fin fin bisbi}. \\

	By \eqref{good cone plus moins},  for each $x \in\cWc_{f_1}(x_0,\widetilde{K}\sigma)$, we have 
	$$
	\Phi(0,x)\in \big(\mathscr{C}^+ \times \mathscr{C}^-\big)\cup \big(\mathscr{C}^- \times \mathscr{C}^+\big).
	$$ 
	Let us denote by $S^+$, resp. $S^-$ the set of all points $x \in \cWc_{f_1}(x_0,\widetilde{K}\sigma)$ such that $\Phi(0,x)\in \mathscr{C}^+ \times \mathscr{C}^-$, resp. $\Phi(0,x)\in \mathscr{C}^- \times \mathscr{C}^+$, so that $S^+ \cup S^-=\cWc_{f_1}(x_0,\widetilde{K}\sigma)$. 
	By \eqref{closeness deriv}-\eqref{determin fin fin bisbi}, %the maps $\Phi(\cdot,x)$ %for $x \in \Sigma(\sigma)$ 
	%are uniform submersions, 
	%thus 
	there exists a perturbation parameter %$\Phi_x\big(B(0_{\R^{2}},\delta_0')\big)$ contains an open neighbourhood of $0_{\R^2}$. Thus, there exists 
	$b\in B(0_{\R^{4}},\delta_1)$ such that 
	\begin{align*}
	\Pi_x^c H_{T,\hat \gamma_1^x}(b,x)&\in \mathscr{C}^+_*=\mathscr{C}^+\setminus \{0_{\R^2}\},\quad \text{ for all } x \in S^+,\\
	\Pi_x^c H_{T,\hat \gamma_2^x}(b,x)&\in \mathscr{C}^-_*=\mathscr{C}^-\setminus \{0_{\R^2}\},\quad \text{ for all } x \in S^-.
	\end{align*} 
	As $S^+ \cup S^-=\cWc_{f_1}(x_0,\widetilde{K}\sigma)$, we deduce that for each $x \in  \cWc_{f_1}(x_0,\widetilde{K}\sigma)$, 
	$$
	\text{either}\quad \Pi_x^c H_{T,\hat \gamma_1^x}(b,x) \notin \mathscr{C}_1,\quad \text{or}\quad \Pi_x^c H_{T,\hat \gamma_2^x}(b,x) \notin \mathscr{C}_1.
	$$
	By \eqref{var clases duex}, we deduce that $\Sigma_b(\sigma)=\emptyset$, i.e., $\Gamma_{f_b}^1=\emptyset$. Therefore,   %implies that $\R^2\setminus \Phi_x\big(B(0_{\R^{2}},\delta_0')\big)\supset \Delta_x$ for some straight line $\Delta_x$ through the origin $0_{\R^2}$. But , a contradiction. 
	by Theorem \ref{theorem structure}, the accessibility class $\ac_{f_b}(x)$ of each point $x \in \cWc_{f_1}(x_0,\widetilde{K}\sigma)$ is open.  
	Moreover, if $\delta_1$ is sufficiently small, then by construction, the diffeomorphism $g:=f_b$ satisfies   
	$$
	d_{C^r}(f,g) \leq d_{C^r}(f,f_1)+d_{C^r}(f_1,f_b)<\delta,
	$$ 
	which concludes the proof of Proposition \ref{prop open global}. 
\end{proof}

\begin{proof}[Proof of Proposition \ref{main prop}]
	Fix $\delta>0$ arbitrarily small. 
	Let  $\tilde{\sigma}_2=\tilde{\sigma}_2(f)>0$, $\tilde K_2=\tilde K_2(f)\in (0,1)$ and $\tilde R_2=\tilde R_2(f)>0$ be the $C^2$-uniform constants given by Proposition \ref{prop open global}. By Proposition \ref{prop acc mod den dis}, there exist $C^1$-uniform constants $\overline{N} = \overline{N}(f,\tilde R_2)>0$, $\overline{\rho} = \overline{\rho}(f, \tilde R_2)\in(0,\tilde R_2^{-1})$ and $\overline{\sigma} = \overline{\sigma}(f,\tilde R_2)>0$ such that for all diffeomorphism $g$ sufficiently $C^1$-close to $f$, there exists   a $(\frac{1}{40},  4)$-spanning $c$-family $\cD_g$ for $g$ with at most $\overline{N}$ elements such that 
	\begin{enumerate}
		\item $\overline{\rho}<\underline{r}(\cD_g)\leq \overline{r}(\cD_g) < \tilde R_2^{-1}$;
		\item $\cD_g$ is $\overline{\sigma}$-sparse;
		\item $R_\pm(g, (\cD_g, \overline{\sigma})) > \tilde R_2$.
	\end{enumerate}
and such that the map $g \mapsto \mathcal{D}_g$ is continuous. 
Let $\sigma \in (0,\frac{1}{10}\min(\tilde{\sigma}_2,\overline{\sigma}))$. By compactness, we can take a finite collection of points $x_1,\dots,x_m\in M$ such that  
$$
\frac{1}{20}\mathcal{D}_f\subset U:=\bigcup_{i=1}^m B(x_i,\tilde K_2 \sigma)  \subset (\cD_f, \overline{\sigma}).
$$
Note that $x_i\in M$ satisfies $R_{\pm}(f,B(x_i,10\sigma))>\tilde R_2$, for each $i\in \{1,\dots,m\}$. Therefore, we can apply Proposition \ref{prop open global} inductively to get  a partially hyperbolic diffeomorphism $g \in \mathscr{F}$ such that $d_{C^r}(f,g)< \delta$ and such that $\ac_g(x_i)\supset B(x_i,\tilde K_2\sigma)$, for all $i\in \{1,\dots,m\}$. By connectedness of the disks in  $\mathcal{D}_f$, each center disk in the family $\frac{1}{20}\mathcal{D}_f$ is contained in a single accessibility class for $g$. Moreover, if $\delta$ is sufficiently small, and by continuity of the map $h \mapsto \mathcal{D}_h$, each center disk in the family $\frac{1}{40}\mathcal{D}_g$ is contained in a single accessibility class for $g$. As $\cD_g$ is a $(\frac{1}{40},  4)$-spanning $c$-family   for $g$, we deduce that $g$ is accessible, as wanted. 
\end{proof}

\small
	\vspace{0.4cm}
\begin{flushleft}
	\textsc{Martin Leguil}\\
	CNRS-Laboratoire de Mathématiques d'Orsay, UMR 8628, Université Paris-Saclay, Orsay Cedex 91405, France \& Laboratoire Ami\'enois de Math\'ematique Fondamentale et Appliqu\'ee (LAMFA), UMR 7352, 
	Université de Picardie Jules Verne, 33 rue Saint Leu, 80039 Amiens, France\\
	email: \texttt{martin.leguil@u-picardie.fr}
	
	\vspace{0.4cm}
	
	\textsc{Luis Pedro Pi\~neyr\'ua}\\
	CMAT, Facultad de Ciencias\\
	Universidad de la Rep\'ublica, Montevideo, Uruguay\\
	email: \texttt{lpineyrua@cmat.edu.uy} 
\end{flushleft}

\end{document}